\documentclass[11pt, a4paper]{article}
\usepackage[utf8]{inputenc}
\usepackage[usenames,dvipsnames]{xcolor}
\usepackage{amsmath}
\usepackage{amsthm}
\usepackage{amsfonts}
\usepackage{amssymb}
\usepackage{array}
\usepackage{graphicx} 
\usepackage{color}
\usepackage[english]{babel}
\usepackage{mathrsfs}
\usepackage{graphicx}
\usepackage{dsfont}
\definecolor{orangebis}{rgb}{0.99,0.25,0.00}
\definecolor{greenbis}{rgb}{0.10,0.85,0.10}
\definecolor{bluebis}{rgb}{0.10,0.30,0.99}
\usepackage[final]{hyperref}   
\hypersetup{
    linktoc=page,
    linkcolor=red,          
    citecolor=blue,        
    filecolor=blue,      
    urlcolor=cyan,
   colorlinks=true      }     

\usepackage{multicol}

\usepackage{lipsum}
\usepackage{geometry}
\author{Hugo Vanneuville\thanks{Univ. Lyon 1, Institut Camille Jordan, 69100 Villeurbanne, France, supported by the ERC grant Liko No 676999}}
\title{Annealed scaling relations for Voronoi percolation}
\date{}

\theoremstyle{plain}
\newtheorem{thm}{Theorem}[section]
\newtheorem{prop}[thm]{Proposition}
\newtheorem{lem}[thm]{Lemma}

\newtheorem{cor}[thm]{Corollary}

\newtheorem{claim}[thm]{Claim}

\theoremstyle{definition}
\newtheorem{defi}[thm]{Definition}

\theoremstyle{remark}
\newtheorem{rem}[thm]{Remark}

\newcommand{\margin}[1]{\textcolor{magenta}{*}\marginpar[\textcolor{magenta} {  \raggedleft  \footnotesize  #1 }  ]{ \textcolor{magenta} { \raggedright  \footnotesize  #1 }  }}

\renewcommand{\margin}[1]{}

\newcommand{\N}{\mathbb{N}}
\newcommand{\R}{\mathbb{R}}

\newcommand{\Z}{\mathbb{Z}}

\newcommand{\half}{\mathbb{H}}
\newcommand{\diam}{\text{\textup{diam}}}
\newcommand{\Pro}{\mathbb{P}}
\newcommand{\E}{\mathbb{E}}
\newcommand{\T}{\mathbb{T}}

\newcommand{\Var}{\mathbb{V}\text{\textup{ar}}}

\newcommand{\Prob}{\text{\textup{\textbf{P}}}}
\newcommand{\Ex}{\text{\textup{\textbf{E}}}}
\newcommand{\Piv}{\text{\textup{\textbf{Piv}}}}
\newcommand{\arm}{\text{\textup{\textbf{A}}}}
\newcommand{\setS}{\text{\textup{\textbf{S}}}}
\newcommand{\un}{\mathds{1}}
\newcommand{\petito}[1]{o\mathopen{}\left(#1\right)}
\newcommand{\grandO}[1]{O\mathopen{}\left(#1\right)}

\newcommand{\cross}{\textup{\text{Cross}}}
\newcommand{\dense}{\textup{\text{Dense}}}
\newcommand{\qbc}{\textup{\text{QBC}}}
\newcommand{\gi}{\textup{\text{GI}}}
\newcommand{\gp}{\textup{\text{GP}}}

\newcommand{\qac}{\textup{\text{QAC}}}

\def\T{\mathbb{T}}

\newcommand{\cond}{\, \Big| \,}
\renewcommand{\textbf}[1]{\begingroup\bfseries\mathversion{bold}#1\endgroup}
\def\ext{\mathrm{ext}}

\def\diam{\mathrm{diam}}

\def\dist{\mathrm{dist}}

\def\E{\mathbb{E}} 

\def\<#1{\langle #1\rangle}

\def\bi{\begin{itemize}}  
\def\ei{\end{itemize}}
\def\bnum{\begin{enumerate}} 
\def\enum{\end{enumerate}}
\def\ni{\noindent}
\def\bf{\bfseries}

\numberwithin{equation}{section}
\begin{document}

\maketitle

\abstract{We prove annealed scaling relations for planar Voronoi percolation. To our knowledge, this is the first result of this kind for a continuum percolation model. We are mostly inspired by the proof of scaling relations for Bernoulli percolation by Kesten~\cite{kesten1987scaling}. Along the way, we show an annealed quasi-multiplicativity property by relying on the quenched box-crossing property proved by Ahlberg, Griffiths, Morris and Tassion~\cite{ahlberg2015quenched}. Intermediate results also include the study of quenched and annealed notions of pivotal events and the extension of the quenched box-crossing property of~\cite{ahlberg2015quenched} to the near-critical regime.}

\tableofcontents

\section{The model and the main result}

\subsection{Percolation on planar lattices}\label{ss.Z2}

Consider bond percolation on the square lattice $\Z^2$ or site percolation on the planar triangular lattice $\T$. In these models, each edge or site is open (respectively closed) with probability $p$ (respectively $1-p$) independently of the others. Let $\theta(p)$ be the probability that there is an infinite open path starting from $0$. It is well known (see for instance~\cite{grimmett1999percolation,bollobas2006percolation}) that there exists a critical point $p_c \in (0,1)$ such that
\begin{enumerate}
\item[i)] $\forall p \in [0,p_c)$, $\theta(p)=0 \,$,
\item[ii)] $\forall p \in (p_c,1]$, $\theta(p) > 0 \,$.
\end{enumerate}
It is a theorem by Kesten \cite{kesten1980critical} that $p_c=1/2$ for these two models. Moreover, it has been proved by Harris~\cite{harris1960lower} that $\theta(1/2)=0$. Let us say a little more about the behaviour of this model at and near the critical point: Thanks to the Russo-Seymour-Welsh (RSW) theory and the study of interfaces between open and dual paths, one can obtain the so-called quasi-multiplicativity property of arm events and derive estimates on ``pivotal events'', see~\cite{kesten1987scaling,werner2007lectures,nolin2008near,schramm2010quantitative,manolescu2012universality}. Both are important tools in order to
\bi
\item[(a)] obtain the \textbf{scaling relations} proved by Kesten (see~\cite{kesten1987scaling,werner2007lectures,nolin2008near}),
\item[(b)] study dynamical percolation and \textbf{noise sensitivity} of percolation, see~\cite{benjamini1999noise,schramm2010quantitative,garban2010fourier,broman2013exclusion,garban2014noise,hammond2015local,garban2016exceptional}
\item[(c)] study the \textbf{scaling limits} of percolation, near-critical percolation, and dynamical percolation, see~\cite{schramm2011scaling, garban2013pivotal, garban2013scaling}.
\ei



The goal of this paper is twofold: (1) We prove the quasi-multiplicativity property (and some estimates on ``pivotal events'') for planar Voronoi percolation, which is a continuum percolation model. (2) We prove two scaling relations for Voronoi percolation.
\medskip

Before recalling the definition of Voronoi percolation, let us note that the authors of \cite{ahlberg2015quenched} and~\cite{ahlberg2017noise} have proved noise sensitivity results for Voronoi pecolation by following ideas from~\cite{benjamini1999noise,schramm2010quantitative,ahlberg2014noise}. We also see the present paper as a first step in order to be able to apply the more quantitative noise sensitivity methods from~\cite{garban2010fourier}. Indeed, to apply methods from~\cite{garban2010fourier}, one needs to have good controls on the probabilities of arm events and pivotal events.

\subsection{Planar Voronoi percolation: box-crossing estimates and the quasi-multiplicativity property}\label{ss.voronoi}

In this subsection, we introduce the model of Voronoi percolation. We refer to Section~8.3 of~\cite{bollobas2006percolation} for more details.

\paragraph{A. Voronoi percolation.} Let us define planar Voronoi percolation. To this purpose, let us consider a homogeneous Poisson process of intensity $1$ in $\R^2$, that we denote by $\eta$. For each point $x \in \eta$, the \textbf{Voronoi cell} of $x$, denoted by $C(x)$, is the set of all points $u \in \R^2$ such that for all $x' \in \eta$, $||u-x||_2 \leq ||u-x'||_2$. We say that $x$ is the center of $C(x)$. Also, we say that two points of $\eta$ are adjacent if their cells intersect each other. It is not difficult to see that a.s. all the cells are bounded convex polygons. Now, let us consider some parameter $p \in [0,1]$ and, given $\eta$, let us declare each $x \in \eta$ open (we will choose to say that ``we color the point \textbf{black}'') with probability $p$ and closed (\textbf{white}) with probability $1-p$, independently of the other points of $\eta$. Let $\omega \in \lbrace -1,1 \rbrace^\eta$ be the colored configuration we thus obtain (where $1$ means black and $-1$ means white).\footnote{There is no problem of measurability here: $\omega$ can be seen for instance as a point process with values in $\R^2 \times \lbrace -1,1 \rbrace$ whose intensity is $\text{Leb}_{\R^2} \otimes \left( p \delta_1 + (1-p) \delta_{-1} \right)$, where $\text{Leb}_{\R^2} $ is the Lebesgue measure in the plane.}


We will always write $\eta$ for the non-colored point process and $\omega$ for the colored point process. The distribution of $\omega$ will be denoted by $\Pro_p$.

Given the configuration $\omega$, we define a coloring of the plane as follows: each point $u \in \R^2$ is colored black if it is contained in the cell of a black point $x \in \eta$ and is colored white if it is contained in the cell of a white point $x \in \eta$ (note that the points on the boundary of the cells may be colored both black and white but this is not important in this paper). Moreover, we call black (respectively white) path a continuous path included in the black (respectively white) region of the plane.

\paragraph{B. The critical point.}
Let $\lbrace 0 \leftrightarrow +\infty \rbrace$ be the event that there is a black path from the origin to infinity and let $\theta^{an}(p) = \Pro_p \left[ 0 \leftrightarrow \infty \right]$ denote the (annealed) percolation function. The critical point $p_c$ is defined as follows:
\[
p_c := \inf \left\lbrace p \in [0,1] \, : \, \theta^{an}(p) > 0 \right\rbrace \, .
\]
It has been proved by Zvavitch~\cite{zvavitch1996critical} that $\theta^{an}(1/2) = 0$ - hence $p_c \geq 1/2$ - and it is a result of Bollob\'{a}s and Riordan~\cite{bollobas2006critical} that $p_c = 1/2$. A crucial fact for this result is the so-called self-duality property of the model: a.s., a rectangle is crossed lengthwise by a black path if and only if it is not crossed widthwise by a white path (see for instance Lemma~$12$ in Chapter~$8$ of~\cite{bollobas2006percolation}). An important step to show the result of Bollob\'{a}s and Riordan is the proof of a weak box-crossing property. A stronger version has more recently been proved by Tassion~\cite{tassion2014crossing} and has led to the derivation of \textbf{quenched crossing estimates} in~\cite{ahlberg2015quenched} that will be crucial in the present paper. Before stating these box-crossing results, let us note that an alternative proof of $p_c=1/2$ can be found in the recent paper~\cite{duminil2017exponential}. In the said article, Duminil-Copin, Raoufi and Tassion prove the exponential decay of connection probabilities for subcritical Voronoi percolation in any dimension. An alternative proof of $p_c=1/2$ can also be found in \cite{ahlberg2017noise} Ahlberg and Baldasso study the near-critical window of Voronoi percolation.

\paragraph{C. Box-crossing properties.}

We first need two definitions/notations:

\begin{defi}\label{d.probeta}
Given $\eta$, we write $\Prob^\eta_p$ for the conditional distribution of $\omega$ given $\eta$ (which is simply the product law $\left( p\delta_1+(1-p)\delta_{-1}\right)^{\otimes \eta}$). More generally, if $E$ is a countable set, we write $\Prob_p^E=\left( p\delta_1+(1-p)\delta_{-1}\right)^{\otimes E}$.
\end{defi}

\begin{defi}\label{d.cross}
For any $\rho_1,\rho_2 > 0$, $\text{\textup{Cross}}(\rho_1,\rho_2)$ (respectively $\text{\textup{Cross}}^*(\rho_1,\rho_2)$) denotes the event that there is a black (respectively a white) path included in $[-\rho_1,\rho_1] \times [-\rho_2,\rho_2]$ that connects the left side of this rectangle to its right side.
\end{defi}
Now, we can state the annealed box-crossing property obtained by Tassion and the quenched box-crossing property obtained by Ahlberg, Griffiths, Morris and Tassion. An important step in the present paper is the extension of these results to the ``near-critical regime'', see Subsection~\ref{ss.pneq1/2_1}.

\begin{thm}[Theorem~$3$ of~\cite{tassion2014crossing}]\label{t.Tassion}
Let $\rho > 0$. There exists a constant $c=c(\rho) \in (0,1)$ such that, for every $R \in (0,+\infty)$,
\[
c \leq \Pro_{1/2} \left[ \text{\textup{Cross}}(\rho R,R) \right] \leq 1-c \, .
\]
\end{thm}


\begin{thm}[Theorem~$1.4$ of~\cite{ahlberg2015quenched} and the paragraph below it. See also our Appendix~\ref{a.tAGMT} where we recall the main ingredients of the proof of this theorem.\footnote{Actually, in Appendix~\ref{a.tAGMT} we will modify a little the proof of~\cite{ahlberg2015quenched} so that this proof will be easier to adapt to the near-critical phase.}]\label{t.AGMT}
Let $\rho > 0$.
\bi
\item[i)] There exists an absolute constant $\epsilon > 0$ and a constant $C = C(\rho) < +\infty$ such that, for every $R \in (0,+\infty)$,
\[
\Var \left( \Prob^\eta_{1/2} \left[ \text{\textup{Cross}}(\rho R,R) \right] \right) \leq C \, R^{-\epsilon} \, .
\]
This implies the following estimate,
\item[ii)] For every $\gamma \in (0,+\infty)$, there exists a positive constant $c=c(\rho,\gamma) \in (0,1)$ such that, for every $R \in (0,+\infty)$:
\[
\Pro \left[ \, c \leq \Prob^\eta_{1/2} \left[ \text{\textup{Cross}}(\rho R,R) \right] \leq 1-c \, \right] \geq 1 - R^{-\gamma} \, .
\]
\ei
\end{thm}

\paragraph{D. Arm events.} Once we have such crossing properties, a natural goal
is to study arm events.
Let us first define these events:

\begin{defi}[$j$-arm events]
Let $j \in \N^*:=\{1,2,\cdots\}$ and $0 \leq r \leq R$. The $j$-arm event between scales $r$ and $R$ is the event that there exist $j$ paths of alternating colors in the annulus $[-R,R]^2 \setminus [-r,r]^2$ from $\partial [-r,r]^2$ to $\partial [-R,R]^2$ (if $j$ is odd, we ask that there are: (a) $j-1$ paths of alternating colors, and: (b) one additional black path such that there is no Voronoi cell intersected by both this additional path and one of the $j-1$ other paths). Let $\arm_j(r,R)$ denote this event. We write the \textbf{annealed} probability of this event as follows:
\[
\alpha^{an}_{j,p}(r,R) = \Pro_p \left[ \arm_j(r,R) \right] \, .
\]
\end{defi}
We write $\alpha^{an}_{j,p}(R) = \alpha^{an}_{j,p}(1,R)$ for any $j \in \N^*$. If $r > R$, we choose that $\alpha^{an}_{j,p}(r,R) = 1$. Also, we will often use the following simplified notation:
\[
\alpha^{an}_j(r,R) = \alpha^{an}_{j,1/2}(r,R) \, .
\]
An important property of the quantities $\alpha^{an}_{j,1/2}(r,R)$ is that they decay polynomially fast: There exists a constant $C=C(j) \in [1, +\infty)$ such that, for every $1 \leq r \leq R$:
\begin{equation}\label{e.poly}
\frac{1}{C} \, \left( \frac{r}{R} \right)^{C} \leq \alpha^{an}_{j,1/2}(r,R) \leq C \, \left( \frac{r}{R} \right)^{1/C} \, .
\end{equation}
The right-hand-inequality is proved in~\cite{tassion2014crossing} (Item~$2$ of Theorem~$3$) and we prove the left-hand-inequality in Subsection~\ref{ss.warm}. In the present paper, we prove the \textbf{annealed quasi-multiplicativity property} for the quantities $\alpha^{an}_{j,1/2}(r,R)$. This is the most delicate part of the paper. Even if this is an annealed result, the quenched box-crossing property Theorem~\ref{t.AGMT} will be a crucial ingredient of the proof.

\begin{prop}[Annealed quasi-multiplicativity property]\label{p.quasi}
Let $j \in \N^*$. There exists a constant $C=C(j) \in [1,+\infty)$ such that, for all $1 \leq r_1 \leq r_2 \leq r_3$,
\begin{equation}
\frac{1}{C} \, \alpha^{an}_{j,1/2}(r_1,r_3) \leq \alpha^{an}_{j,1/2}(r_1,r_2) \, \alpha^{an}_{j,1/2}(r_2,r_3) \leq C \, \alpha^{an}_{j,1/2}(r_1,r_3) \, .
\end{equation}
\end{prop}

\begin{rem}\label{r.quasi_j=1}
In Proposition~\ref{p.quasi}, the case $j=1$ is easier. More precisely, the right-hand-inequality in this case is a direct consequence of the box-crossing property Theorem~\ref{t.Tassion} and of the (annealed) FKG-Harris inequality (stated in Subsection~\ref{ss.correlation_in}). Moreover, the proof of the left-hand-inequality in the case $j=1$ is written in Subsection~\ref{ss.warm}.
\end{rem}

\begin{rem}\label{r.1leqr_1}
Our choice to impose that the radii $r_i$ are at least $1$ is arbitrary. In fact, we could have chosen any $a > 0$ and rather asked that $a \leq r_1 \leq r_2 \leq r_3$. We would have obtained the same result with some constant $C=C(j,a)$.
\end{rem}

The main difficulty in the study of arm events (compared to crossing events for instance) is that they are \textbf{degenerate events}. As a result, it could a priori be the case that if $r \ll R$ and if we condition on $\arm_j(r,R)$, then with high probability the point process $\eta$ is very degenerate at scale $r$, see Figure~\ref{f.difficulties}. We refer to Subsection~\ref{ss.quasi} for some key properties and some tools developed to overcome this difficulty (see in particular Propositions~\ref{p.fandalpha} and~\ref{p.techniquegeneral}).
\medskip

\begin{figure}[!h]
\begin{center}
\includegraphics[scale=0.23]{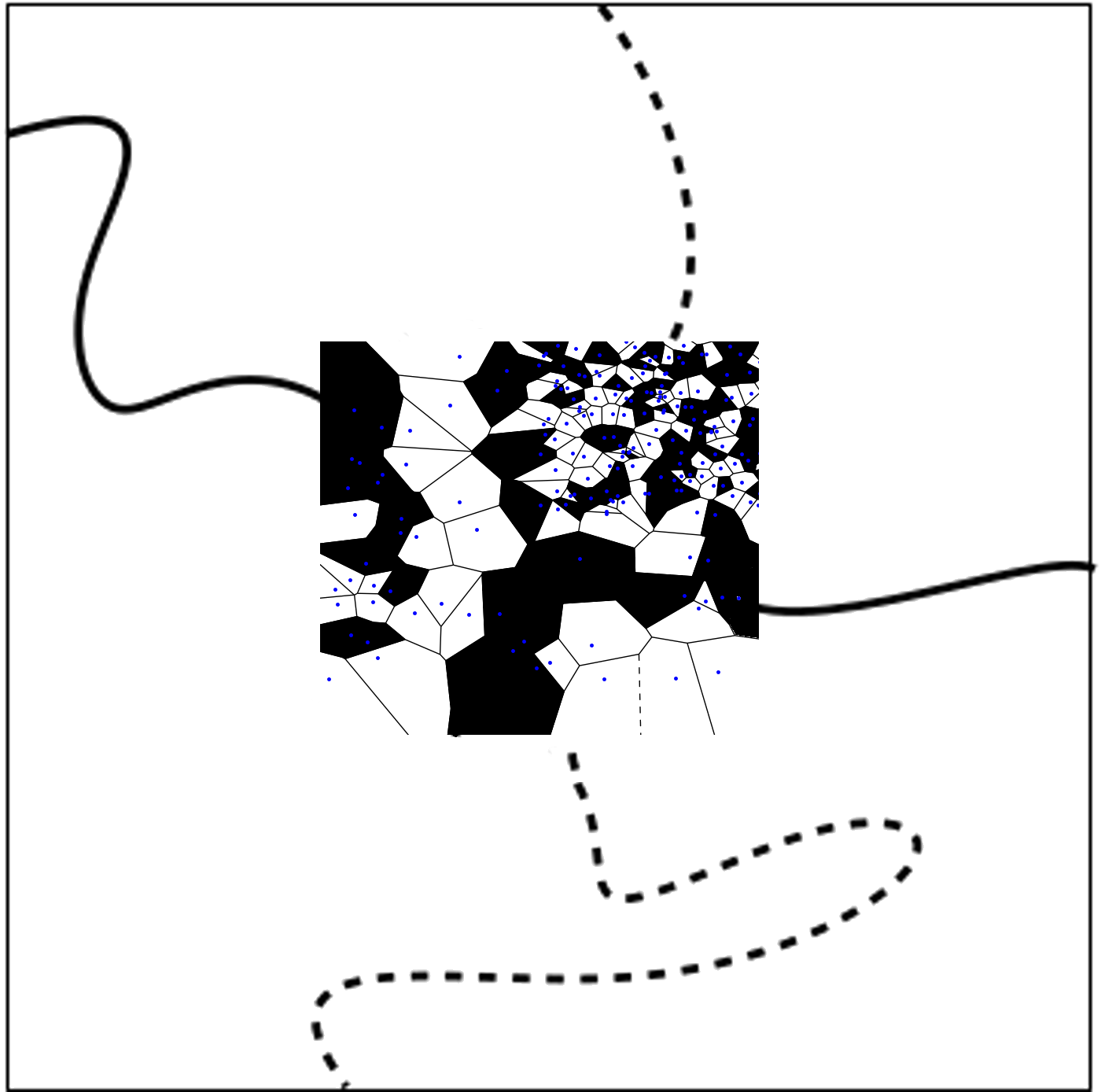}
\end{center}
\caption{In this paper, we deal with degenerate events: the $j$-arm events $\arm_j(r,R)$. It is not clear that, when conditionning on $\arm_j(r,R)$ with $r \ll R$, the random environment at scale $r$ is not typically degenerate. An example of a degenerate environment is illustrated in the figure: the Voronoi tiling is extremly dense in some regions and not dense at all in some other regions. In the region where the Voronoi tiling is extremly dense, it might be very costly to extend the arms to other scales. The biggest issue comes from the regions where the Voronoi tiling is not dense at all. In these regions, there are a lot of spatial dependences.}\label{f.difficulties}
\end{figure}

Let us now state the main result of our paper.

\subsection{The main result: annealed scaling relations for Voronoi percolation}

It is believed that, for a wide class of percolation models, the evolution as $p$ goes to $p_c = 1/2$ of some key quantities is determined by some \textbf{critical exponents}. Such quantities are for instance the percolation function, the correlation length and the probabilities of arm events. The famous \textbf{scaling relations} proved by Kesten~\cite{kesten1987scaling} are simple relations between these exponents. More precisely, Kesten proved that, for bond percolation on the square lattice (and site percolation on the triangular lattice): i) if we assume that these key quantities are indeed described by exponents, then these exponents satisfy the relations predicted by theoritical physicists in the 70's (we refer to~\cite{kesten1987scaling} for references concerning these predictions), and ii) even if we do not assume that these exponents exist, the corresponding relations between the percolation function, the correlation length etc hold. There is only one planar percolation model for which it is known that such exponents exist: site percolation on the triangular lattice, that is the only model for which conformal invariance has been proved, see the proof by Smirnov~\cite{smirnov2001criticalp}. These exponents have even been computed thanks to the theory of SLE's (Schramm-Loewner Evolution), see~\cite{Smirnov2001critical,lawler2002onearm,werner2007lectures}. 

Let us go back to Voronoi percolation. For this model, the existence of these exponents is not known (conformal invariance is not proved for this model even if a first important step has been made in this direction by Benjamini and Schramm~\cite{benjamini1998conformal}). To state our main result, let us define the annealed correlation length.

\begin{defi}\label{d.corrlength}
Let $\epsilon_0 \in (0,1)$ be sufficiently small\footnote{More precisely, we need that $1-2\epsilon_0 > \Pro_{1/2} \left[ \cross(2R,R) \right]$ for every $R \geq 1$ - which is possible thanks to Theorem~\ref{t.Tassion} - and that $\epsilon_0$ is sufficiently small so that a Peierls argument works - see the proof of Lemma~\ref{l.scaling2} for more about this second condition.} and let $p \in (1/2,1]$. The \textbf{annealed correlation length} at parameter $p$, denoted by $L^{an}(p)=L^{an,\epsilon_0}(p)$, is defined as follows:
\[
L^{an}(p) = \inf \left\lbrace  R \geq 1 \, : \, \Pro_p \left[ \text{\textup{Cross}}(2R,R) \right] \geq 1-\epsilon_0 \right\rbrace \, .
\]
\end{defi}

An important property is that, for every $p > 1/2$, $L^{an}(p) < +\infty$, see Lemma~\ref{l.exp_decay}. The idea behind the definition of the correlation length is that this is the largest scale such that the percolation configuration at this scale ``looks critical". In particular, we prove in Subsection~\ref{ss.pneq1/2_1} that the annealed and quenched box-crossing properties Theorems~\ref{t.Tassion} and~\ref{t.AGMT} are also true for $p>1/2$ as soon as we work at scales smaller than the correlation length (i.e. as soon as we work in the ``near-critical phase''). Moreover, we prove the following result in Section~\ref{s.kesten}.

\begin{prop}\label{p.alpha_jandalpha_1/2}
Let\footnote{The number $3/4$ does not have to be taken seriously, we consider $p \in (1/2,3/4]$ only to avoid problems with $p$ close to $1$.} $p \in (1/2,3/4]$ and let $\epsilon_0$ be the parameter of Definition~\ref{d.corrlength}. Also, let $j \in \N^*$. There exists a constant $C=C(\epsilon_0,j) \in [1,+\infty)$ such that, for every $1 \leq r \leq R \leq L^{an}(p)$,
\[
\frac{1}{C} \alpha^{an}_{j,1/2}(r,R) \leq \alpha^{an}_{j,p}(r,R) \leq C \alpha^{an}_{j,1/2}(r,R) \, .
\]
\end{prop}

In the present paper, we focus on the following exponents: It is believed that there exist $\nu \in (0,+\infty)$, $\beta \in (0,+\infty)$ and $\zeta_j \in (0,+\infty)$ such that:
\begin{align*}
&\forall p \in (1/2,1), \, \theta^{an}(p) = (p-1/2)^{\beta+\petito{1}} \, ,\\
&\forall p \in (1/2,1), \, L^{an}(p) = (p-1/2)^{-\nu+\petito{1}} \, ,\\
&\forall j \in \N^* \text{ and } \forall 1 \leq r \leq R, \, \alpha^{an}_{j,1/2}(r,R) =\left(\frac{r}{R}\right)^{\zeta_j+\petito{1}} \, ,
\end{align*}
where $\petito{1}$ goes to $0$ as $p$ goes to $1/2$ (respectively as $r/R$ goes to $0$). Moreover, it is believed that the following relations hold between these exponents:
\[
\beta = \nu \, \zeta_1 \, ; \: \nu = \frac{1}{2-\zeta_4} \, .
\]

The main results of the present paper is that, if these exponents exist, then these two scaling relations hold. As in~\cite{kesten1987scaling}, we also prove that, even if we do not assume that the exponents exist, then the corresponding relations between the percolation function, the correlation length and the probabilities of arm events hold. More precisely, we obtain the following:
\begin{thm}\label{t.scaling} Let $p \in (1/2,3/4]$ and let $\epsilon_0$ be the parameter of Definition~\ref{d.corrlength}. There exists a constant $C=C(\epsilon_0) \in [1,+\infty)$ such that
\begin{equation}\label{e.scaling1}
\frac{1}{C} \, \alpha^{an}_{1,1/2}(L^{an}(p)) \leq \theta^{an}(p) \leq C \, \alpha^{an}_{1,1/2}(L^{an}(p)) \, ,
\end{equation}
and
\begin{equation}\label{e.scaling2}
\frac{1}{C} \, \frac{1}{p-1/2} \leq L^{an}(p)^2 \, \alpha^{an}_{4,1/2}(L^{an}(p)) \leq C \, \frac{1}{p-1/2} \, .
\end{equation}
\end{thm}
Proposition~\ref{p.alpha_jandalpha_1/2} and Theorem~\ref{t.scaling} are proved in Section~\ref{s.kesten} by relying on all the other sections.\\

Let us note that the recent paper \cite{ahlberg2017noise} by Ahlberg and Baldasso also deals with near-critical Voronoi percolation. More precisely, the authors of \cite{ahlberg2017noise} use randomized algorithms in the spirit of \cite{duminil2017sharp,duminil2017exponential} and thinning procedures to prove that the near-critical window is of polynomial size. The present paper implies the following more precise result: the size of the near-critical window is of order $1/(R^2 \alpha^{an}_4(R))$. Our techniques are different from \cite{ahlberg2017noise} (the techniques of the present paper are much more geometrical). In particular, our use of randomized algorithms is different from \cite{ahlberg2017noise} (we use inequalities by Schramm and Steif \cite{schramm2010quantitative} in the spirit of \cite{ahlberg2015quenched} to prove quenched estimates and estimates on the $4$-arm event while Ahlberg and Baldasso use the OSSS inequality \cite{o2005every} to prove estimates on the derivative of crossing probabilities).

\begin{rem}
Note that~\eqref{e.scaling1} (together with the quasi-multiplicativity property and~\eqref{e.poly}) implies that for every $\epsilon_0' > \epsilon_0 > 0$ sufficiently small, there exist $c=c(\epsilon_0,\epsilon_0')>0$ such that, for every $p \in (1/2,3/4]$,
\[
cL^{an,\epsilon_0}(p) \leq L^{an,\epsilon_0'}(p) \leq L^{an,\epsilon_0}(p) \, .
\]
\end{rem}

In~\cite{kesten1987scaling}, Kesten also proves other scaling relations. We believe that, with the results of the present paper, analogues of these other scaling relations can also be proved, but we have restricted ourself to the two scaling relations~\eqref{e.scaling1} and~\eqref{e.scaling2}.

\subsection{Estimates on the $4$-arm events, $\theta^{an}(p)$, and $L^{an}(p)$}\label{ss.estimates_corr_perco}

In the present paper, we prove some estimates on arm events. In particular, we obtain the following estimates on the $4$-arm events in Subsections~\ref{ss.pivotals1} and~\ref{ss.alpha4}:
\begin{prop}\label{p.alpha4}
There exists an absolute constant $\epsilon > 0$ such that the following holds:
\bi 
\item[i)] For every $R \in [1,+\infty)$,
\[
\alpha^{an}_{4,1/2}(R) \leq \frac{1}{\epsilon} R^{-(1+\epsilon)} \, .
\]
\item[ii)] For every $1 \leq r \leq R$,
\[
\alpha^{an}_{4,1/2}(r,R) \geq \epsilon \, \left( \frac{r}{R} \right)^{2-\epsilon} \, .
\]
\ei
\end{prop}

If we apply the first part of Proposition~\ref{p.alpha4} to the scaling relation~\eqref{e.scaling2} of Theorem~\ref{t.scaling}, then we obtain that
\[
L^{an}(p) \geq \epsilon \, \left( p-1/2 \right)^{-(1+\epsilon)} \, ,
\]
for some $\epsilon > 0$. If we rather use the second part of Proposition~\ref{p.alpha4}, then we obtain that
\[
L^{an}(p) \leq C (p-1/2)^{-C} \, ,
\]
for some $C<+\infty$. As a result, if the exponent $\nu$ exists, then $\nu \in (1,+\infty)$ (which is exactly - as far as we know - what is known for Bernoulli percolation on $\Z^2$, see~\cite{kesten1987scaling}). By using the polynomial decay property~\eqref{e.poly} and the scaling relation~\eqref{e.scaling1} of Theorem~\ref{t.scaling}, we deduce from this that
\begin{equation}\label{e.estimate_for_theta_an}
\epsilon \left( p-1/2 \right)^C \leq \theta^{an}(p) \leq C \left( p-1/2 \right)^\epsilon \, ,
\end{equation}
for some $C<+\infty$ and $\epsilon > 0$. In~\cite{kesten1987strict}, Kesten and Zhang have proved the following for Bernoulli percolation on $\Z^2$:
\begin{equation}\label{e.est_theta}
\theta(p) \geq \epsilon \, \left( p-1/2 \right)^{1-\epsilon} \, .
\end{equation}
In the case of Bernoulli percolation on the triangular lattice, it is known (see~\cite{lawler2002onearm} and~\cite{Smirnov2001critical}) that
\[
L(p) = \left( p-1/2 \right)^{-4/3+\petito{1}}
\]
and
\[
\theta(p) = \left( p-1/2 \right)^{5/36+\petito{1}} \, ,
\]
where $\petito{1} \rightarrow 0$ as $p \searrow 1/2$. The estimate~\eqref{e.estimate_for_theta_an} is strengthened in the two following papers:
\bi
\item In~\cite{quant_voro}, we prove that $\theta^{an}(p) \geq \epsilon (p-1/2)^{1-\epsilon}$ (by relying a lot on the present paper and in particular on Appendix~\ref{a.quant}).
\item In the recent work~\cite{duminil2017exponential} Duminil-Copin, Raoufi and Tassion use the OSSS inequality to prove that, for Voronoi percolation in any dimension $d \geq 2$, there exists $c=c(d)>0$ such that, for any $p>p_c=p_c(d)$,
\[
\theta^{an}(p) \geq c \,(p-p_c) \, .
\]
\ei

\subsection{Quenched or annealed results?}

In the present paper, our main goal is to prove \textbf{annealed} properties. The most important ones are the annealed scaling relations (Theorem~\ref{t.scaling}) and the annealed quasi-multiplicativity property (Proposition~\ref{p.quasi}). However, the \textbf{quenched} property Theorem~\ref{t.AGMT} will be one of our main tools. The multiple passages from quenched to annealed properties will be rather technical, see in particular Section~\ref{s.quasi}. As a result, it seems at first sight that it would be easier to prove quenched properties. We indeed believe that one could use Theorem~\ref{t.AGMT} to prove a quenched quasi-multiplicativity property (with a less technical proof than the annealed quasi-multiplicativity property). However, proving scaling relations at the quenched level seems much more complicated than proving them at the annealed level since the classical methods (that we follow in the present paper) deeply rely on \textbf{translation invariance} properties.

\paragraph{Acknowledgments:} I would like to thank Christophe Garban for many helpful discussions and for his comments on earlier versions of the manuscript. I would also like to thank Vincent Tassion for fruitful discussions and for having welcomed me in Zürich several times. Finally, I would like to thank the anonymous referee for his$\cdot$her careful reading and helpful comments.

\section{Strategy and organization of the paper}\label{s.strat}

\subsection{Some notations}\label{ss.notations}

Before stating the main intermediate results and explaining the global strategy, let us introduce some notations.

\paragraph{Boxes, annuli and quads.} In all the paper, we will write $B_R=[-R,R]^2$ and we will write $A(r,R)=[-R,R]^2 \setminus (-r,r)^2$. Also, for every $y \in \R^2$, we will write $B_r(y) = y + B_r$ and $A(y;r,R) = y + A(r,R)$. A quad $Q$ is a topological rectangle in the plane with two distinguished opposite sides. Also, a black (respectively white) path included in $Q$ that joins one distinguished side to the other is called a crossing (respectively dual crossing). The event that $Q$ is crossed (respectively dual-crossed) will be written $\cross(Q)$ (respectively $\cross^*(Q)$).

\paragraph{Other notations.} In all the paper, we will use the following notations: (a) $\grandO{1}$ is a positive bounded function, (b) $\Omega(1)$ is a positive function bounded away from $0$ and (c) if $f$ and $g$ are two non-negative functions, then $f \asymp g$ means $\Omega(1) f \leq g \leq \grandO{1} f$.

We will also use the following notation: Let $(\Omega,\mathcal{F},\Pro)$ be a probability space. If $\mathcal{G} \subseteq \mathcal{F}$ is a $\sigma$-field, $B$ is some event such that $\Pro \left[ B \right] > 0$, and $A$ is some event, then
\[
\Pro \left[ A \mid B, \, \mathcal{G} \right] := \frac{\Pro \left[ A \cap B \mid \mathcal{G} \right]}{\Pro \left[ B \mid \mathcal{G} \right]} \, \un_{\lbrace \Pro [ B \mid \mathcal{G} ] > 0 \rbrace} \, .
\]
Note that, $\Pro [ \: \cdot \mid B ]$-a.s., we have: $\Pro[ A \mid B, \, \mathcal{G} ]$ is the conditional expectation of $A$ with respect to $\mathcal{G}$ and under $\Pro [ \: \cdot \mid B]$.


\subsection{Correlation inequalities for Voronoi percolation}\label{ss.correlation_in}

In this subsection, we recall two very useful families of correlation inequalities: the FKG-Harris inequalities and the BK inequalities, which are inequalities for increasing events. First, let us define what is an increasing event in our context. Since we work in random environment, it is interesting to consider quenched and annealed notions of increasing events.

\begin{defi}
\bi 
\item[i)] First, we recall the classical notion of increasing events. Let $E$ be a countable set. An event $A$ of the product $\sigma$-algebra on $\lbrace -1,1 \rbrace^E$ is increasing if for any $\omega,\omega' \in \lbrace -1,1 \rbrace^E $ such that $\omega \leq \omega'$ and $\omega \in A$, we have $\omega' \in A$.
\item[ii)] An event $A$ measurable with respect to the colored configuration $\omega$ is quenched-increasing if, for every point configuration of the plane $\eta$ and every $\omega,\omega' \in \lbrace -1,1 \rbrace^\eta$ such that $\omega \in A$ and $\omega \leq \omega'$, we have $\omega' \in A$.
\item[iii)] An event $A$ is annealed-increasing if, for any colored configuration $\omega \in A$ and any $\omega'$ obtained from $\omega$ by adding black points or deleting white points, we have $\omega' \in A$.
\ei
\end{defi}
Note that, if $A$ is annealed-increasing, then $A$ is quenched-increasing. 

\paragraph{The FKG-Harris inequalities.}

\bi
\item[i)] The classical FKG-Harris inequality is the following (see~\cite{grimmett1999percolation,bollobas2006percolation}): Let $E$ be a countable set. Remember that we write $\Prob^E_p := \left( p\delta_1+(1-p)\delta_{-1} \right)^{\otimes E}$. Let $A$ and $B$ be increasing events. Then, for every $p$ we have
\[
\Prob^E_p \left[ A \cap B \right] \geq \Prob^E_p \left[ A \right] \cdot \Prob^E_p \left[ B \right] \, .
\]
\item[ii)]
In the quenched case, the FKG-Harris inequality is a direct consequence of the above inequality and can be stated as follows: Let $A$ and $B$ be two quenched-increasing events. Then, for every point configuration of the plane $\eta$ and every $p$ we have
\[
\Prob^\eta_p \left[ A \cap B \right] \geq \Prob^\eta_p \left[ A \right] \cdot \Prob^\eta_p \left[ B \right] \, .
\]
\item[iii)] In the annealed case, we have: Let $A$ and $B$ be two annealed-increasing events. Then, for every $p$,
\[
\Pro_p \left[ A \cap B \right] \geq \Pro_p \left[ A \right] \cdot \Pro_p \left[ B \right] \, .
\]
See Lemma~$14$ in Chapter~$8$ of~\cite{bollobas2006percolation} for the proof of this inequality. (Note that this does not hold in general for quenched-increasing events; indeed if $A$ depends only on $\eta$ and if $\Pro \left[ A \right] \in ]0,1[$ then $A$ and $A^c$ are quenched-increasing and $0=\Pro \left[ A \cap A^c \right] < \Pro \left[ A \right] \Pro \left[ A^c \right]$.)
\ei

\paragraph{The BK inequalities.}

Let $A$ and $B$ be two quenched increasing events measurable with respect to $\omega$ restricted to a bounded domain. Define the disjoint occurrence of $A$ and $B$ as follows (where, for every colored configuration $\omega$, we write $\eta(\omega)$ for the underlying (non-colored) point configuration):
\begin{equation}\label{e.disjoint_occ}
A \square B = \left\lbrace \omega \in \Omega  \; : \; \exists I_1,I_2 \text{ finite disjoint subsets of } \eta(\omega), \; \omega^{I_1} \subseteq A \text{ and } \omega^{I_2} \subseteq B \right\rbrace  \, ,
\end{equation}
where $\Omega$ is the set of all colored configurations and, if $I \subseteq \eta(\omega)$, $\omega^I \subseteq \lbrace -1,1 \rbrace^{\eta(\omega)}$ is the set of all $\omega'$ such that $\omega'_i=\omega_i$ for every $i \in I$.

We will use the following quenched BK inequality which is a direct consequence of the classical BK inequality (see for instance~\cite{grimmett1999percolation} or~\cite{bollobas2006percolation}): For every $\eta$ and every $p$ we have
\begin{equation}\label{e.quenched_BK}
\Prob^\eta_p \left[ A \square B \right] \leq \Prob^\eta_p \left[ A \right] \cdot \Prob^\eta_p \left[ B \right] \, .
\end{equation}
Unfortunately, the annealed-version of the BK-inequality is only known for $p=1/2$ (and it seems actually not clear whether or not it should be true for $p \neq 1/2$). This will cause some difficulties when we want to extend some results to the near-critical phase, see Section~\ref{s.pneq1/2}.

\begin{prop}[Lemma~$3.4$ of~\cite{ahlberg2015quenched},\cite{joosten2012random} - both refer to van den Berg]\label{p.BK}
Let $A$ and $B$ be two annealed increasing events measurable with respect to the colored configuration $\omega$ restricted to a bounded domain. Then
\[
\Pro_{1/2} \left[ A \square B \right] \leq \Pro_{1/2} \left[ A \right] \cdot \Pro_{1/2} \left[ B \right] \, .
\]
\end{prop}

\subsection{Consequences of the annealed quasi-multiplicativity property}\label{ss.quasi}

In this subsection, we discuss important consequences of the (annealed) quasi-multiplicativity property Proposition~\ref{p.quasi}. As mentioned in Subsection~\ref{ss.voronoi}, Proposition~\ref{p.quasi} is the most technical result of the paper. For this reason, we have chosen to postpone its proof to the final section: Section~\ref{s.quasi}.  We will use this property in most of the other sections of the paper (for more about which section depends on which other section, see the beginning of Subsection~\ref{ss.orga}). Let us state some results that will be useful all along the paper and which are consequences of the quasi-multiplicativity property (and of intermediate results from Section~\ref{s.quasi}). These results are essentially useful to \textbf{overcome the spatial dependencies of the model} and will be crucial in Section~\ref{s.pivotals} where we deal with ``pivotal events''. We first need a definition.

\begin{defi}\label{d.hatarm}
We let
\[
\widehat{\arm}_j(r,R) = \left\lbrace \Pro \left[ \text{\textbf{A}}_j(r,R) \cond \omega \cap A(r,R) \right] > 0 \right\rbrace \, .
\]
In words, $\widehat{\arm}_j(r,R)$ is the event that, conditionally on the colored configuration in the annulus $A(r,R)$, the arm event $\arm_j(r,R)$ holds with positive probability.
\end{defi}

What is interesting with $\widehat{\arm}_j(r,R)$ is that \textbf{it is measurable with respect to $\omega \cap A(r,R)$}. Note also that a.s. $\arm_j(r,R) \subseteq \widehat{\arm}_j(r,R)$ (i.e. $\Pro \left[ \arm_j(r,R) \setminus \widehat{\arm}_j(r,R) \right] = 0$).\footnote{To prove this, use for instance the following result with $X=\un_{\arm_j(r,R)}$ and $\mathcal{G} = \sigma(\omega \cap A(r,R))$: Let $X$ be a non-negative random variable and let $\mathcal{G}$ be a sub-$\sigma$-field of the underlying $\sigma$-field. Then, a.s. we have: $\E \left[ X \cond \mathcal{G} \right] = 0 \Rightarrow X = 0$.} The following result will be proved in Subsection~\ref{ss.QM_consequences}:
\begin{prop}\label{p.fandalpha}
Let $j \in \N^*$, let $1 \leq r \leq R$, and write
\begin{eqnarray}
f_j(r,R)=f_{j,1/2}(r,R) & := & \Pro_{1/2} \left[ \widehat{\arm}_j(r,R) \right] \ .
\end{eqnarray}
There exists a constant $C = C(j) < +\infty$ such that
\[
\alpha^{an}_{j,1/2}(r,R) \leq f_j(r,R) \leq C \, \alpha^{an}_{j,1/2}(r,R) \, .
\]
\end{prop}
The following is a consequence of Proposition~\ref{p.fandalpha} and illustrates how this last proposition can help us to overcome spatial dependency problems.
\begin{prop}\label{p.techniquegeneral}
Let $j \in \N^*$. For every $h \in (0,1)$, there exists a constant $\epsilon = \epsilon(j,h) \in (0,1)$ such that, for every $1 \leq r \leq R$ and for every event $G$ which is measurable with respect to $\omega \setminus A(2r,R/2)$ and satisfies $\Pro_{1/2} \left[ G \right] \geq 1-\epsilon$, we have
\[
\Pro_{1/2} \left[ \arm_j(r,R) \cap G \right] \geq (1-h) \, \alpha^{an}_{j,1/2}(r,R) \, .
\]
\end{prop}
\begin{proof}
We have
\begin{eqnarray*}
\Pro_{1/2} \left[ \arm_j(r,R) \setminus G \right] & \leq & \Pro_{1/2} \left[ \widehat{\arm}_j(2r,R/2) \setminus G \right]\\
& = & f_j(2r,R/2) \cdot \Pro_{1/2} \left[ \neg G \right] \, ,
\end{eqnarray*}
by spatial independence. Proposition~\ref{p.fandalpha} implies that $f_j(2r,R/2) \asymp \alpha^{an}_{j,1/2}(2r,R/2)$. Moreover, the quasi-multiplicativity property and~\eqref{e.poly} imply that $\alpha^{an}_{j,1/2}(2r,R/2) \asymp \alpha^{an}_{j,1/2}(r,R)$, which ends the proof.
\end{proof}

\begin{rem}
Note that, with essentially the same proof, we obtain the following result: Let $j \in \N^*$. For every $h \in (0,1)$, there exists a constant $\epsilon = \epsilon(j,h) \in (0,1)$ such that, for every $1 \leq r \leq \rho \leq R$ and for every event $G$ which is measurable with respect to $\omega \setminus \left( A(2r,\rho/2) \cup A(2\rho,R/2) \right)$ and satisfies $\Pro_{1/2} \left[ G \right] \geq 1-\epsilon$, we have
\[
\Pro_{1/2} \left[ \arm_j(r,R) \cap G \right] \geq (1-h) \, \alpha^{an}_{j,1/2}(r,R) \, .
\] 
\end{rem}



In Section~\ref{s.quasi}, we also use the quasi-multiplicativity property to compute universal arm exponents. For every $1 \leq r \leq R$, let $\alpha^{an,+}_{j,1/2}(r,R)$ denote the probability of the $j$-event in the half plane (i.e. the event that there are $j$ paths of alternating colors from $\partial B_r$ to $\partial B_R$ that live in the upper half-plane). See Subsection~\ref{ss.half-plane}: the quasi-multiplicativity property can also be proved for these quantities. We have the following:

\begin{prop}\label{p.universal}
The computation of the universal arm-exponents (that goes back as far as we know to Aizenman\footnote{See the first exercise sheet of \cite{werner2007lectures} for the proof in the case of Bernoulli percolation on the triangular lattice.}) holds for Voronoi percolation: Let $1 \leq r \leq R$, we have
\bi
\item[\textup{i)}] $\alpha^{an,+}_{2,1/2}(r,R) \asymp r/R \,$,
\item[\textup{ii)}] $\alpha^{an,+}_{3,1/2}(r,R) \asymp \left( r/R \right)^2 \,$,
\item[\textup{iii)}] $\alpha^{an}_{5,1/2}(r,R) \asymp \left( r/R \right)^2 \,$. 
\ei
\end{prop}

Items~i) and~ii) of Proposition~\ref{p.universal} are proved in Subsection~\ref{ss.half-plane} while Item~iii) is proved in Subsection~\ref{ss.QM_odd}.

\subsection{Some important events: the pivotal events and the ``good'' events}

\subsubsection{Pivotal events}\label{ss.intro_piv}

A crucial step in the proof of the scaling relations is the study of pivotal events for crossing and arm events. In the present work, we introduce \textbf{a quenched and an annealed definitions for pivotal events}. Let us begin with a classical definition: Let $E$ be a countable set and let $A$ be an event of the product $\sigma$-algebra on $\lbrace -1,1 \rbrace^E$. A point $i \in E$ is pivotal for a configuration $\omega \in \lbrace -1,1 \rbrace^E$ and the event $A$ if changing the value of $\omega_i$ changes the value of $\un_A(\omega)$. We write $\Piv_i^E(A)$ for the event that $i$ is pivotal for $A$ (if $E = \lbrace 1, \cdots, n \rbrace$, we denote this event by $\Piv_i^n(A)$). More generally, if $I$ is a finite subset of $E$, we say that $I$ is pivotal for $\omega$ and $A$ if there exists $\omega' \in \lbrace -1,1 \rbrace^E$ such that $\omega$ and $\omega'$ coincide outside of $I$ and $\un_A(\omega') \neq \un_A(\omega)$. We denote by $\Piv^E_I(A)$ the corresponding event. Let us now introduce a quenched and an annealed notions of pivotal sets. The quenched version is very similar to the above notion: 

\begin{defi}\label{d.quenched_piv}
Let $A$ be an event measurable with respect to the colored configuration $\omega$ and let $\eta$ be the underlying (non-colored) point configuration. A bounded Borel set $D$ is quenched-pivotal for $\omega$ and $A$ if there exists $\omega' \in \lbrace -1,1 \rbrace^\eta$ (note that $\omega'$ has the same underlying point configuration as $\omega$) such that $\omega$ and $\omega'$ coincide on $\eta \cap D^c$ and $\un_A(\omega') \neq \un_A(\omega)$. We write $\Piv^q_D(A)$ for the event that $D$ is quenched-pivotal for $A$.\\
We also use the following terminology: if $x \in \eta$, we say that $x$ is quenched-pivotal for $A$ if changing the color of $x$ modifies the value of $\un_A$. If we work conditionally on $\eta$ and if $x \in \eta$, then $\left\lbrace  x \text{ is quenched-pivotal for } A \right\rbrace $ is an event of the product space $\lbrace -1,1 \rbrace^\eta$, and we denote this event by $\Piv^q_x(A)$.
\end{defi}

\begin{defi}\label{d.annealed_piv}
A bounded Borel set $D$ is annealed-pivotal for some colored configuration $\omega$ and some event $A$ if both $\Pro_p \left[ A \, | \, \omega \setminus D \right]$ and $\Pro_p \left[ \neg A \, | \, \omega \setminus D \right]$ are positive. We write $\Piv_D(A)$ for the event that $D$ is annealed-pivotal for $A$ (note that we omit the parameter $p$ in the notation; actually, as far as $p \in (0,1)$ and since $D$ is bounded, the event $\Piv_D(A)$ does not depend on $p$).
\end{defi}


We have the following link between annealed and quenched pivotal events: Let $p \in (0,1)$, let $D$ be a bounded Borel set, and let $A$ be an event measurable with respect to the colored configuration $\omega$. Then, a.s. we have $\Piv^q_D(A) \subseteq \Piv_D(A)$, i.e. $\Pro_p \left[ \Piv_D^q(A) \setminus \Piv_D(A) \right] = 0$. This is an easy consequence\footnote{Use for instance the following result with $X=\un_A$, $\mathcal{G}_1 = \sigma(\omega \setminus D)$ and $\mathcal{G}_2 = \sigma(\eta, \, \omega \setminus D)$: Let $X$ be a non-negative random variable and let $\mathcal{G}_1 \subseteq \mathcal{G}_2$ be two sub-$\sigma$-fields of the underlying $\sigma$-field. Then, a.s. we have: $\E \left[ X \cond \mathcal{G}_1 \right] = 0 \Rightarrow \E \left[ X \cond \mathcal{G}_2 \right] = 0$.} of the fact that, if $D$ is quenched-pivotal for $A$, then a.s. (since $\eta \cap D$ is finite) $\displaystyle \Pro \left[ A \, | \, \omega \setminus D, \, \eta \cap D \right]$ and $\displaystyle \Pro \left[ \neg A \, | \, \omega \setminus D, \, \eta \cap D \right]$ are positive.\\

Let $Q$ be a quad. The event that some box is (quenched or annealed) pivotal for the event $\cross(Q)$ is closely related to arm events and particularly to the $4$-arm events. We prove estimates in this spirit in Subsections~\ref{ss.pivotals1} and~\ref{ss.pivotals2}.

\subsubsection{The events $\dense$, $\qbc$ and $\gi$}\label{ss.formal_events}

In this subsection, we define three ``good'' events that we will use all along the paper. Their introduction is motivated by the three following observations: i) There are less spatial dependencies when the point configuration $\eta$ is sufficiently dense. ii) It is often interesting to condition on $\eta$ since the conditional measure is the product measure $\Pro^\eta_p=(p\delta_1+(1-p)\delta_{-1})^{\otimes \eta}$. To apply geometric arguments under this quenched measure, we need $\Pro^\eta_p[\cross(Q)]$ to be non-negligible for a large family of quads $Q$. iii) It is easier to deal with arm events when the arms are well separated.

\begin{defi}\label{d.dense}
Let $D$ be a bounded subset of the plane and let $\delta \in (0,1)$. We denote by $\dense_\delta(D)$ the event that, for every point $u \in D$, there exists $x \in \eta \cap D$ such that $||x-u||_2 \leq \delta \cdot \text{diam}(D)$.
\end{defi}

\begin{lem}\label{l.dense}
Let $R \geq 1$ and $\delta \in (0,1)$. We have
\[
\Pro \left[ \dense_\delta(B_R) \right] \geq 1-\grandO{1} \delta^{-2} \exp \left( -\frac{(\delta \cdot R)^2}{2} \right) \, .
\]
\end{lem}

\begin{proof}
This lemma can be obtained by covering $B_R$ by a family $(S_i)_{1 \leq i \leq N}$ of $N \asymp \delta^{-2}$ squares of side-length $\delta \cdot R/\sqrt{2}$ and by observing that:
\[
\dense_\delta(B_R) \supseteq \lbrace \forall i,  \, \eta \cap S_i \neq \emptyset \rbrace
\]
and:
\[
\forall i, \, \Pro \left[ \eta \cap S_i = \emptyset \right] = \exp \left( -\frac{(\delta \cdot R)^2}{2} \right)\, .
\]
See Lemma~$18$ in Chapter~$8$ of~\cite{bollobas2006percolation} for the proof of a similar result.
\end{proof}

In the following, we restrict ourselves to the case $p=1/2$. See~Subsection~\ref{ss.pneq1/2_2} for the extension of the results to the near-critical phase.

\begin{defi}\label{d.a_lot_of_quads}
Let $D$ be a subset of the plane and let $\delta \in (0,1)$. We denote by $\mathcal{Q}'_\delta(D)$ the set of all quads $Q \subseteq D$ which are drawn on the grid $(\delta \, \diam(D)) \cdot \Z^2$ (i.e. whose sides are included in the edges of $(\delta \, \diam(D)) \cdot \Z^2$ and whose corners are vertices of $(\delta \, \diam(D)) \cdot \Z^2$). Also, we denote by $\mathcal{Q}_\delta(D)$ the set of all quads $Q \subseteq D$ such that there exists a quad $Q' \in \mathcal{Q}'_\delta(D)$ satisfying $\cross(Q') \subseteq \cross(Q)$.
\end{defi}

The following result will be proved in Subsection~\ref{ss.a_lot_of_quads} by using Theorem~\ref{t.AGMT}.

\begin{prop}\label{p.a_lot_of_quads}
There is an absolute constant $C<+\infty$ such that the following holds: Let $\delta \in (0,1)$ and $\gamma \in (0,+\infty)$. There exists a constant $c = c(\delta,\gamma) \in (0,1)$ such that, for every bounded subset of the plane $D$ that satisfies $\diam(D) \geq \delta^{-2}/100$, we have
\[
\Pro \left[ \qbc^\gamma_\delta(D) \right] \geq 1 - C \diam(D)^{-\gamma} \, , 
\]
where
\[
\qbc^\gamma_\delta(D) = \left\lbrace \forall Q \in \mathcal{Q}_{\delta}(D), \, \Prob^\eta_{1/2} \left[ \cross(Q) \right] \geq c(\delta,\gamma) \right\rbrace \, .
\]
The notation $\qbc$ means ``Quenched Box-Crossing property''.
\end{prop}


Let us end this subsection by defining quantities related to the well-separateness of interfaces. Let $\delta \in (0,1)$, let $1 \leq r \leq R$, and let $\beta_1, \cdots, \beta_k$ be the interfaces from $\partial B_r$ to $\partial B_R$ (an interface is a continuous path $\beta$ drawn on the edges of the Voronoi tiling and such that one side of $\beta$ is black and its other side is white). Also, let $z_i^{ext}$ (respectively $z_i^{int}$) denote the endpoint on $\partial B_R$ (respectively on $\partial B_r$) of $\beta_i$, and let $s^{ext}(r,R)$ (resp. $s^{int}(r,R)$) be the least distance between $z_i^{ext}$ (respectively $z_i^{int}$) and $\cup_{j \neq i} \beta_j$.

Let $\gi^{ext}_\delta (R)$ (for ``Good Interfaces'') be the event that there does not exist $y \in \partial B_R$ such that the $3$-arm event in $A(y;10 \delta R, R/4) \cap B_R$ holds. Note that, if $r \leq 3R/4$, then
\[
\gi^{ext}_\delta (R) \subseteq \lbrace s^{ext}(r,R) \geq 10 \delta R \rbrace \, .
\]
This inclusion will be very useful. Note that $\lbrace s^{ext}(r,R) \geq 10 \delta R \rbrace$ is not monotonic in $r$. This is actually the reason why we have introduced the event $\gi^{ext}_\delta (R)$: this event ``depends only on the crossings in $A(3R/4,R)$'' and is included in $\lbrace s^{ext}(r,R) \geq 10 \delta R \rbrace$ (for $r \leq 3R/4$).

Similarly, let $\gi^{int}_\delta (r)$ be the event that there does not exist $y \in \partial B_r$ such that the $3$-arm event in $A(y;10 \delta r, r/2) \setminus B_r$ holds. Note that, if $R \geq 3r/2$,
\[
\gi^{int}_\delta (r) \subseteq \lbrace s^{int}(r,R) \geq 10 \delta r \rbrace \, .
\]
We will prove the following lemma in Subsection~\ref{ss.preliminary}:
\begin{lem}\label{l.interfaces}
Let $\delta \in (0,1)$ and let $r,R \geq 100 \, \delta^{-1}$. There exist absolute constants $C < +\infty$ and $\epsilon > 0$ such that
\[
\Pro_{1/2} \left[ \gi^{ext}_\delta(R) \right] \geq 1- C \, \delta \, ,
\]
and
\[
\Pro_{1/2} \left[ \gi^{int}_\delta(r) \right] \geq  1- C \, \delta^{\epsilon} \, .
\]
\end{lem}

\subsection{Organization of the paper and interdependence of the sections}\label{ss.orga}

As explained in Subsection~\ref{ss.quasi}, we postpone the proof of the quasi-multiplicativity property to the final section: Section~\ref{s.quasi}. We summarize the interdependence of the sections of the paper in Figures~\ref{f.dep_1} and~\ref{f.dep_2}.

\begin{figure}[!h]
\begin{center}
\includegraphics[scale=0.60]{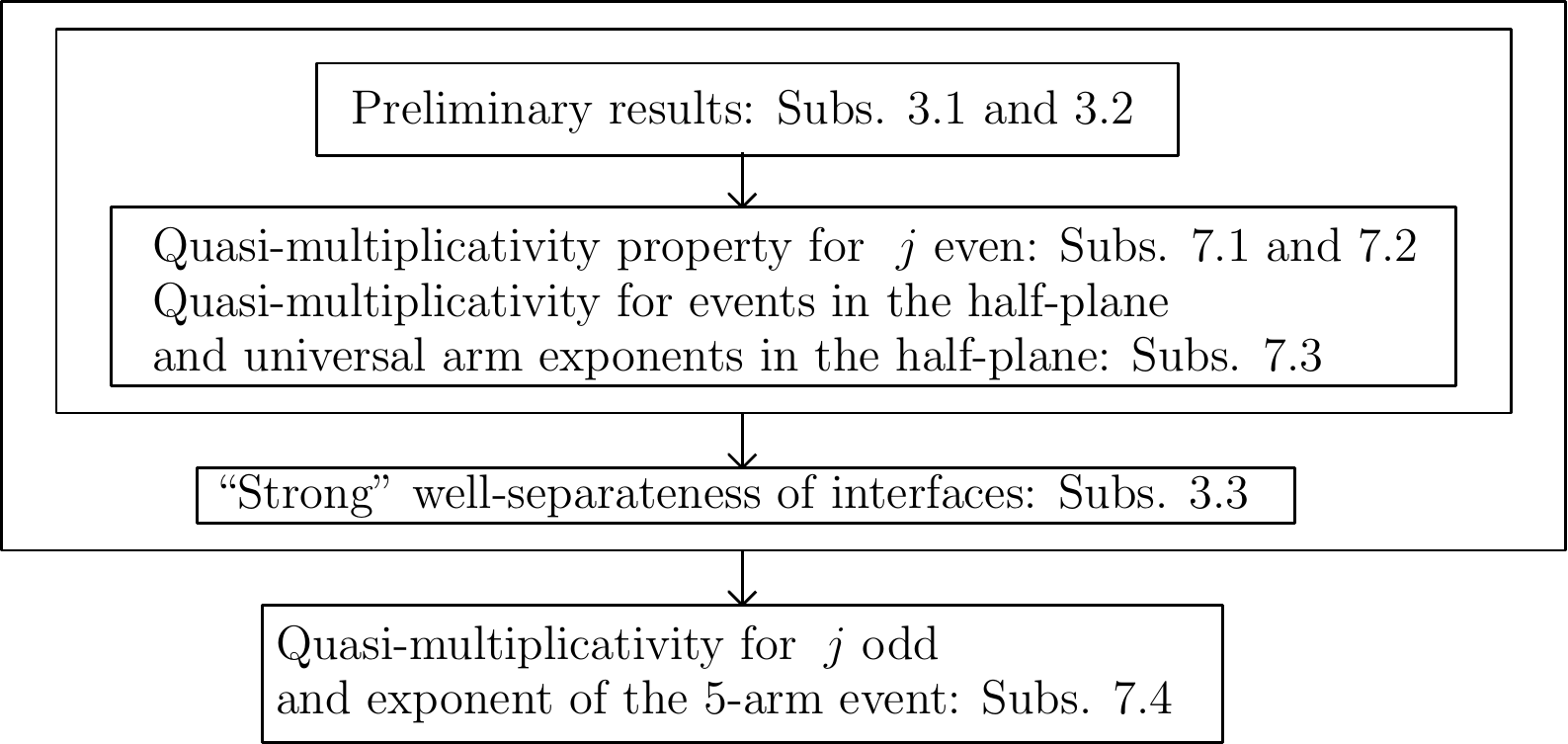}
\end{center}
\caption{Interdependence between Sections~\ref{s.first} and~\ref{s.quasi}.}\label{f.dep_1}
\end{figure}

\begin{figure}[!h]
\begin{center}
\includegraphics[scale=0.60]{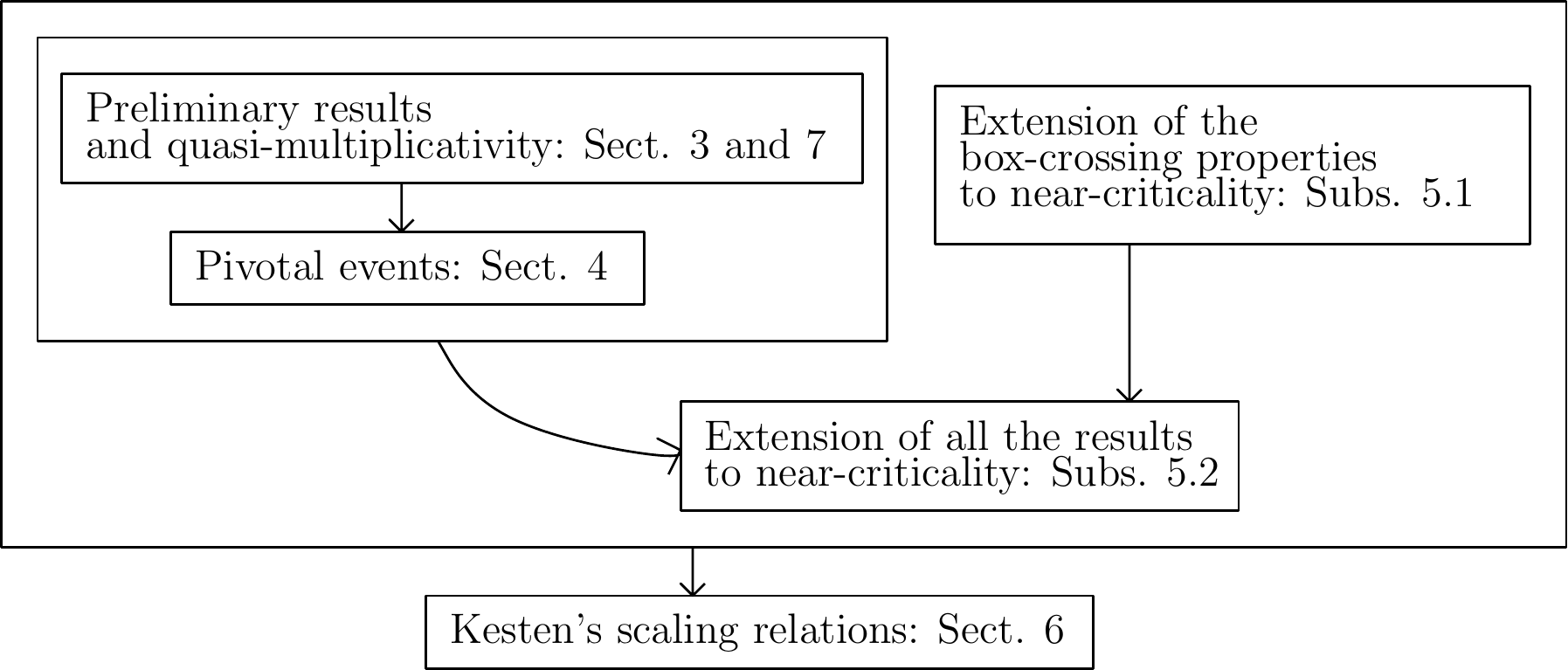}
\end{center}
\caption{Interdependence between Sections~\ref{s.first} to~\ref{s.quasi}.}\label{f.dep_2}
\end{figure}

\subsection{Some ideas of proof}\label{ss.strategy}

Let us end Section~\ref{s.strat} by giving a few more details about our strategies of proofs.

\subsubsection{The quasi-multiplicativity property (at $p=1/2$)}\label{ss.strategy_quasi}

In this subsection, we work at the parameter $p=1/2$, and we explain ideas behind the proof of the quasi-multiplicativity property Proposition~\ref{p.quasi} in the case $j \geq 2$ (see Subsection~\ref{ss.warm} for the proof of the easier case $j=1$). The proof is written in Section~\ref{s.quasi}. We begin with two observations that illustrate the new difficulties compared to the study of Bernoulli percolation on a deterministic lattice.

\bi 
\item[(a)] The first observation is that, for Voronoi percolation, the following result does not seem easier to prove than the quasi-multiplicativity property itself: There exists a constant $C=C(j)$ such that, for every $R \geq 100$,
\[
\alpha^{an}_{j,1/2}(100,R) \leq C \, \alpha^{an}_{j,1/2}(10,R) \, .
\]
A first idea to prove the above would be to condition on the event $\arm_j(100,R)$ and on the colored configuration outside of $B_{100}$ and then extend the arms to $\partial B_{10}$ ``by hands''. The problem is that $\arm_j(100,R)$ is a degenerated event and thus, as already suggested in Figure~\ref{f.difficulties}, it does not seem obvious at all that the following does not happen: ``If we condition on $\arm_j(100,R)$, then, with high probability, the point configuration in the neighbourhood of $\partial B_{100}$ is very dense''. This may be a problem since it is difficult to extend the arms ``by hands'' when the point configuration is very dense.
\item[(b)] For Bernoulli percolation on a deterministic lattice, the left-hand-inequality of the quasi-multiplicativity property is an easy consequence of the independence on disjoint sets. For Voronoi percolation, even the following does not seem easy to prove: There exists $C=C(j) < +\infty$ such that, for every $1 \leq r_1 \leq r_2 \leq r_3/2 \,$: $\alpha^{an}_j(r_1,r_3) \leq C \, \alpha^{an}_j(r_1,r_2) \, \alpha^{an}_j(2r_2,r_3)$. 
However, one can note that the left-hand-inequality of the quasi-multiplicativity property is a direct consequence of Proposition~\ref{p.fandalpha} (which enables to use spatial independence properties). Actually, our strategy will be the following: we will first prove Lemma~\ref{l.looksgood} and Corollary~\ref{c.looksgood} which are results analogous to Proposition~\ref{p.fandalpha}. Then, we will prove
the quasi-multiplicativity property, and finally we will prove Proposition~\ref{p.fandalpha}.
\ei




Now, let us be a little more precise about the proof of the quasi-multiplicativity property. In the spirit of~\cite{kesten1987scaling,werner2007lectures,nolin2008near,schramm2010quantitative}, we will prove the following properties:
\bi 
\item[i)] If $\arm_j(r,R)$ holds and if the configuration ``looks good'' near $\partial B_R$, then we can extend the arms at larger scale, see Lemma~\ref{l.extension}.
\item[ii)] If we condition on $\arm_j(r,R)$, then the configuration near $\partial B_R$ ``looks good" with non-negligible probability, see Lemma~\ref{l.looksgood}.
\ei

A difference with the case of Bernoulli percolation on a deterministic lattice is that, in the notion of ``looking good'', we will have to ask that \textbf{both the random tiling and the random coloring} look good. Concerning the random coloring: As in the case of Bernoulli percolation on $\Z^2$ or on $\T$, we will ask that the interfaces between black and white crossings are well separated so that we can use  box-crossing estimates. Concerning the random tiling: $(1)$ To avoid spatial dependence problems, we will ask that $\dense_\delta(A(R/2,2R))$ (see Definitions~\ref{d.dense}) holds for some well-chosen $\delta>0$. $(2)$ In order to use box-crossing estimates when we condition on $\eta$, we will ask that $\qbc^1_\delta(A(R))$ (see Proposition~\ref{p.a_lot_of_quads}) holds for some well-chosen annulus $A(R)$ at scale $R$ and $\delta>0$. The idea is that, if the interfaces are well separated, if the two conditions $(1)$ and $(2)$ above are satisfied, and if we condition on $\eta$ and on the interfaces, then we can extend the arms by using box-crossing techniques and the (quenched) Harris-FKG inequality. As we will see in Section~\ref{s.quasi}, we will have to consider events a little more complicated because we will want the events to be \textbf{measurable with respect to $\omega \cap A(R/2,2R)$}.

Also, we will see that, for technical reasons, we will have to consider different notions of well-separateness of interfaces. More precisely, we will first prove the quasi-multiplicativity property in the case $j$ even (in Subsection~\ref{ss.quasi_even}) and by using the following definition of well-separateness: two interfaces are well separated if their end-points are. By following the same proof, we will also obtain the quasi-multiplicativity property for $j$-arm events in the half-plane (with $j$ either even or odd). Thanks to this last property, we will be able to compute the universal exponent of the $3$-arm event in the half-plane. Then, it will be possible (by using our knowledge on $\alpha_{3,1/2}^+(r,R)$) to deal with the following slightly different definition of well-separateness of interfaces: two interfaces are well-separated if the end-point of each of them is far enough from the union of the other interfaces (and not only far enough from the other end-points). This other notion of well-separateness is the one defined in Subsection~\ref{ss.formal_events}, and we will need this notion to prove the quasi-multiplicativity property in the case $j$ odd (see Subsection~\ref{ss.QM_odd}).

\subsubsection{The anneled scaling relations}\label{ss.strategy_kesten}

Once we have proved the quasi-multiplicativity property and all the results stated in Section~\ref{s.strat}, the ideas for the proof of the annealed scaling relations are the same as in the original paper of Kesten~\cite{kesten1987scaling} (see also~\cite{werner2007lectures,nolin2008near}). The only difference is that we will need to combine quenched and annealed notions of pivotal events.

\section{Preliminary results}\label{s.first}

In this section, we only work at the parameter $p=1/2$, hence we intentionally forget the subscript $p$ in the notations.




\subsection{Warm-up: proof of~\eqref{e.poly} and of the quasi-multiplicativity property for $j=1$}\label{ss.warm}

In this subsection, we prove that the probabilities of arm events decay polynomially fast, i.e. we prove~\eqref{e.poly} (this can be seen as an illustration of how we use the events ``$\dense$'' from Definition~\ref{d.dense}). We also prove the quasi-multiplicativity property in the case $j=1$ (this can be seen as an illustration of how we use the events ``$\dense$'' and events of the kind $\widehat{\arm}_j(r,R)$ from Definition~\ref{d.hatarm}). To prove these inequalities, we do not rely on any result proved in this paper but only on the (annealed) FKG property and on the (annealed) box-crossing property Theorem~\ref{t.Tassion}.

\begin{proof}[Proof of~\eqref{e.poly}] As explained below~\eqref{e.poly}, the upper-bound is proved in~\cite{tassion2014crossing}. Let us prove the lower-bound. First, note that we can choose a constant $M=M(j) \in [10,+\infty)$ such that we can define $j$ sets of $2n \asymp \log(R/r)$ rectangles: $\lbrace Q_i^1 \cdots, Q_i^{2n} \rbrace$, $i = 0, \cdots, j-1$ that satisfy:
\bi 
\item[(a)] For every $i \in \lbrace 0, \cdots, j-1 \rbrace$ and every $l \in \lbrace 1, \cdots, n \rbrace$, $Q_i^{2l-1}$ and $Q_i^{2l}$ are $(2^lr) \times (2^{l-M}r)$ rectangles;
\item[(b)] For all $i \neq i' \in \lbrace 0, \cdots, j-1 \rbrace$ and all $l,l' \in \lbrace 1, \cdots, 2n \rbrace$, $Q_i^l$ is at distance at least $\max(2^{l-M}r,2^{l'-M}r)$ from $Q_{i'}^{l'}\,$;
\item[(c)] If for every $l \in \lbrace 1, \cdots, 2n \rbrace$ and every $i \in \lbrace 0, \cdots, j-1 \rbrace$ even (respectively odd) the rectangle $Q_i^l$ is crossed lengthwise (respectively dual-crossed lengthwise), then $\arm_j(r,R)$ holds. (See Figure~\ref{f.poly}.)
\ei

\begin{figure}[!h]
\begin{center}
\includegraphics[scale=0.8]{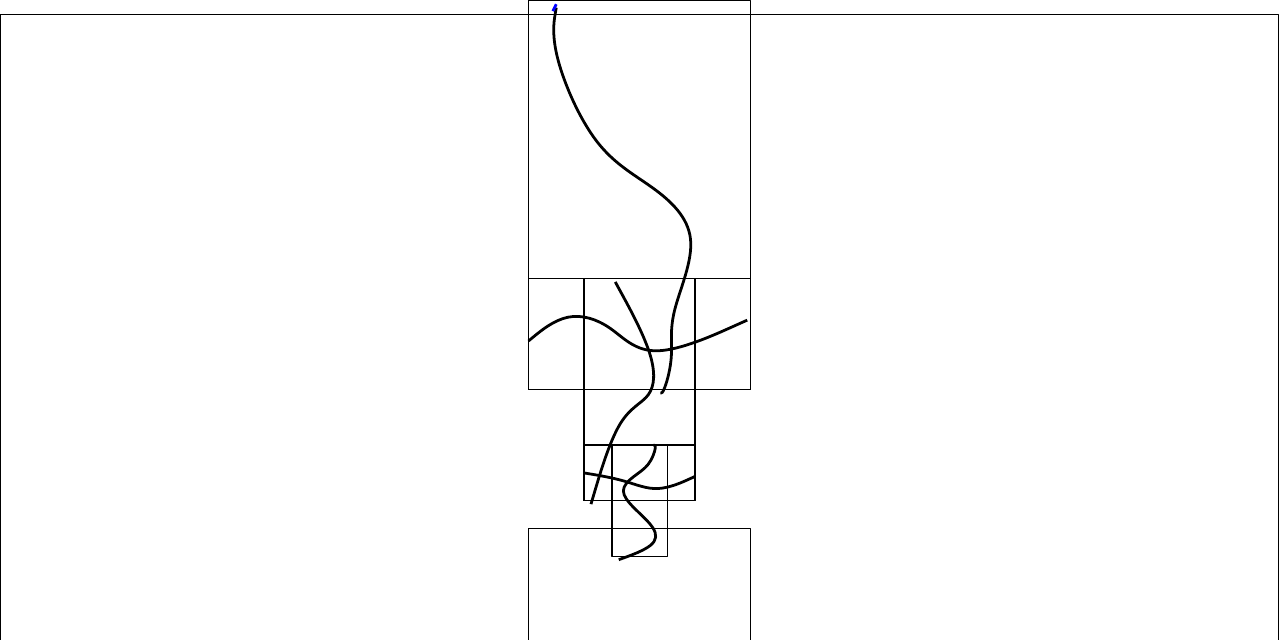}
\end{center}
\caption{The rectangles $Q_i^l$ for some $i$.}\label{f.poly}
\end{figure}

Let $i \in \lbrace 0, \cdots, j-1 \rbrace$ even (respectively odd), and let $\widetilde{\dense}(Q_i^l)$ be the event that, for any $u \in Q_i^l$, there exists a black (respectively white) point $x \in \eta \cap Q_i^l$ at Euclidean distance less than $2^{l-2M}r$ from $x$. Note that the event $\widetilde{\dense}(Q_i^l)$ is slightly different from the event $\dense(Q_i^l)$ of Definition~\ref{d.dense}; in particular, it is annealed increasing (respectively annealed decreasing) if $i$ is even (respectively odd).

We have
\begin{eqnarray*}
\alpha^{an}_j(r,R)\geq \Pro \left[ \bigcap_{i,l} \cross(Q_i^l) \right] \geq \Pro \left[ \bigcap_{i,l} \cross(Q_i^l) \cap \widetilde{\dense}(Q_i^l) \right] \, .
\end{eqnarray*}

Next, note that the $j$ events $\cap_{l=1}^{2n} \left( \cross(Q_i^l) \cap \widetilde{\dense}(Q_i^l) \right)$ are independent. As a result the above equals
\[
\prod_{i=0}^{j-1} \Pro \left[ \bigcap_{l=1}^{2n} \cross(Q_i^l) \cap \widetilde{\dense}(Q_i^l) \right] \, .
\]
We can now use the (annealed) FKG-Harris inequality. Indeed, for every $i$ even (respectively odd) and every $l \in \lbrace 1, \cdots, 2n \rbrace$, the event $\cross(Q_i^l) \cap \widetilde{\dense}(Q^l_i)$ is annealed increasing (respectively annealed decreasing). We thus obtain that the above is at least
\[
\prod_{i=0}^{j-1} \prod_{l=1}^{2n} \Pro \left[ \cross(Q_i^l) \cap \widetilde{\dense}(Q_i^l) \right] \, .
\]
By the same proof as Lemma~\ref{l.dense}, we have: $\Pro \left[ \widetilde{\dense}(Q_i^l) \right] \geq 1-C \exp(-c2^{2l})$, for some $C=C(M)<+\infty$ and $c=c(M) > 0$. By using this estimate and the box-crossing property Theorem~\ref{t.Tassion}, we obtain that there exists a constant $c'=c'(M) > 0$ such that, for every $l$ large enough (larger than some $l_0(M)$, say) and every $i$, $\Pro_{1/2} \left[ \cross(Q_i^l) \cap \widetilde{\dense}(Q_i^l) \right] \geq c'$. Moreover, it is easy to see that, for every $i \in \{ 0, \cdots, j-1 \}$ and every $l \in \{ 1, \cdots, l_0 \}$, we have $\Pro_{1/2} \left[ \cross(Q_i^l) \cap \widetilde{\dense}(Q_i^l) \right] \geq c''$ for some $c''=c''(M,l_0)>0$. Thus, we have
\[
\alpha_j^{an}(r,R) \geq (\min\{c',c''\})^{2nj} \, ,
\]
which ends the proof.
\end{proof}

\begin{proof}[Proof of the quasi-multiplicativity property in the case $j=1$] As pointed out in Remark~\ref{r.quasi_j=1}, if $j=1$ then the right-hand-inequality of the quasi-multiplicativity property is a direct consequence of the annealed FKG-Harris inequality and of the annealed box-crossing result Theorem~\ref{t.Tassion}. Here, we prove the left-hand-inequality (by relying on the right-hand-inequality). The main difficulty is the lack of spatial independence. To overcome it, we work with the following events and quantities (analogous to those introduced in Subsection~\ref{ss.quasi}). Let $1 \leq r \leq R$ and
\begin{eqnarray*}
\widehat{\arm}^{ext}_1(r,R) & := & \left\lbrace \Pro \left[ \arm_1(r,R) \cond \omega \cap B_R \right] > 0 \right\rbrace \, ,\\
\widehat{\arm}^{int}_1(r,R) & := & \left\lbrace \Pro \left[ \arm_1(r,R) \cond \omega \setminus B_r \right] > 0 \right\rbrace \, ,\\
f_1^{ext}(r,R) & := & \Pro \left[ \widehat{\arm}^{ext}_1(r,R) \right] \, ,\\
f_1^{int}(r,R) & := & \Pro \left[ \widehat{\arm}^{int}_1(r,R) \right] \, .
\end{eqnarray*}
Note that $\alpha^{an}_1(r,R) \leq f_1^{ext}(r,R)$ and $\alpha^{an}_1(r,R) \leq f_1^ {int}(r,R)$. What is interesting with these events is that,
if $1 \leq r_1 \leq r_2 \leq r_3$, then $\widehat{\arm}_1^{ext}(r_1,r_2)$ and $\widehat{\arm}_1^{int}(r_2,r_3)$ are independent (indeed, the first one is measurable with respect to $\omega \cap B_{r_2}$ while the second one is measurable with respect to $\omega \setminus B_{r_2}$). Hence we have
\begin{eqnarray*}
\alpha^{an}_1(r_1,r_3) & \leq & \Pro \left[ \widehat{\arm}_1^{ext}(r_1,r_2) \cap \widehat{\arm}_1^{int}(r_2,r_3) \right]\\
& = & f_1^{ext}(r_1,r_2) \, f_1^{int}(r_2,r_3) \, .
\end{eqnarray*}
As a result, it is sufficient to prove that $f_1^{ext}(r,R),f_1^{int}(r,R) \leq \grandO{1} \alpha^{an}_1(r,R)$. We prove this only for $f_1^{int}(r,R)$ since the proof for $f_1^{ext}(r,R)$ is the same.

Let $\dense(r):= \dense_{1/100} \left( A(r/2,2r) \right)$ where $\dense_\delta(D)$ is defined in Definition~\ref{d.dense}. With the same proof as Lemma~\ref{l.dense}, we have: $\Pro \left[ \dense(r) \right] \geq 1 - \grandO{1} \exp \left( - \Omega(1) \, r^2 \right)$. If $\dense(r)$ holds, then we have the following: if $x \in \eta$ is such that the Voronoi cell of $x$ intersects $A(2r,R)$, then $x \notin B_r$. As a result, $\widehat{\arm}_1^{int}(r,R) \cap \dense(r) \subseteq \arm_1(2r,R)$. So
\[
f_1^{int}(r,R) \leq \alpha^{an}_1(2r,R) + \Pro \left[ \widehat{\arm}_1^{int}(r,R) \setminus \dense(r) \right] \, .
\]
Moreover, by using the fact that $\dense(r)$ and $\widehat{\arm}_1^{int}(2r,R)$ are independent (the first one is measurable with respect to $\eta \cap A(r/2,2r)$ while the second one is measurable with respect to $\omega \setminus B_{2r}$), we obtain that $f^{int}_1(r,R)$ is at most
\begin{eqnarray*}
 \alpha^{an}_1(2r,R) +\Pro \left[ \widehat{\arm}_1^{int}(r,R) \setminus \dense(r) \right] & \leq &  \alpha^{an}_1(2r,R) +\Pro \left[ \widehat{\arm}_1^{int}(2r,R) \setminus \dense(r) \right]\\
& = &  \alpha^{an}_1(2r,R) +f_1^{int}(2r,R) \, (1 - \Pro \left[ \dense(r) \right])\\
& \leq &  \alpha^{an}_1(2r,R) +f_1^{int}(2r,R) \, \grandO{1} \exp \left( - \Omega(1) \, r^2 \right) \, .
\end{eqnarray*}
By iterating the above inequality, we obtain that
\begin{multline*}
f_1^{int}(r,R) \leq \alpha^{an}_1(2r,R) + \grandO{1}  \sum_{i=0}^{\lfloor \log_2(R/r) \rfloor - 2} \Big( \alpha^{an}_1(2^{i+2}r,R) \, \exp \left( - \Omega(1) \, (2^i r)^2 \right) \Big)\\
+ \grandO{1} \exp \left( - \Omega(1) \, (2^{\lfloor \log_2(R/r) \rfloor - 1} r)^2 \right) \, .
\end{multline*}
We now use the right-hand-inequality of the quasi-multiplicativity property and~\eqref{e.poly}, which imply that there exists a constant $C_1 < +\infty$ such that, for every $i \in \lbrace 1, \cdots, \lfloor \log_2(R/r) \rfloor +1 \rbrace$, we have
\[
\alpha^{an}_1(2^ir,R) \leq C_1^i \, \alpha^{an}_1(r,R) \, .
\]
We finally obtain
\begin{multline*}
f_1^{int}(r,R) \leq \grandO{1} \alpha^{an}_1(r,R) \times \bigg( C_1 + \sum_{i=0}^{\lfloor \log_2(R/r) \rfloor - 2} \left( C_1^{i+2} \, \exp \left( - \Omega(1) \, (2^i r)^2 \right) \right)\\
+ C_1^{\lfloor \log_2(R/r) \rfloor +1} \, \exp \left( - \Omega(1) \, (2^{\lfloor \log_2(R/r) \rfloor - 1} r)^2 \right) \bigg) \, .
\end{multline*}
This ends the proof since the quantity between parentheses can be bounded by some absolute constant.
\end{proof}

\subsection{A generalization of Theorem~\ref{t.AGMT} to a family of quads}\label{ss.a_lot_of_quads}

The fact that we can choose any $\gamma > 0$ in the quenched box-crossing property Theorem~\ref{t.AGMT} is crucial for us. In particular, this implies that the quenched box crossing property is true \textbf{for a lot of quads simultaneously} with high probability. In this subsection, we use the notations from Definition~\ref{d.a_lot_of_quads} and Proposition~\ref{p.a_lot_of_quads} and we prove Proposition~\ref{p.a_lot_of_quads}.

\begin{proof}[Proof of Proposition~\ref{p.a_lot_of_quads}]
Let $\left( Q_i \right)_{i \in \lbrace 1, \cdots, N(D,\delta)\rbrace}$ be an enumeration of all $2 \delta \, \text{diam}(D) \times \delta \text{diam}(D)$ rectangles that intersect $D$ and that are drawn on the grid $(\delta \, \diam(D)) \cdot \Z^2$. Note that, if $Q \in \mathcal{Q}_{\delta}(D)$, then
\[
\bigcap_{i=1}^{N(D,\delta)} \lbrace  Q_i \text{ is crossed lengthwise} \rbrace \subseteq \cross(Q) \, .
\]
The (quenched) FKG-Harris inequality implies that, for every $Q \in \mathcal{Q}_{\delta}(D)$ and for every $\eta$ we have
\begin{equation}\label{e.GL2_1}
\prod_{i \in \lbrace 1, \cdots, N(D,\delta)\rbrace} \Prob^\eta \left[ Q_i \text{ is crossed lengthwise} \right] \leq \Prob^\eta \left[ \cross(Q) \right] \, .
\end{equation}
Now, let $\gamma'>0$ to be fixed later. Theorem~\ref{t.AGMT} implies that there exists a constant $c_0=c_0(\gamma') \in (0,1)$ such that
\[
\forall i, \, \Pro \left[ \Prob^\eta \left[ Q_i \text{ is crossed lengthwise} \right] \geq c_0 \right] \geq 1-(\delta \, \diam(D))^{-\gamma'} \, .
\]
By a union bound we obtain that
\begin{eqnarray*}
\Pro \left[ \forall i, \, \Prob^\eta \left[ Q_i \text{ is crossed lengthwise} \right] \geq c_0 \right] & \geq & 1-\grandO{1} N(D,\delta) \, (\delta \, \diam(D))^{-\gamma'}\\
& \geq & 1-\grandO{1} \delta^{-2} \, (\delta \, \diam(D))^{-\gamma'} \, .
\end{eqnarray*}
Together with~\eqref{e.GL2_1}, this implies that
\[
\Pro \left[ \forall Q \in \mathcal{Q}_\delta(D), \, \Prob^\eta \left[ \cross(Q) \right] \geq c_0^{N(D,\delta)} \right] \geq 1-\grandO{1} \delta^{-2} \, (\delta \, \diam(D))^{-\gamma'} \, .
\]
We now use the fact that $\diam(D) \geq \delta^{-2}/100$ and we choose $\gamma'=2+4\gamma$. We have
\begin{eqnarray*}
\delta^{-2} \, (\delta \, \diam(D))^{-\gamma'} & = & \delta^{-2-\gamma'} \,  \diam(D)^{-\gamma'}\\
& \leq & (100)^{1+\gamma'/2} \diam(D)^{1-\gamma'/2} = (100)^{2(1+\gamma)} \diam(D)^{-2\gamma} \, .
\end{eqnarray*}
This ends the result if $\diam(D)$ is sufficiently large (e.g. $\diam(D) \geq (100)^2$) and if $c=c_0^{\sup_D N(D,\delta)}(=c_0^{\grandO{1}\delta^{-2}})$. If $\diam(D) \leq (100)^2$ then the proof is easy.
\end{proof}


In Section~\ref{s.quasi}, we will work with the following family of quads.

\begin{defi}\label{d.a_lot_of_quads_bis}
Let $\widetilde{\mathcal{Q}}'_\delta(D)$ be the set of all quads $Q \subseteq D$ such that there exists $k \in \N$ such that $Q$ is  drawn on the grid $(2^k \, \delta \, \diam(D)) \cdot \Z^2$ and the length of each side of $Q$ is less than $100 \cdot 2^k \, \delta \, \diam(D)$. Also, let $\widetilde{\mathcal{Q}}_\delta(D)$ be the set of all quads $Q \subseteq D$ such that there exists a quad $Q' \in \widetilde{\mathcal{Q}}'_\delta(D)$ satisfying $\cross(Q') \subseteq \cross(Q)$.
\end{defi}

\begin{prop}\label{p.a_lot_of_quads_bis}
Let $\delta \in (0,1)$ and $\gamma \in (0,+\infty)$. There exists $\widetilde{c} = \widetilde{c}(\gamma) \in (0,1)$ such that,\footnote{The fact that $\widetilde{c}$ does not depend on $\delta$ will be crucial.} for every bounded subset of the plane $D$ satisfying $\diam(D) \geq \delta^{-2}/100$, we have
\[
\Pro \left[ \forall Q \in \widetilde{\mathcal{Q}}_{\delta}(D), \, \Prob^\eta \left[ \cross(Q) \right] \geq \widetilde{c} \right] \geq 1 - \grandO{1} \diam(D)^{-\gamma} \, ,
\]
where the constants in $\grandO{1}$ are absolute constants.
\end{prop}
\begin{rem}
One could use Proposition~\ref{p.a_lot_of_quads_bis} and gluing arguments to prove Proposition~\ref{p.a_lot_of_quads} (with $c(\delta,\gamma)=\widetilde{c}(\gamma)^{\grandO{1}\delta^{-2}})$ but since we will essentially use Proposition~\ref{p.a_lot_of_quads} we have chosen to write the proof of this proposition and then mimic it in order to obtain Proposition~\ref{p.a_lot_of_quads_bis}.
\end{rem}
\begin{proof}[Proof of Proposition~\ref{p.a_lot_of_quads_bis}]
First, we work with the following set of quads: Let $\widehat{\mathcal{Q}}_\delta(D) \subseteq \mathcal{Q}_\delta(D)$ be the set of all quads $Q \subseteq D$ drawn on the grid $(\delta \, \diam(D)) \cdot \Z^2$ such that the length of each side of $Q$ is less than $100 \cdot \delta \, \diam(D)$. We have
\[
\widetilde{\mathcal{Q}}_\delta(D) = \bigcup_{k=0}^{+\infty} \widehat{\mathcal{Q}}_{2^k\delta}(D) \, .
\]
By following the proof of Proposition~\ref{p.a_lot_of_quads} we obtain that there exists $c_0=c_0(\gamma) \in (0,1)$ such that
\begin{equation}\label{e.GL2_2}
\Pro \left[ \forall Q \in \widehat{\mathcal{Q}}_{\delta}(D), \, \Prob^\eta \left[ \cross(Q) \right] \geq c_0^{\grandO{1}} \right] \geq 1 - \grandO{1} \diam(D)^{-\gamma} \, . 
\end{equation}
The fact that we have $c_0^{\grandO{1}}$ instead of $c_0^{\grandO{1}\delta^{-2}}$ comes from the fact that we only consider quads of side length $\leq 100 \cdot \delta \, \diam(D)$. Now, note that the sets $\widehat{\mathcal{Q}}_{2^k\delta}(D)$ are empty when $k > -\log_2(\delta)$, hence
\[
\widetilde{\mathcal{Q}}_\delta(D) = \bigcup_{k=0}^{-\log_2(\delta)} \widehat{\mathcal{Q}}_{2^k\delta}(D) \, .
\]
Note also that $-\log_2(\delta) \leq \grandO{1} \log_2 (\diam(D)) $ since $\diam(D) \geq \delta^{-2}/100$. Let us apply~\eqref{e.GL2_2} to $\widehat{\mathcal{Q}}_{2^k\delta}(D)$ for every $k \in \{ 0, \cdots, -\log_2(\delta) \}$ and with $\gamma+1$ instead of $\gamma$. A union bound implies that there exists a constant $\widetilde{c} = \widetilde{c}(\gamma) > 0$ such that
\begin{eqnarray*}
\Pro \left[ \forall Q \in \widetilde{\mathcal{Q}}_{\delta}(D), \, \Prob^\eta \left[ \cross(Q) \right] \geq \widetilde{c} \right] & \geq & 1 - \grandO{1} \log_2 (\diam(D)) \, \diam(D)^{-(\gamma+1)}\\
& \geq & 1 - \grandO{1} \diam(D)^{-\gamma} \, .
\end{eqnarray*}
\end{proof}

\subsection{``Strong'' well-separateness of interfaces}\label{ss.preliminary}

In this subsection, we prove Lemma~\ref{l.interfaces} i.e. we prove that the interfaces are well separated with high probability. Subsections~\ref{ss.warm} and~\ref{ss.a_lot_of_quads} do not depend on the other subsections of the paper but this is not the case of the present subsection. Indeed, we are going to rely on the results of Subsections~\ref{ss.quasi_even},~\ref{ss.QM_consequences} and~\ref{ss.half-plane} where the quasi-multiplicativity property is proved in the case of an even number of arms and in the case of arm events in the half-plane, and where the exponent of the $3$-arm event in the half-plane is computed.

\begin{rem}
In Subsection~\ref{ss.quasi_even} we will prove another ``well-separateness of interfaces lemma'': Lemma~\ref{l.good}; but the notion of well-separateness of Lemma~\ref{l.good} is weaker than the one in Lemma~\ref{l.interfaces}. Lemma~\ref{l.good} is actually enough to deal with an even number of arms or with arm events restricted to a wedge, but is not enough to deal with an odd number of arms.
\end{rem}

\begin{proof}[Proof of Lemma~\ref{l.interfaces}]
First, note that there exist $\asymp \delta^{-1}$ points $y \in \partial B_R$ such that, if the event $\gi_\delta^{ext}(R)$ does not hold, then there is a $3$-arm event in one of the sets $A(y;20 \delta R,R/4) \cap B_R$. Note also that, if $y \in \partial B_R$, then $A(y;20 \delta R,R/4) \cap B_R$ is included in a half-plane whose boundary contains $y$. Together with Item~ii) of Proposition~\ref{p.universal}, this implies that
\[
\Pro \left[ \gi_\delta^{ext}(R) \right] \geq 1-\grandO{1} \delta^{-1} \, \left( \frac{\delta \, R}{R} \right)^2 = 1-\grandO{1} \delta \, .
\]

Now, let us study $\gi^{int}_\delta(r)$. As above, there exist $\asymp \delta^{-1}$ points $y \in \partial B_r$ such that, if the event $\gi_\delta^{int}(r)$ does not hold, then there is a $3$-arm event in one of the sets $A(y;20 \delta r,r/2) \setminus B_r$. However, it is not true that, for every $y \in \partial B_r$, $A(y;20 \delta r,r/2) \setminus B_r$ is included in a half-plane whose boundary contains $y$ (there are problems at the corners of $B_r$). This is why we need the following result:
\begin{claim}
Let $y \in \partial B_r$ and let $\rho$ be such that $y$ is at distance at least $\rho$ from the corners of $B_r$. Assume that $\rho  \in [20 \delta r,r/2]$. Then, there exists a constant $\epsilon > 0$ such that
\[
\Pro \left[ 3\text{-arm event in } A(y;20 \delta r,r/2) \setminus B_r \right] \leq \grandO{1} \left( \frac{\delta \, r}{\rho} \right)^2 \, \left( \frac{\rho}{r} \right)^\epsilon \, .
\]
\end{claim}
\begin{proof}
Note that $A(y;20 \delta r,\rho) \setminus B_r$ is included in a half-plane whose boundary contains $y$. Write $\arm_3^+(y;20 \delta r,\rho)$ for the $3$-arm event in $A(y;20 \delta r,\rho) \setminus B_r$ and let
\[
\widehat{\arm}_3^+(y;20 \delta r,\rho) = \left\lbrace \Pro \left[ \arm_3^+(y;20\delta r,\rho) \cond \omega \cap A(y;20 \delta r,\rho) \right] > 0 \right\rbrace \, .
\]
Let $y_0$ be the corner of $B_r$ closest to $y$, let $\arm_3(y_0;2\rho,r/2)$ be the $3$-arm event $\arm_3(2\rho,r/2)$ translated by $y_0$, and let
\[
\widehat{\arm}_3(y_0;2\rho,r/2) = \left\lbrace \Pro \left[ \arm_3(y_0;2\rho,r/2) \cond \omega \cap A(y_0;2\rho,r/2) \right] > 0 \right\rbrace  \, .
\]
The events $\widehat{\arm}_3^+(y;20 \delta r,\rho)$ and $\widehat{\arm}_3(y_0;2\rho,r/2)$ are independent. Moreover, if the $3$-arm event in $A(y;20 \delta r,r/2) \setminus B_r$ holds then both these events hold. Remember that in the present subsection we rely on the results of Subsections~\ref{ss.quasi_even},~\ref{ss.QM_consequences} and~\ref{ss.half-plane} where the quasi-multiplicativity property and its consequences (e.g. Proposition~\ref{p.fandalpha}) are proved for $j$ odd and also for arm events in the half plane for any $j$. We apply Proposition~\ref{p.fandalpha} to the $2$-arm event in the whole plane and to the $3$-arm event in the half-plane. We obtain that
\[
\Pro \left[ \widehat{\arm}_3(y_0;2\rho,r/2) \right] \leq \Pro \left[ \widehat{\arm}_2(y_0;2\rho,r/2) \right] \asymp \alpha^{an}_2(2\rho,r/2)
\]
and
\[
\Pro \left[ \widehat{\arm}_3^+(y;20 \delta r,\rho) \right] \asymp \alpha^{an,+}_3(20 \delta r,\rho) \, .
\]
If we combine these estimates with~\eqref{e.poly} and with the computation of the $3$-arm event in the half-plane (Item~(ii) of Proposition~\ref{p.universal}), we obtain that $\Pro \left[ \widehat{\arm}_3(y_0;2\rho,r/2) \right] \leq \grandO{1} \left( \frac{\rho}{r} \right)^{\Omega(1)}$ and $\Pro \left[ \widehat{\arm}_3^+(y;20 \delta \, r,\rho) \right] \asymp \left( \frac{\delta r}{\rho} \right)^2$. Finally,
\begin{eqnarray*}
\Pro \left[ 3\text{-arm event in } A(y;20 \, \delta \, r,r/2) \setminus B_r \right] & \leq & \Pro \left[ \widehat{\arm}_3^+(y;20  \delta r,\rho) \cap \widehat{\arm}_3(y_0;2\rho,r/2) \right]\\
& \leq & \grandO{1} \, \left( \frac{\delta \, r}{\rho} \right)^2 \, \left( \frac{\rho}{r} \right)^{\Omega(1)} \, ,
\end{eqnarray*}
which ends the proof.
\end{proof}

We can (and do) assume that the constant $\epsilon$ of the claim is in $(0,1)$. Now, note that there exist $N(\delta) \asymp \log_2 ( \delta^{-1} )$ finite subsets of $\partial B_r$: $Y_1, \cdots, Y_{N(\delta)}$ such that: (a) $| Y_i | \asymp 2^i$, (b) for every $y \in Y_i$, there exists a corner of $B_r$ at distance $\asymp 2^i \delta r$ from $y$ and (c) if $\gi^{int}_\delta(r)$ does not hold, then there exists $y \in \cup_{i=1}^{N(\delta)} Y_i$ such that the $3$-arm event in $A(y;20 \delta r,r/2) \setminus B_r$ holds. Combined with the claim, this observation implies that
\begin{eqnarray*}
\Pro \left[ \gi^{int}_\delta(r) \right] & \leq & \grandO{1} \, \sum_{i=1}^{N(\delta)} 2^i \left( \frac{\delta \, r}{2^i \delta r} \right)^2 \,\left( \frac{2^i \delta r}{r} \right)^\epsilon\\
& \leq & \grandO{1} \, \delta^\epsilon \sum_{i=1}^{N(\delta)} 2^{i(\epsilon-1)}\\
& = & \grandO{1} \, \delta^\epsilon \, .
\end{eqnarray*}
\end{proof}

\section{Pivotal events and some estimates on arm events}\label{s.pivotals}

In this section, we only work at the parameter $p=1/2$, hence we intentionally forget the subscript $p$ in the notations. We will rely on the quasi-multiplicativity property (proved in Section~\ref{s.quasi}), on its consequences Propositions~\ref{p.fandalpha} and~\ref{p.techniquegeneral}, and on the preliminary results from Section~\ref{s.first}. Our goal is to estimate pivotal events. We refer to Subsection~\ref{ss.intro_piv} for the notations we use for these events. Our main goal is to prove the following result:
\begin{prop}\label{p.asymp_R2alpha4}
Let $\rho \geq 1$ and $R \geq 100\rho$. We have\footnote{The constant $2$ in ``$2\rho \Z^2$'' does not have to be taken seriously. The reason why we look at grids of mesh $\geq 2$ is only that we have stated the quasi-multiplicativity property Proposition~\ref{p.quasi} for $1 \leq r_1 \leq r_2 \leq r_3$ i.e. for arm events around boxes of side length at least $2$.}
\[
\sum_{S \text{ square of the grid } (2 \rho \Z)^2} \Pro_{1/2} \left[ \Piv_S(\cross(2R,R)) \right] \asymp \left( \frac{R}{\rho} \right)^2 \alpha^{an}_{4,1/2}(\rho,R) \, .
\]
\end{prop}
The event $\Piv_S(\cross(2R,R))$ is an annealed-pivotal event. We will also prove similar bounds for quenched-pivotal events, see Lemma~\ref{l.piv_one_point}. Let us make two observations in order to illustrate the difficulties that will arise in the proof of Proposition~\ref{p.asymp_R2alpha4}.
Let $S$ be a square of the grid $2\rho \, \Z^2$.
\bi 
\item[i)] Even if $S$ is far-away from $[-2R,2R] \times [-R,R]$, we have $\Pro_{1/2} \left[ \Piv_S(\cross(2R,R)) \right] > 0$.
\item[ii)] Assume that $S \subseteq [-2R,2R] \times [-R,R]$ and let  $\arm_4^\square(S,R)$ denote the event that there are two black arms included in $[-2R,2R] \times [-R,R] \setminus S$ from $\partial S$ to the left and right sides of $[-2R,2R] \times [-R,R]$ and two white arms included in $[-2R,2R] \times [-R,R] \setminus S$ from $\partial S$ to the top and bottom sides of $[-2R,2R] \times [-R,R]$. The events $\Piv_S ( \cross(2R,R) )$ and $\arm_4^\square(S,R)$ are closely related. However, we do not have $\arm_4^\square(S,R) = \Piv_S( \cross(2R,R) )$ (contrary to Bernoulli percolation on $\Z^2$).
\ei
\begin{rem}
Proposition~\ref{p.asymp_R2alpha4} is stated for the crossing events $\cross(2R,R)$ since we will apply this result to $2R \times R$ rectangles, but of course the proof works for any shape of rectangle.
\end{rem}




\subsection{The case of the bulk}\label{ss.pivotals1}

Let $1 \leq \rho \leq R/10 \leq R$, let $y$ be a point of the plane and let $S=B_\rho(y)$ be the square of size length $2\rho$ centered at $y$. In this subsection, we use the quasi-multiplicativity property and its consequences to estimate the probability of $\Piv_S \left( \cross(2R,R) \right)$ when $S$ is ``in the bulk''. We start with the following lemma:

\begin{lem}\label{l.warm-up_piv}
Let $\rho$, $R$ and $S$ be as above, and assume that $S$ is at distance at least $R/3$ from the sides of the rectangle $[-2R,2R] \times [-R,R]$. Also, let $\arm_4^\square(S,R)$ be the event that there are two black arms in $[-2R,2R] \times [-R,R] \setminus S$ from $\partial S$ to the left and right sides of $[-2R,2R] \times [-R,R]$ and two white arms  in $[-2R,2R] \times [-R,R] \setminus S$ from $\partial S$ to the top and bottom sides of $[-2R,2R] \times [-R,R]$. Then
\[
\Pro \left[ \arm_4^\square(S,R) \right] \asymp \alpha^{an}_4(\rho,R) \, ,
\]
where the constants in $\asymp$ are absolute constants.
\end{lem}



\begin{proof} The proof of the inequality $\Pro \left[ \arm_4^\square(S,R) \right] \leq \grandO{1} \alpha^{an}_4(\rho,R)$ is a direct consequence of the quasi-multiplicativity property (and of~\eqref{e.poly}). Let us prove the other inequality. We write the proof only for $y=0$ (i.e. for $S = B_\rho$) since the proof for other values of $y$ is the same. Note that it is sufficient to prove the result for $R$ sufficiently large. Let $\delta \in (0,1)$ to be determined later and assume that $R \geq \delta^{-2}$. Consider the following events (see Definition~\ref{d.dense} and Proposition~\ref{p.a_lot_of_quads}):
\[
\dense_\delta(R) := \dense_{\delta/100} \left( A(R/4,2R) \right) \, ,
\]  
\[
\qbc_\delta(R) := \qbc^1_{\delta} \left( A(3R/8,2R) \right) \, ,
\]
and let $\gi^{ext}_\delta(R/2)$ be defined as in~Subsection~\ref{ss.formal_events}.

Note that the event $\dense_\delta(R) \cap \qbc_\delta(R) \cap \gi^{ext}_\delta(R/2)$ is measurable with respect to $\omega \setminus B_{R/4}$. With exactly the same proof as Lemma~\ref{l.dense}, we obtain that $\Pro \left[ \dense_\delta(R) \right] \geq 1 - \grandO{1} \delta^{-2} \exp(-\Omega(1)\left( \delta \cdot R)^2 \right) \geq 1-\grandO{1} \exp \left( -\Omega(1) \delta^{-2} \right)$ (since $R \geq \delta^{-2}$). Moreover, Proposition~\ref{p.a_lot_of_quads} implies that $\Pro \left[ \qbc_\delta(R) \right] \geq 1-\grandO{1} R^{-1} \geq 1-\grandO{1} \delta^2$ and Lemma~\ref{l.interfaces} implies that $\Pro \left[ \gi^{ext}_\delta(R/2) \right] \geq 1 - \grandO{1} \delta$. Therefore, $\Pro \left[ \dense_\delta(R) \cap \qbc_\delta(R) \cap \gi^{ext}_\delta(R/2) \right]$ can be made as close to $1$ as we want provided that we take $\delta$ sufficiently small. Hence, we can use Proposition~\ref{p.techniquegeneral} (which is the key result here) to say that, if $\delta$ is sufficiently small, then
\[
\Pro \left[ \arm_4(\rho,R/2) \cap \dense_\delta(R) \cap \qbc_\delta(R) \cap \gi^{ext}_\delta(R/2) \right] \geq \alpha^{an}_4(\rho,R/2)/2 \geq \alpha^{an}_4(\rho,R)/2 \, .
\]
See Subsection~\ref{ss.formal_events}: we have $\lbrace s^{ext}(\rho,R/2) \geq 5 \delta  R \rbrace = \lbrace s^{ext}(\rho,R/2) \geq 10 \delta  R/2 \rbrace \supseteq \gi^{ext}_\delta(R)$. Hence, we also have
\begin{equation}\label{e.looks_good_nice}
\Pro \left[ \arm_4(\rho,R/2) \cap \dense_\delta(R) \cap \qbc_\delta(R) \cap \lbrace s^{ext}(\rho,R/2) \geq 5 \delta R \rbrace \right] \geq \alpha^{an}_4(\rho,R)/2 \, .
\end{equation}
Let $\eta \in \dense_\delta(R) \cap \qbc_\delta(R)$ be such that $\Prob^\eta \left[ \arm_4(\rho,R/2) \cap \lbrace s^{ext}(\rho,R/2) \geq 5\delta R \rbrace  \right] > 0$ and
write $\beta_0, \cdots, \beta_{k-1}$ for the interfaces that cross $A(\rho,R)$ (in counter-clockwise order and such that the right-hand-side of $\beta_0$ - when going from $\partial S$ to $\partial B_R$ - is black, say). First, we work under the following conditional probability measure:
\begin{multline*}
\nu^\eta_{\rho,R,(\beta_j)_j}\\
\hspace{1.2em }:= \Prob^\eta \left[ \; \cdot \; \cond \arm_4(\rho,R/2) \cap \dense_\delta(R) \cap \qbc_\delta(R) \cap \lbrace s^{ext}(\rho,R/2) \geq 5 \delta R \rbrace, \, \beta_0, \cdots, \beta_{k-1} \right] .
\end{multline*}
Thanks to~\eqref{e.looks_good_nice}, it is sufficient to prove that there exists a constant $c=c(\delta) > 0$ such that
\[
\nu^\eta_{\rho,R,(\beta_j)_j} \left[ \arm_4^\square(S,R) \right] \geq c \, .
\]
Since $\eta \in \dense_\delta(R) \cap \lbrace s^{ext}(\rho,R/2) \geq 5 \delta R \rbrace$, we can choose four quads $Q(\beta_j)$, $j \in \{0,\cdots,3 \}$ such that
\bi 
\item[(a)] For every $j \in \{0, \cdots, 3 \}$, $Q(\beta_j) \in  \mathcal{Q}_\delta \left( A(3R/8,2R) \right)$;
\item[ (b)] For every $j \in \lbrace 0, \cdots, 3 \rbrace$, one of the distinguished sides of $Q(\beta_j)$ is included in $\beta_j$;
\item[(c)] The other distinguished side of $Q(\beta_0)$ (respectively $Q(\beta_1)$, $Q(\beta_2)$ and $Q(\beta_3)$) is included in the right (respectively top, left and bottom) side of $[-2R,2R] \times [-R,R]$;
\item[(d)] For every $j \in \{0, \cdots, 3 \}$, $Q(\beta_j) \cap B_{R/2}$ is included in the region between $\beta_j$ and $\beta_{j-1}$ (where $\beta_{-1} := \beta_{k-1}$);
\item[(e)] If $0 \leq i \neq j \leq 3$, then there is no Voronoi cell that intersects both $Q(\beta_i)$ and $Q(\beta_j)$.
\ei See Figure~\ref{f.Q(beta)_square}. Let $F$ be the event that, for every $j \in \lbrace 0, \cdots, 3 \rbrace$, $Q(\beta_j)$ is crossed (respectively dual-crossed) when $j$ is even (respectively odd). Note that conditioning on $(\beta_j)_j$ affects the percolation process as follows: if $j$ is even (respectively odd) then there is a black (respectively white) crossing from $\beta_j$ to $\beta_{j-1}$. Hence, by using the fact that $\eta \in \qbc_\delta(R)$ and by applying the (quenched) Harris-FKG inequality, we obtain that there exists $c=c(\delta) > 0$ such that
\[
\nu^\eta_{\rho,R,(\beta_j)_j} \left[ F \right] \geq c > 0 \, .
\]
This ends the proof since $F \subseteq \arm_4^\square(\rho,R)$.
\begin{figure}[!h]
\begin{center}
\includegraphics[scale=1.2]{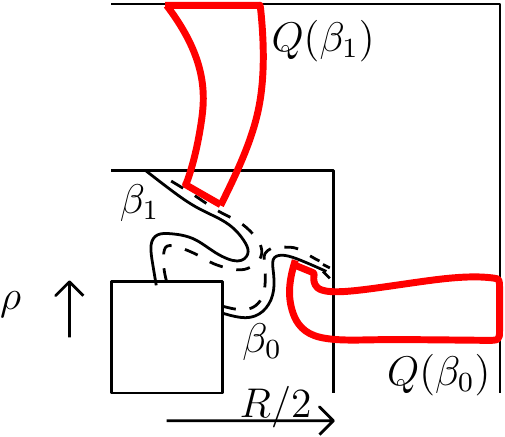}
\end{center}
\caption{The quads $Q(\beta_1)$ and $Q(\beta_2)$.}
\label{f.Q(beta)_square}
\end{figure}
\end{proof}

We have the following strengthening of Lemma~\ref{l.warm-up_piv}, in the spirit of Proposition~\ref{p.techniquegeneral}:

\begin{cor}\label{c.warm-up_piv}
There exists an absolute constant $\epsilon \in (0,1)$ such that, for every event $G$ measurable with respect to $\omega \setminus A(y;2\rho,R/6)$ that satisfies $\Pro \left[ G \right] \geq 1-\epsilon$, we have
\[
\Pro \left[ \arm_4^\square(S,R) \cap G \right] \geq \epsilon \, \alpha^{an}_4(\rho,R) \, .
\]
\end{cor}
\begin{proof}
The proof is exactly the same as the similar result Proposition~\ref{p.techniquegeneral}. More precisely, this is a direct consequence of Lemma~\ref{l.warm-up_piv} and Proposition~\ref{p.fandalpha}.
\end{proof}

Now, let us prove the following result:

\begin{lem}\label{l.piv_bulk}
Let $\rho$, $R$ and $S$ be as in Lemma~\ref{l.warm-up_piv}. Then
\[
\Pro \left[ \Piv_S(\cross(2R,R)) \right] \asymp \alpha^{an}_4(\rho,R) \, .
\]
\end{lem}

\begin{proof}
The fact that $\Pro \left[ \Piv_S(\cross(2R,R)) \right] \geq \Omega(1) \alpha^{an}_4(\rho,R)$ is a direct consequence of Lemma~\ref{l.warm-up_piv}. Indeed, (except on a zero probability set) we have\footnote{Consider a configuration for which $\arm_4^\square(S,R)$ holds. If we replace the configuration restricted to $S$ by a sufficiently dense set of black (respectively white) points then $\cross(2R,R)$ is satisfied (respectively not satisfied).}
\[
\arm_4^\square(S,R) \subseteq \Piv_S(\cross(2R,R)) \, .
\]

Now, let us prove that $\Pro \left[ \Piv_S(\cross(2R,R)) \right] \leq \grandO{1} \alpha^{an}_4(\rho,R)$. We write the proof in the case $y=0$. For every $\rho' > 0$, let $\dense(\rho'):=\dense_{1/100} \left( A(\rho',2 \rho') \right)$ (remember Definition~\ref{d.dense}). Note that we have
\[
\Piv_S(\cross(2R,R)) \subseteq \arm_4(2\rho,R) \cup \left( \Piv_S(\cross(2R,R)) \setminus \dense(\rho) \right).
\]
More generally, for all $k \in \left\lbrace 0, \cdots, \lfloor \log_2(\frac{R}{4\rho}) \rfloor =: k_0 \right\rbrace$ we have
\begin{equation}\label{e.from_piv_to_arm}
\Piv_S(\cross(2R,R)) \subseteq \arm_4(2^{k+1}\rho,R) \cup \left( \Piv_S(\cross(2R,R)) \setminus \dense (2^k\rho) \right),
\end{equation}
which implies that $\Piv_S(\cross(2R,R))$ is included in:
\begin{align}
& \arm_4 \big(2\rho,R \big) \, \bigcup \, \Big( \bigcup_{k = 0}^{k_0} \arm_4 \big(2^{k+2}\rho,R \big) \setminus \dense(2^k\rho) \Big) \, \bigcup \, \neg \dense(2^{k_0+1}\rho) \nonumber \\
& \subseteq \widehat{\arm}_4 \big(2\rho,R \big) \, \bigcup \, \Big( \bigcup_{k = 0}^{k_0} \widehat{\arm}_4 \big(2^{k+2}\rho,R \big) \setminus \dense(2^k\rho) \Big) \, \bigcup \, \neg \dense(2^{k_0+1}\rho) \, , \label{e.union_bound_piv}
\end{align}
where the events $\widehat{\arm}_4 \big( \cdot,\cdot \big)$ are the events defined in Definition~\ref{d.hatarm}. By using Proposition~\ref{p.fandalpha} and the fact that $\widehat{\arm}_4 \big(2^{k+2}\rho,R \big)$ and $\dense(2^k\rho)$ are independent, we obtain that, for each $k \in \left\lbrace  0, \cdots, k_0 \right\rbrace$),
\[
\Pro \left[ \widehat{\arm}_4 \big(2^{k+2}\rho,R \big) \setminus \dense(2^k\rho) \right] \leq \grandO{1} \alpha^{an}_4 \big( 2^{k+2}\rho,R \big) \, \Pro \left[ \neg \dense(2^k\rho) \right] \, .
\]
With the same proof as Lemma~\ref{l.dense}, we obtain that
\begin{equation}\label{e.denseforpiv}
\Pro \left[ \neg \dense(2^k\rho) \right] \leq \grandO{1} \, e^{-\Omega(1)(2^k \rho)^2} \, .
\end{equation}
Note also that the quasi-multiplicativity property and~\eqref{e.poly} imply that
\[
\alpha^{an}_4 \big( 2^{k+2}\rho,R \big) \leq \grandO{1} \, 2^{\grandO{1} k} \, \alpha^{an}_4(\rho,R) \, .
\]
Therefore,
\begin{equation}\label{e.dense2^k}
\Pro \left[ \widehat{\arm}_4 \big(2^{k+2}\rho,R \big) \setminus \dense(2^k\rho) \right] \leq \grandO{1} \, \alpha^{an}_4(\rho,R) \, 2^{\grandO{1} k} \, e^{-\Omega(1)(2^k \rho)^2} \, .
\end{equation}
Similarly, $1 \leq \grandO{1} \, \alpha^{an}_4(\rho,2^{k_0}\rho) \, 2^{\grandO{1} k_0}$, hence
\begin{eqnarray}
\Pro \left[ \neg \dense(2^{k_0+1}\rho) \right] & \leq & \grandO{1} \, \alpha^{an}_4(\rho,2^{k_0}\rho) \, 2^{\grandO{1} k_0} \, e^{-\Omega(1)(2^{k_0} \rho)^2} \nonumber\\
& \leq & \grandO{1} \alpha^{an}_4(\rho,R) \, 2^{\grandO{1} k_0} \, e^{-\Omega(1)(2^{k_0} \rho)^2} \, . \label{e.dense2^k_0}
\end{eqnarray}
Now, we can conclude by applying the union-bound to~\eqref{e.union_bound_piv} and by using the inequalities~\eqref{e.dense2^k} and~\eqref{e.dense2^k_0}.
\end{proof}

We still consider the case where $S$ is in the bulk. We end this subsection by showing another result which will be useful in the proof of the annealed scaling relations. The difference with Lemma~\ref{l.piv_bulk} is that we study \textbf{quenched} pivotal events (see Subsection~\ref{ss.intro_piv} for the definition of these pivotal events). 

\begin{lem}\label{l.piv_one_point}
Let $R$ and $S$ be as in Lemma~\ref{l.warm-up_piv}, and assume that $\rho = 1$ (i.e. $S$ is a $2 \times 2$ square). We have
\[
\Pro \left[ \left\lbrace \left| \eta \cap S \right| = 1 \right\rbrace \cap \Piv^q_S \left( \cross(2R,R) \right) \right] \geq \Omega(1) \, \alpha^{an}_4(R) \, .
\]
\end{lem}

Before proving Lemma~\ref{l.piv_one_point}, let us note that this lemma together with results from~\cite{ahlberg2015quenched} implies that $\alpha^{an}_4(R) \leq \grandO{1} R^{-(1+\epsilon)}$ for some $\epsilon > 0$, which is the first part of Proposition~\ref{p.alpha4}:
\begin{proof}[Proof of the first part of Proposition~\ref{p.alpha4}]
By~\cite{ahlberg2015quenched}, if we let $\setS_1$ be the set of all the squares of the grid $2 \Z^2$ that are included in $[-2R,2R] \times [-R,R]$ and at distance less than $R/3$ from the sides of $[-2R,2R] \times [-R,R]$, then
\begin{equation}\label{e.agmt_est}
\E \left[ \sum_{S \in \setS_1} \sum_{x \in \eta \cap S} \Prob^\eta \left[ \Piv_x^q(\cross(2R,R) \right]^2 \right] \leq \grandO{1} R^{-\Omega(1)} \, .
\end{equation}
See the end of Appendix~\ref{a.tAGMT} where we recall how the authors of~\cite{ahlberg2015quenched} have obtained this estimate. (The definition of $\setS_1$ is not the same as in Appendix~\ref{a.tAGMT} but the proof is exactly the same with the present definition.) The left-hand-side of~\eqref{e.agmt_est} is at least
\begin{multline*}
\sum_{S \in \setS_1} \E \left[ \Prob^\eta \left[ \Piv^q_S(\cross(2R,R) \right]^2 \un_{|\eta \cap S| = 1} \right]\\
\geq \sum_{S \in \setS_1} \Pro \left[ \left\lbrace \left| \eta \cap S \right| = 1 \right\rbrace \cap \Piv^q_S \left( \cross(2R,R) \right) \right]^2 \text{ (by Jensen).}
\end{multline*}
We conclude by applying Lemma~\ref{l.piv_one_point}.
\end{proof}

The difficulty in the proof of Lemma~\ref{l.piv_one_point} is that it is not obvious that, if $\arm_4(100,R)$ holds (for instance), then we can easily extend the arms until scale $1$. We overcome this difficulty by considering the event that the Voronoi tiling near $0$ ``looks like the hexagonal lattice''.

\begin{figure}[!h]
\begin{center}
\includegraphics[scale=0.52]{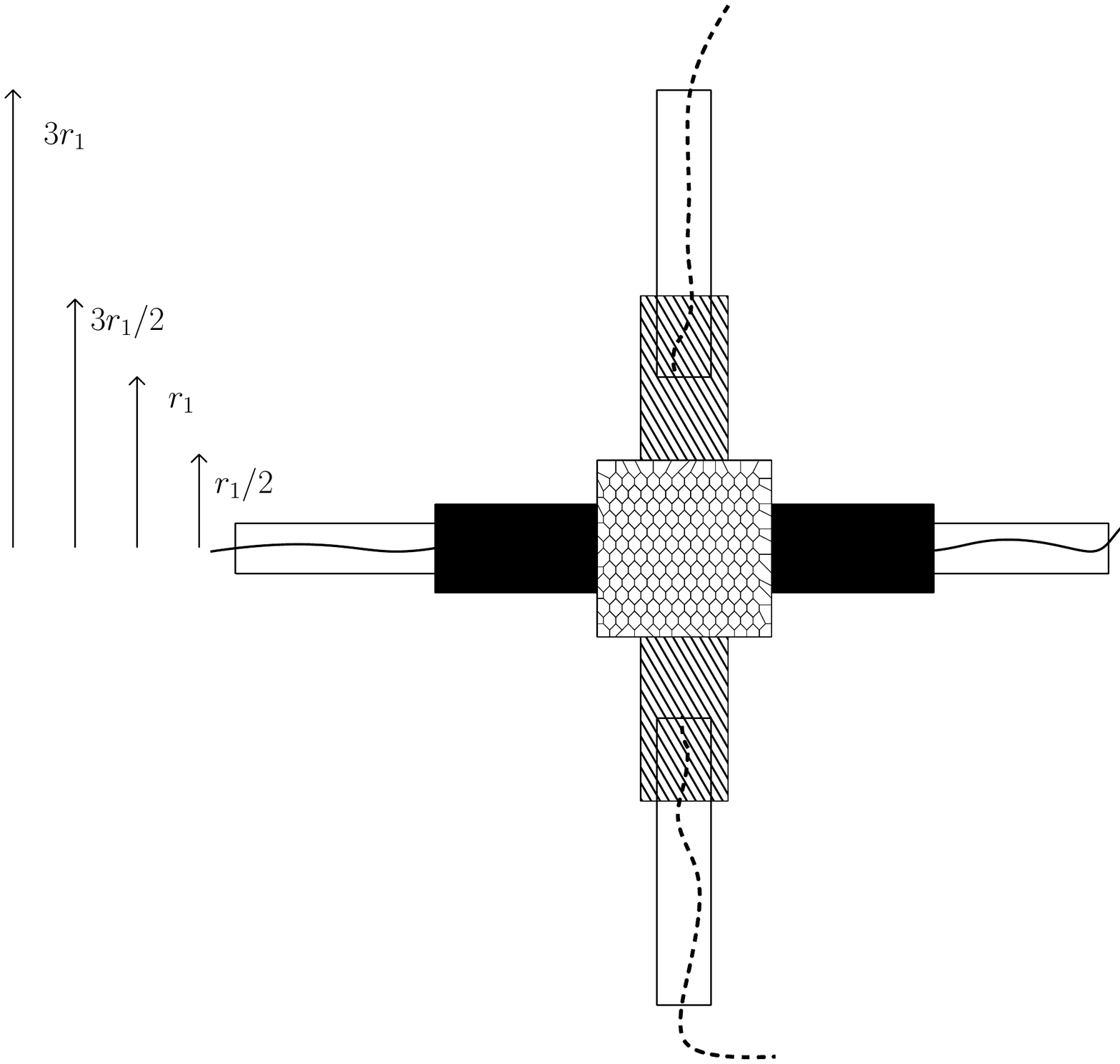}
\end{center}
\caption{The proof of Lemma~\ref{l.piv_one_point}.}\label{f.l_piv_one_pt}
\end{figure}

\begin{proof}[Proof of Lemma~\ref{l.piv_one_point}]
We write the proof in the case $y=0$ (i.e. $S=B_\rho=B_1$). The strategy is illustrated in Figure~\ref{f.l_piv_one_pt}. Note that it is sufficient to prove the result for $R$ larger than some constant. Let $r_1 \geq 1000$ to be determined later and assume that $R \geq 10r_1$.

We first need the following definition. Let $r \in [1,R/10]$. The event $\widetilde{\arm}^\square_4(B_r,R)$ is the event that i) there are two black paths $\gamma_0$ and $\gamma_2$ in $[-2R,2R] \times [-R,R] \setminus B_r$ from $\partial B_r$ to the left and right sides of $[-2R,2R] \times [-R,R]$, ii) there are two white paths $\gamma_1$ and $\gamma_3$ in $[-2R,2R] \times [-R,R] \setminus B_r$ from $\partial B_r$ to the top and bottom sides of $[-2R,2R] \times [-R,R]$, iii) we can choose the four paths such that, for every $i \in \lbrace 0, \cdots, 3 \rbrace$, $\gamma_i \cap A(r,2r) \subseteq Q_i$ where the rectangles $Q_i=Q_i(r)$ are defined in Figure~\ref{f.Qr0}.\\

\begin{figure}[!h]
\begin{center}
\includegraphics[scale=0.485]{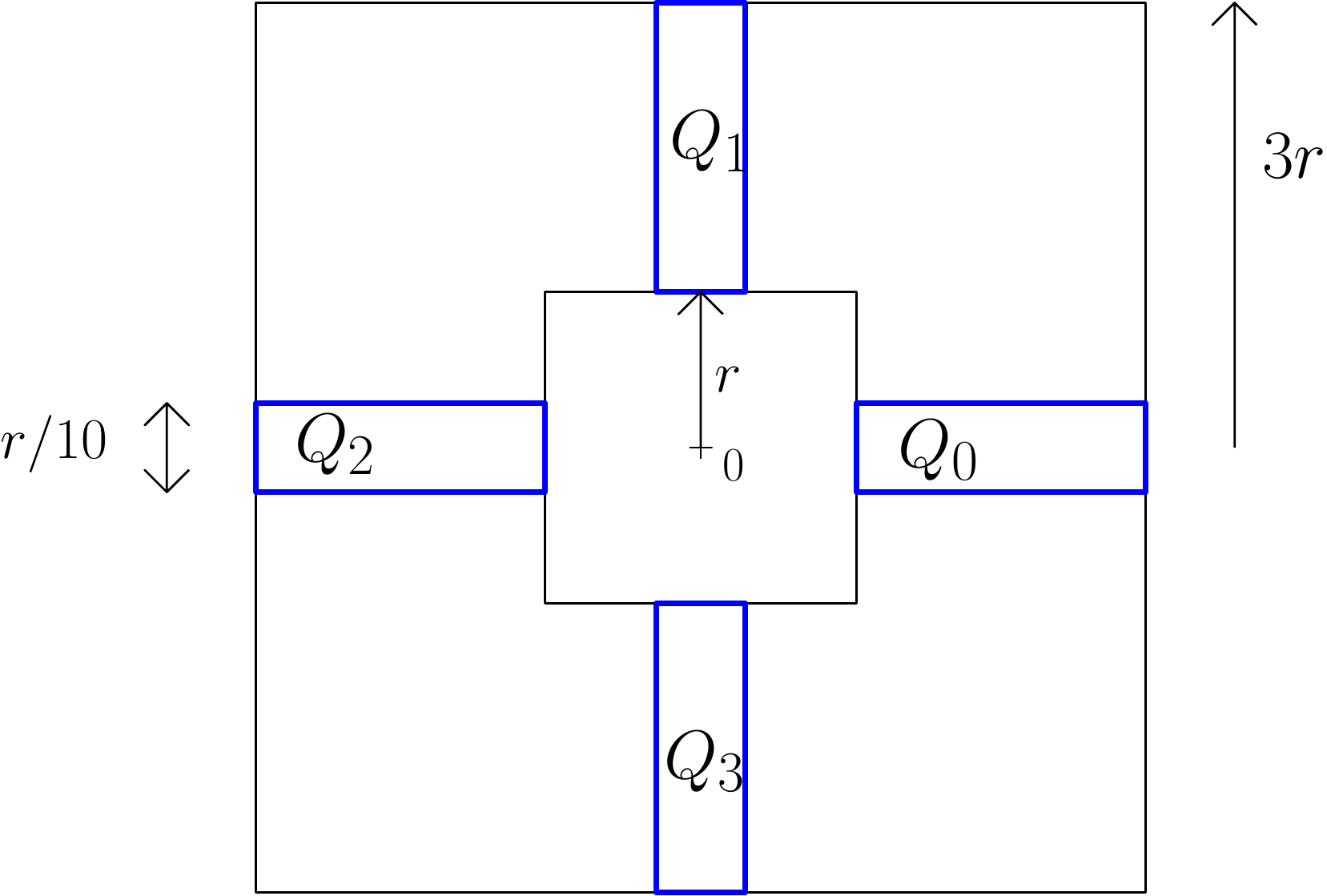}
\includegraphics[scale=0.485]{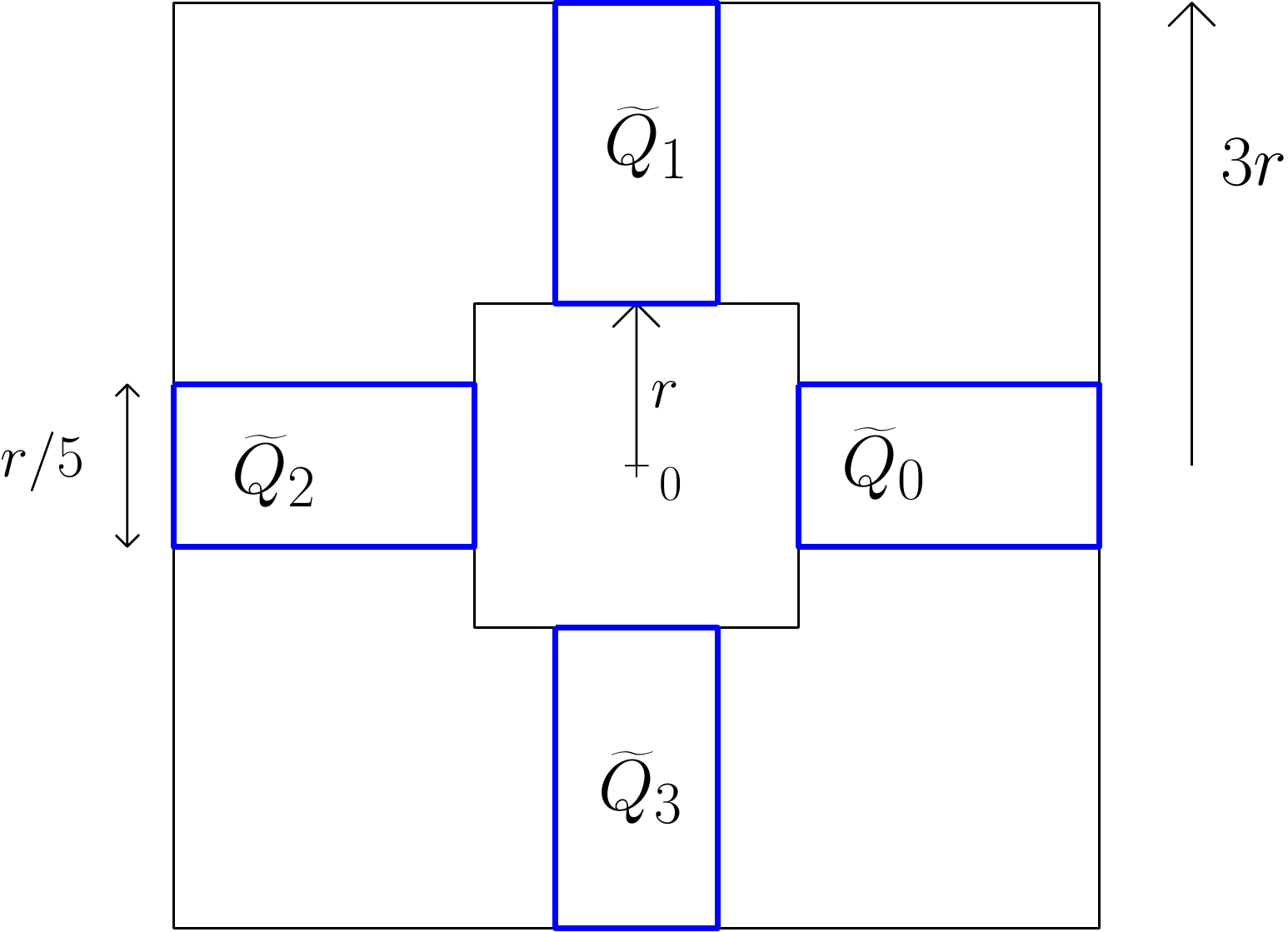}
\end{center}
\caption{The rectangles $Q_i=Q_i(r)$ and the rectangles $\widetilde{Q}_i=\widetilde{Q}_i(r)$.}\label{f.Qr0}
\end{figure}

With the same proof as Lemma~\ref{l.warm-up_piv} (except that we have to work both at scale $r$ and at scale $R$ instead of working only at scale $R$) we obtain that, if $r \leq R/10$ and if $r$ is sufficiently large, then
\[
\Pro \left[ \widetilde{\arm}^\square_4(B_r,R) \right] \geq \Omega(1) \, \alpha_4^{an}(r,R) \, .
\]

As in Corollary~\ref{c.warm-up_piv}, we also have the following stronger result: There exist $r_0\geq 1$ and $\epsilon \in (0,1)$ such that, if $r \in [r_0,R/10]$ and if $G$ is an event measurable with respect to $\omega \setminus A(2r,R/6)$ that satisfies $\Pro \left[ G \right] \geq 1-\epsilon$, then
\begin{equation}\label{e.warm-up_piv}
\Pro \left[ \widetilde{\arm}^\square_4(B_r,R) \cap G \right] \geq \epsilon \, \alpha_4^{an}(r,R) \, .
\end{equation}
Now, for any $r \geq 1$ and any $N \in \N$, write $\dense^N(r)$ for the event that $\dense_{1/100}(A(r/2,2r))$ holds and that $\left| \eta \cap A(r/2,2r) \right| \leq N$. This event is a little different from the other events ``$\dense$'' that we study in this paper since this is an event that $\eta$ is sufficiently dense \textbf{but not too much}. Let $\epsilon > 0$ as above and note that there exist $r_1 \geq r_0$ and $N \in \N$ such that $\Pro \left[ \dense^N(r_1) \right] \geq 1-\epsilon$. Fix such an $r_1$ and an $N$. By the above, we have
\begin{equation}\label{e.PPPPP}
\Pro \left[ \widetilde{\arm}^\square_4(B_{r_1},R) \cap \dense^N(r_1) \right] \geq \epsilon \, \alpha_4^{an}(r_1,R) \geq \Omega(1) \alpha_4^{an}(R) \, .
\end{equation}
The event $\dense^N(r_1)$ provides sufficiently spatial independence so that, given 	a colored configuration that satisfies $\widetilde{\arm}^\square_4(B_{r_1},R) \cap \dense^N(r_1)$, one can extend the four arms until scale $1$ with probability larger than some constant independent of $R$. This can be done for instance as follows:

Let $\text{Color}(r)$ denote the event that each point of $\eta \cap \widetilde{Q}_i$ is black (respectively white) if $i$ is even (respectively odd), where the $\widetilde{Q}_i=\widetilde{Q}_i(r)$'s are the rectangles defined in Figure~\ref{f.Qr0}. Note that we have
\bi 
\item[i)] By the quenched FKG-Harris inequality\footnote{Here, one actually needs a generalized FKG-Harris inequality which is a consequence of the classical FKG-Harris inequality and can be stated as follows (see Lemma~13 of \cite{nolin2008near} for the proof): Work conditionally on $\eta$, let $A^+,B^+ \subseteq \{-1,1\}^\eta$ be two increasing events, let $A^-,B^- \subseteq \{-1,1\}^\eta$ be two decreasing events, and let $\mathcal{A}^+,\mathcal{A},\mathcal{A}^-$ be three mutually disjoint finite subsets of $\eta$ such that $A^+,A^-,B^+,B^-$ depend only on the configuration in, respectively, $\mathcal{A} \cup \mathcal{A}^+$, $\mathcal{A} \cup \mathcal{A}^-$, $\mathcal{A}^+$ and $\mathcal{A}^-$. Then, $\Prob^\eta \left[ B^+ \cap B^- \, | \, A^+ \cap A^- \right] \geq \Prob^\eta \left[ B^+ \right] \Prob^\eta \left[ B^- \right]$.} (for the first inequality) and~\eqref{e.PPPPP} (for the second inequality), we have
\begin{multline*}
\Pro \left[ \widetilde{\arm}^\square_4(B_{r_1},R) \cap \dense^N(r_1) \cap \text{Color}(r_1/2) \right]\\
= \E \left[ \Prob^\eta \left[ \widetilde{\arm}^\square_4(B_{r_1},R) \cap \text{Color}(r_1/2) \right] \un_{\dense^N(r_1) } \right]\\
\geq \frac{1}{2^N} \, \E \left[ \Prob^\eta \left[ \widetilde{\arm}^\square_4(B_{r_1},R) \right] \un_{\dense^N(r_1) } \right]\\ = \frac{1}{2^N} \Pro \left[ \widetilde{\arm}^\square_4(B_{r_1},R) \cap \dense^N(r_1) \right] \geq c_1 \alpha_4^{an}(R) \, ,
\end{multline*}
where $c_1 > 0$ is a constant that depends only on $r_1$ and $N$.
\item[ii)] The event
\[
\widetilde{\arm}^\square_4(B_{r_1},R) \cap \dense^N(r_1) \cap \text{Color}(r_1/2)
\]
is independent of $\omega \cap B_{r_1/2}$.
\ei
As a result, for any event $A$ measurable with respect to $\omega \cap B_{r_1/2}$ we have
\[
\Pro \left[ \widetilde{\arm}^\square_4(B_{r_1},R) \cap \dense^N(r_1) \cap \text{Color}(r_1/2) \cap A \right] \geq c_1 \, \Pro \left[A \right] \alpha_4^{an}(R) \, ,
\]
So it is sufficient for our purpose to find an event $A$ measurable with respect to $\omega \cap B_{r_1/2}$ such that $\Pro \left[ A \right] > 0$ and
\begin{multline}\label{e.QQQQQ}
\Pro \left[ \left\lbrace \left| \eta \cap S \right| = 1 \right\rbrace \cap \Piv^q_S \left( \cross(2R,R) \right) \right]\\
\geq \Omega(1) \Pro \left[ \widetilde{\arm}^\square_4(B_{r_1},R) \cap \dense^N(r_1) \cap \text{Color}(r_1/2) \cap A \right] \, ,
\end{multline}
where the constants in $\Omega(1)$ only depend on $r_1$. We choose $A = \text{Hex}(r_1/2)$ where $\text{Hex}(r)$ is the event (measurable with respect to $\omega \cap B_r$) that the Voronoi diagram ``looks like the hexagonal lattice'' in $B_r$. More precisely, we let $\T$ denote the triangular lattice of mesh size $2$ and we define $\text{Hex}(r)$ as the event that there exists a bijection $f \, : \, \T \cap B_r \rightarrow \eta \cap B_r$ such that $|f(y)-y|\leq 1/100$ for every $y$. On the event $\text{Hex}(r_1/2)$, we have $| \eta \cap S | =1$. It is easy to see that $\Pro \left[ \text{Hex}(r_1/2) \right] > 0$ and that, if we condition on the event $\widetilde{\arm}^\square_4(B_{r_1},R) \cap \dense^N(r_1) \cap \text{Color}(r_1/2) \cap \text{Hex}(r_1/2)$, then we can extend the four arms ``by hand'' until the Voronoi cell of $f(0)$ (where $f$ is the above bijection) with probability larger than some constant that depends only on $r_1$ (see Figure~\ref{f.l_piv_one_pt}). Hence,~\eqref{e.QQQQQ} holds and we are done.
\end{proof}

\subsection{An estimate on the $4$-arm event}\label{ss.alpha4}

Thanks to Proposition~\ref{p.universal} (whose proof is written in Section~\ref{s.quasi}), we have the following: Let $1 \leq r \leq R$, then
\begin{equation}\label{e.alpha4easy}
\alpha^{an,+}_3(r,R) \asymp \left( \frac{r}{R} \right)^2 \asymp \alpha^{an}_5(r,R) \leq \alpha^{an}_4(r,R) \, .
\end{equation}
We now prove that $\alpha_4^{an}(r,R) \geq \Omega(1) (r/R)^{2-\epsilon}$ for some $\epsilon > 0$ (which strengthens the above inequality) i.e. we prove the second part of Proposition~\ref{p.alpha4}. In the case of percolation on $\Z^2$ or on the triangular lattice, the analogue of this proposition is a direct consequence of Reimer's inequality (\cite{reimer2000proof}). In the context of Voronoi percolation, it seems a priori natural to try to prove the following annealed Reimer's inequality: Let $A$ and $B$ be two events measurable with respect to $\omega$ restricted to a bounded domain, and define the disjoint occurrence of $A$ and $B$ as in~\eqref{e.disjoint_occ}; then, $\Pro \left[A \square B \right] \leq \Pro \left[ A \right] \Pro \left[ B \right]$. Unfortunately, this inequality is not true in general since it is not true as soon as $A=B$, $A$ depends only on $\eta$, and $\Pro \left[ A \right] \in ]0,1[$ (indeed, if $A$ and $B$ depend only on $\eta$, then $A \square B = A \cap B$).



\begin{proof}[Proof of the second part of Proposition~\ref{p.alpha4}]
Let $M \in [100,+\infty)$ to be determined later. We are inspired by the proof (by Beffara) of Proposition~A$.1$ of~\cite{garban2010fourier}. For any $\rho \geq M$, let $\dense(\rho) := \dense_{1/100} (A(\rho,2\rho))$ (remember Definition~\ref{d.dense}). Also, let $\text{Circ}(r_1,r_2)$ denote the event that there is a black circuit (i.e. an injective continuous function from $\R/\Z$ to the black region) in $A(r_1,r_2)$ surrounding the origin and, for any $c \in (0,1)$, let
\[
\qac_c(\rho) = \lbrace \Prob^\eta \left[ \text{Circ}(\rho,2\rho) \right] \geq c \rbrace
\]
(for ``Quenched Annulus Circuit'').
Theorem~\ref{t.AGMT} (applied for instance to four rectangles that surround the origin) and the (quenched) Harris-FKG inequality imply that there exists a constant $c \in (0,1)$ such that, for every $\rho$,
\begin{equation}\label{e.estimateQAC}
\Pro \left[ \qac_c(\rho) \right] \geq 1-\rho^{-3} \, .
\end{equation}
Fix such a constant $c$. Now, let $\gp(\rho,M)$ (for ``Good Point configuration'') be the following event
\[
\bigcap_{k=0}^{\lfloor \log_5 \left( M \right) \rfloor - 1} \dense(5^k\rho) \cap \qac_c(5^k\rho) \, .
\]
If we use (a direct analogue of) Lemma~\ref{l.dense} and~\eqref{e.estimateQAC}, we obtain that
\[
\Pro \left[ \gp(\rho,M) \right] \geq 1-\sum_{k=0}^{\lfloor \log_5 \left( M \right) \rfloor - 1} \left( \grandO{1} e^{-\Omega(1)(5^k \rho)^2} + (5^k \rho)^{-3} \right) \geq 1-\grandO{1} \rho^{-3} \, .
\]
Now, let $\eta \in \gp(\rho,M)$ be such that $\Prob^\eta \left[ \arm_5(\rho,M\rho) \right] > 0$. Also, let $\beta_0, \beta_1, \beta_2$ be three simple paths drawn in the Voronoi grid, included in $A(\rho,M\rho)$, that go from $\partial B_\rho$ to $\partial B_{M\rho}$, and that can arise as three consecutive interfaces. Write $S_{\beta_0,\beta_1,\beta_2}$ for the region between $\beta_0$ and $\beta_2$ that does not contain $\beta_1$. Write $\mathcal{A}_{\beta_0,\beta_1,\beta_2}$ for the event that $\beta_0,\beta_1,\beta_2$ are indeed consecutive interfaces, and write $\mathcal{B}_{\beta_0,\beta_1,\beta_2}$ for the event that $\mathcal{A}_{\beta_0, \beta_1, \beta_2}$ holds and that there is an additional (i.e. disjoint from the Voronoi cells adjacent to $\beta_0 \cup \beta_1 \cup \beta_2$) black path in $S_{\beta_0,\beta_1,\beta_2}$. Observe that, since $\eta \in \dense(2^k \rho)$, the $\lfloor \log_5(M) \rfloor$ events
\[
\lbrace \exists \text{ a black path in } S_{\beta_0,\beta_1,\beta_2} \cap A(5^k\rho, 2 \cdot 5^k \rho) \text{ from a cell adjacent to } \beta_0 \text{ to a cell adjacent to } \beta_2 \rbrace \, ,
\]
for $k = 0, \cdots, \lfloor \log_5(M) \rfloor - 1$, are independent under $\Prob^\eta$. Therefore,
\[
\Prob^\eta \left[ \mathcal{B}_{\beta_0,\beta_1,\beta_2} \cond \mathcal{A}_{\beta_0,\beta_1,\beta_2} \right] \leq (1-c)^{\lfloor \log_5 \left( M \right) \rfloor} \, .
\]
Since $\arm_5(\rho,M\rho)$ is the union over every possible $\beta_0,\beta_1,\beta_2$ of $\mathcal{B}_{\beta_0,\beta_1,\beta_2}$ , we have
\[
\Prob^\eta \left[ \arm_5(\rho,M\rho) \right] \leq (1-c)^{\lfloor \log_5 \left( M \right) \rfloor} \, \Ex^\eta \left[ Y^3 \un_{Y \geq 4} \right] \, ,
\]
where $Y=Y(\rho,M)$ is the number of interfaces from $\partial B_{\rho}$ to $\partial B_{M \rho}$. By taking the expectation, we obtain that
\begin{eqnarray*}
\alpha^{an}_5(\rho,M\rho) & \leq & (1-c)^{\lfloor \log_5 \left( M\right) \rfloor} \, \E \left[ Y^3 \un_{Y \geq 4} \right] + \Pro \left[ \neg \gp(\rho,M) \right]\\
& \leq & (1-c)^{\lfloor \log_5 \left( M \right) \rfloor} \, \E \left[ Y^3 \un_{Y \geq 4} \right] + \grandO{1}\rho^{-3} \, .
\end{eqnarray*}
Now, we use the annealed BK inequality Proposition~\ref{p.BK}. Since $\arm_{2j}(\rho, M \rho )$ is included in the $j$-disjoint occurrence of $\arm_1(\rho,M \rho)$, we have
\[
\Pro \left[ Y \geq 2j \right] = \alpha^{an}_{2j}(\rho, M \rho) \leq \alpha^{an}_1(\rho,M \rho)^j \, .
\]
The above together with~\eqref{e.poly} imply that $\Pro \left[ Y \geq j \right] \leq \grandO{1} M^{-j a}$ for some $a > 0$. Moreover, $\Pro  \left[ Y \geq 4 \right] = \alpha^{an}_4(\rho,M\rho) \geq \Omega(1) M^{-b}$ for some $b < +\infty$. Hence,
\[
\E \left[ Y^3 \un_{Y \geq 4} \right] \leq \grandO{1} \alpha^{an}_4(\rho,M\rho) \, .
\]
Therefore,
\[
\alpha^{an}_5(\rho,M\rho) \leq \grandO{1} \, (1-c)^{\lfloor \log_5 \left( M \right) \rfloor} \, \alpha^{an}_4(\rho,M\rho) + \grandO{1}\rho^{-3} \, .
\]
Remember that $\rho \geq M$. By using the fact that the exponent of the $5$-arm event is $2$ (see Proposition~\ref{p.universal}), we obtain that, if $M$ is sufficiently large, then
\[
\alpha^{an}_5(\rho,M\rho) - \grandO{1}\rho^{-3} \geq \Omega(1) \, M^{-2} - \grandO{1} \, M^{-3} \geq \Omega(1) \, M^{-2} \, .
\]
Hence, if $M$ is sufficiently large then for every $\rho \geq M$ we have
\begin{equation}\label{e.alpha4beforeQM}
M^{-2} \leq \grandO{1} \, (1-c)^{\lfloor \log_5 \left( M \right) \rfloor} \, \alpha^{an}_4(\rho,M\rho) \leq M^{-2\epsilon} \, \alpha^{an}_4(\rho,M\rho) \, ,
\end{equation}
for some $\epsilon > 0$.

Let us end the proof. Let $C=C(j=4)$ be the constant that appears in the statement of the quasi-multiplicativity property Proposition~\ref{p.quasi} and fix $M \geq 100$ sufficiently large so that~\eqref{e.alpha4beforeQM} holds and so that $M^{\epsilon} \geq C$. First, note that it is sufficient to prove the result for quantities of the form $\alpha^{an}_4(M^i,M^j)$, where $j \geq i$ are positive integers. Next, note that the quasi-multiplicativity property implies that
\[
\alpha^{an}_4(M^i,M^j) \geq C^{-(j-i)} \, \prod_{k=i}^{j-1} \alpha^{an}_4(M^k,M^{k+1}) \, .
\]
If we use~\eqref{e.alpha4beforeQM}, we obtain that
\[
\alpha^{an}_4(M^i,M^j) \geq C^{-(j-i)} \, M^{(-2+2\epsilon)(j-i)} \, ,
\]
which is at least $M^{(-2+\epsilon)(j-i)}$ since $M^{\epsilon} \geq C$. This ends the proof.
\end{proof}

\subsection{Pivotal events for crossing events and arm events}\label{ss.pivotals2}

In this subsection, we prove Proposition~\ref{p.asymp_R2alpha4}. Note that, if in this proposition we had summed only on the squares $S$ in the ``bulk'' of the rectangle $[-2R,2R] \times [-R,R]$, it would have been a direct consequence of Lemma~\ref{l.piv_bulk}. We now have to deal with all the other squares. This is essentially technical so the reader can skip this whole subsection in a first reading and only keep in mind that we also prove the following analogue of Proposition~\ref{p.asymp_R2alpha4} for arm events:

\begin{prop}\label{p.arm_event_asymp_R2alpha4}
Let $\rho \geq 1$ and let $R \geq 100 \rho$. Also, let $j \in \N^*$. Then,
\[
\sum_{S \text{ square of the grid } (2 \rho \Z)^2} \Pro \left[ \Piv_S ( \arm_j(1,R) ) \right] \leq \grandO{1} \alpha^{an}_j(\rho,R) + \grandO{1}  \alpha^{an}_j(R) \left( \frac{R}{\rho} \right)^2 \alpha^{an}_4(\rho,R) \, ,
\]
where the constants in the $\grandO{1}$'s may only depend on $j$. Note that, if $\rho=1$ (and since $R^2 \alpha^{an}_4(R) \geq \Omega(1)$ by~\eqref{e.alpha4easy}), we have the following simpler formula:
\[
\sum_{S \text{ square of the grid } (2 \rho \Z)^2} \Pro \left[ \Piv_S ( \arm_j(1,R) ) \right] \leq \grandO{1} \alpha^{an}_j(R) \, R^2 \alpha^{an}_4(R) \, .
\]
\end{prop}

In order to prove Proposition~\ref{p.asymp_R2alpha4}, we first pursue the analysis of Subsection~\ref{ss.pivotals1}. To deal with the spatial dependencies of the model, we first need to introduce the notation $\Piv_D^E(A)$ which in words denotes the event that, conditionally on the colored configuration in the set $E$, the probability that the set $D$ is annealed-pivotal for $A$ is positive. We introduce this quantity since it is measurable with respect to $E$. We will often let $E$ be an annulus which surrounds the set $D$. Let $D$ be a bounded Borel set, let $A$ be an event measurable with respect to the colored configuration $\omega$ and let $E$ be a Borel set. We write
\[
\Piv_D^E(A) := \left\lbrace \Pro \left[ \Piv_D(A) \cond \omega \cap E \right] > 0 \right\rbrace \, .
\]
Let $1 \leq \rho \leq R/10 \leq R$, let $y$ be a point of the plane and let $S=B_\rho(y)$ be the square of side length $2\rho$ centered at $y$. 

\begin{lem}\label{l.pivhat}
Let $y$, $\rho$, $R$ and $S=B_{\rho}(y)$ be as above. Let $\rho_1 \in [\rho,+\infty)$ and $\rho_2 \in [\rho_1,+\infty)$ and assume that $S$ is included in the bounded connected component of $A(y;\rho_1,\rho_2)^c$ and that $A(y;\rho_1,\rho_2) \subseteq [-2R,2R] \times [-R,R]$ (in particular, $y \in [-2R,2R] \times [-R,R]$). Then,
\[
\Pro \left[ \Piv_S^{A(y;\rho_1,\rho_2)}\left( \cross(2R,R) \right) \right] \leq \grandO{1} \alpha^{an}_4(\rho_1,\rho_2) \, .
\]
\end{lem}
\begin{proof}
We write the proof for $y = 0$ since the proof in the other cases is the same. The proof is very similar to the proof of the inequality $\Pro \left[ \Piv_S(\cross(2R,R)) \right] \leq \grandO{1} \alpha^{an}_4(\rho,R)$ of Lemma~\ref{l.piv_bulk}. Hence, we choose to indicate what is the result analogous to~\eqref{e.from_piv_to_arm} (that is the key estimate in the proof of Lemma~\ref{l.piv_bulk}) and to omit the rest of the proof. For $0<\rho'\leq \rho''$, let
\[
\dense(\rho',\rho'') := \dense_{1/100} \left( A(\rho',2\rho') \right) \cap \dense_{1/100} \left( A(\rho'',2\rho'') \right) \, .
\]
Then, for every $k \in \lbrace 0, \cdots, \lfloor \log_2 ( \rho_2/(4 \rho_1)) \rfloor \rbrace$, $\Piv_S^{A(\rho_1,\rho_2)}\left(\cross(2R,R)\right)$ is included in
\[
 \arm_4(2^{k+1}\rho_1,\rho_2/2) \cup \left( \Piv_S^{A(\rho_1,\rho_2)}(\cross(2R,R)) \setminus \dense(2^k\rho_1,\rho_2/2) \right) \, .
\]
\end{proof}

We now use Lemma~\ref{l.pivhat} to estimate the quantity $\Pro \left[ \Piv_S(\cross(2R,R)) \right]$ when $S$ intersects the rectangle $[-2R,2R] \times [-R,R]$ (for instance when $S$ is included in this rectangle). We first need the following notations: Let $d_0=d_0(S)$ be the distance between $S$ and the closest side of $[-2R,2R] \times [-R,R]$ and let $y_0$ be the orthogonal projection of $y$ on this side. Also, let $d_1=d_1(S) \geq d_0$ be the distance between $y_0$ and the closest corner of $[-2R,2R] \times [-R,R]$ and let $y_1$ be this corner. Write $\alpha^{an,++}_j(\cdot,\cdot)$ for the probability of the $j$-arm event in the quarter plane. The following lemma is a generalization of Lemma~\ref{l.piv_bulk}. 

\begin{lem}\label{l.piv_boundary}
Let $1 \leq \rho \leq R/10$, let $y$ be a point of the plane and let $S=B_\rho(y)$. Assume that $S$ intersects the rectangle $[-2R,2R] \times [-R,R]$. We have
\[
\Pro \left[ \Piv_S(\cross(2R,R)) \right] \leq \grandO{1} \, \alpha^{an,++}_2(d_1+\rho,R) \, \alpha^{an,+}_3(d_0 + \rho,d_1) \, \alpha^{an}_4(\rho,d_0) \, .
\]
\end{lem}
\begin{proof}
We use the notations from above the lemma and we let $S_0=B_{10(d_0 + \rho)}(y_0)$ and $S_1=B_{100(d_1 + \rho)}(y_1)$. We also consider the annuli $A(y;\rho,(\rho+d_0)/10)$ and $A(y_0;10(d_0+\rho),d_1)$ (note that these annuli may be empty), see Figure~\ref{f.annuli_and_boxes}. Since $S \subseteq S_0 \subseteq S_1$, $\Piv_S(\cross(2R,R))$ is included in the following event:
\begin{figure}[!h]
\begin{center}
\includegraphics[scale=0.6]{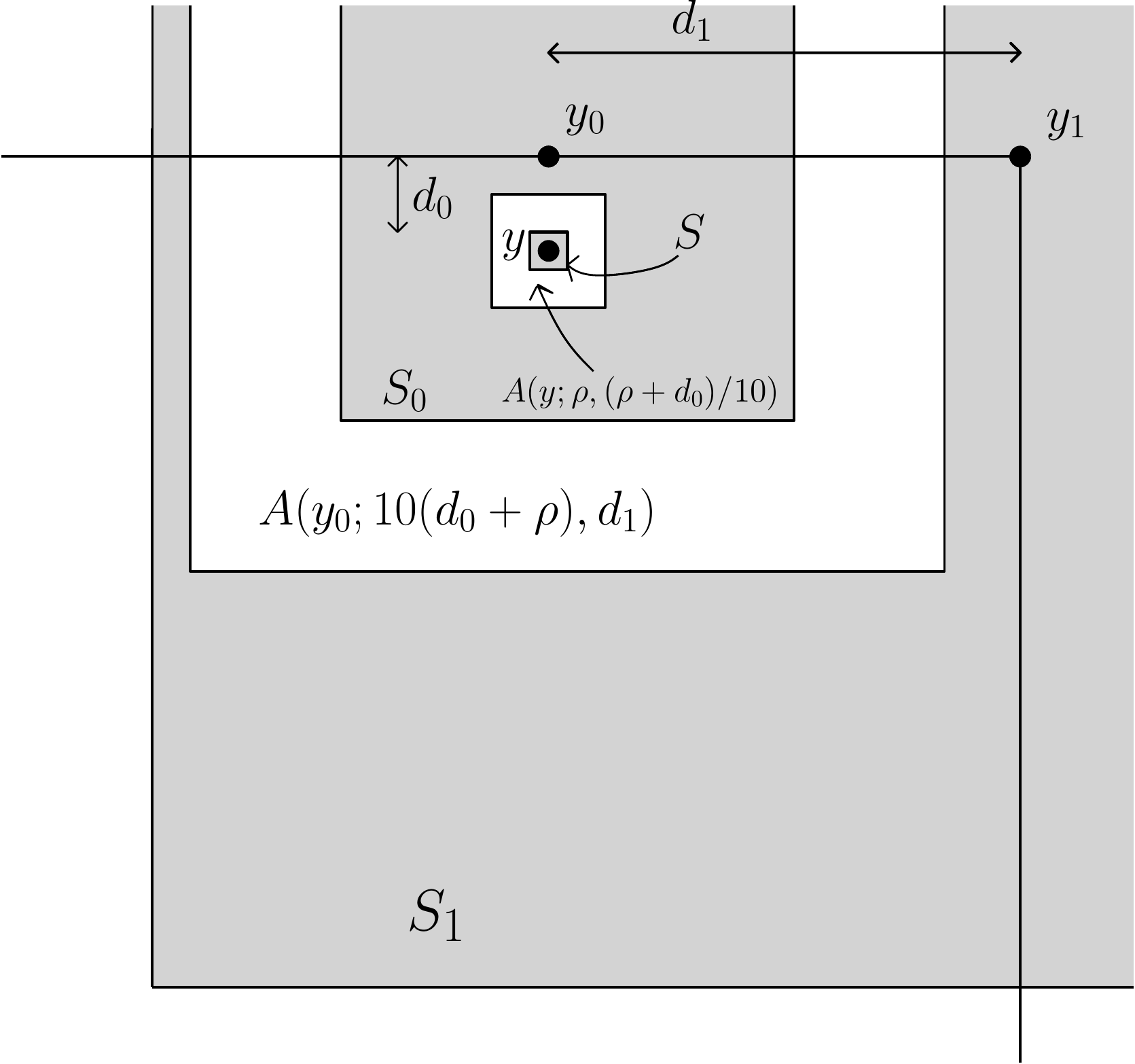}
\end{center}
\caption{The points $y$, $y_0$ and $y_1$, the boxes $S$, $S_0$ and $S_1$, and some annuli centered at $y$ or $y_0$.}\label{f.annuli_and_boxes}
\end{figure}
\begin{multline*}
\Piv_{S_1}(\cross(2R,R)) \cap \Piv_{S_0}^{A(y_0;10(d_0 + \rho),d_1)}(\cross(2R,R))\\
\cap \Piv_{S}^{A(y;\rho,(\rho+d_0)/10)}(\cross(2R,R)) \, .
\end{multline*}
Note furthermore that: i) $S$ is the inner square of $A(y;\rho,(\rho+d_0)/10)$, ii) $A(y;\rho,(\rho+d_0)/10)$ is included in $S_0$, iii) $S_0$ is the inner square of $A(y_0;10(d_0+\rho),d_1)$ and iv) $A(y_0;10(d_0+\rho),d_1)$ is included in $S_1$. Note also that
\bi 
\item[i)] $\Piv_S^{A(y;\rho,(\rho+d_0)/10)}(\cross(2R,R))$ is measurable with respect to $\omega \cap A(y;\rho,(\rho+d_0)/10)$,
\item[ii)] $\Piv_{S_0}^{A(y_0;10(d_0 + \rho),d_1)}(\cross(2R,R))$ is measurable with respect to $\omega \cap A(y_0;10(d_0 + \rho),d_1)$,
\item[iii)] $\Piv_{S_1}(\cross(2R,R))$ is measurable with respect to $\omega \setminus S_1$.
\ei
Hence, by spatial independence, $\Pro \left[ \Piv_S(\cross(2R,R)) \right]$ is at most
\begin{multline*}
\Pro \left[ \Piv_{S_1}(\cross(2R,R)) \right]
\times \Pro \left[  \Piv_{S_0}^{A(y_0;10(d_0 + \rho),d_1)}(\cross(2R,R)) \right] \\\times \Pro \left[ \Piv_{S}^{A(y;\rho,(\rho+d_0)/10)}(\cross(2R,R))  \right] \, .
\end{multline*}
Lemma~\ref{l.pivhat} implies that
\[
\Pro \left[\Piv_S^{A(y;\rho,(\rho+d_0)/10)}(\cross(2R,R)) \right] \leq \grandO{1} \alpha^{an}_4(\rho,\rho+d_0) \, .
\]
Moreover, by the quasi-multiplicativity property and~\eqref{e.poly}, we have $\alpha^{an}_4(\rho,\rho+d_0) \leq \grandO{1} \alpha^{an}_4(\rho,d_0)$.\\

By exactly the same proof as Lemma~\ref{l.pivhat} but applied to the $3$-arm event in the half-plane, we have
\[
\Pro \left[ \Piv_{S_0}^{A(y_0;10(d_0 + \rho),d_1)}(\cross(2R,R)) \right] \leq \grandO{1} \alpha^{an,+}_3(d_0+\rho,d_1)
\]
and
\[
\Pro \left[ \Piv_{S_1}(\cross(2R,R)) \right] \leq \grandO{1} \alpha^{an,++}_2(d_1+\rho,R) \, ,
\]
which ends the proof.
\end{proof}

Let us now prove an estimate about the probability that boxes outside of $[-2R,2R] \times [-R,R]$ are pivotal. Roughly speaking, this estimate implies that, if we want to bound
\[
\sum_{S \text{ square of the grid } 2\rho\Z^2 \text{ not included in } [-2R,2R] \times [-R,R]} \Pro \left[ \Piv_S(\cross(2R,R)) \right] \, ,
\]
then it is enough to control the sum over the squares $S$ that \textbf{intersect} $\partial ([-2R,2R] \times [-R,R])$.
\begin{lem}\label{l.piv_far_away}
Let $\rho \geq 1$ and let $R \geq 100\rho$. Also, let $S$ be a square of the grid $2\rho\Z^2$ that intersects $\partial ([-2R,2R] \times [-R,R])$. Moreover, let $\setS$ be the set of all squares $S'$ of the grid $2\rho \Z^2$ that do not intersect $[-2R,2R] \times [-R,R]$ and are such that $S$ is the argmin of $S'' \mapsto \dist(S'',S')$ where $S''$ ranges over the set of squares of the grid $2\rho\Z^2$ that intersect $\partial ([-2R,2R] \times [-R,R])$. Then,
\[
\sum_{S' \in \setS} \Pro \left[ \Piv_{S'}(\cross(2R,R)) \right] \leq \grandO{1} \, \alpha^{an,++}_2(d_1+\rho,R) \, \alpha^{an,+}_3(d_0 + \rho,d_1) \, \alpha^{an}_4(\rho,d_0) \, ,
\]
where $d_0=d_0(S)$ and $d_1=d_1(S)$ are the distances defined above Lemma~\ref{l.piv_boundary}.
\end{lem}
\begin{proof}
If $S' \in \setS$, we let $d'$ be the distance between $S'$ and $[-2R,2R] \times [-R,R]$. We first observe that, if we sum only on the squares $S'$ that are at distance at least $R/1000$ from $[-2R,2R] \times [-R,R]$, then the result is easy. Indeed, $\Piv_{S'}(\cross(2R,R))$ implies that, given $\eta \setminus S'$, the probability that a Voronoi cell intersects both $S'$ and $[-2R,2R] \times [-R,R]$ is positive, which is an event of probability less than $\grandO{1}\exp(-\Omega(1)(d')^2)$ if $d'$ is at least of order $R$. Thus, the sum over such squares $S'$ is less than $\grandO{1}\exp(-\Omega(1)R^2)$, which is much less than the desired bound.\\

Now, let $S' \in \setS$ be such that $d' \leq R/1000$. Let $y$ be the center of $S$ and let $\widetilde{S}=B_{3(\rho+d')}(y)$. Note that $\widetilde{S} \supseteq S,S'$. In particular, $\Piv_{S'}(\cross(2R,R)) \subseteq \Piv_{\widetilde{S}}(\cross(2R,R))$. Let $\widetilde{\rho}=3(\rho+d')$. Since $\widetilde{\rho} \leq R/10$ (this comes from the fact that $d'\leq R/1000$), we can apply Lemma~\ref{l.piv_boundary} to $\widetilde{S}$ and we obtain that
\begin{equation}\label{e.piv_S''}
\Pro \left[ \Piv_{\widetilde{S}}(\cross(2R,R)) \right] \leq \grandO{1} \alpha^{an,++}_2(\widetilde{d}_1+\widetilde{\rho},R) \, \alpha^{an,+}_3(\widetilde{d}_0 + \widetilde{\rho},\widetilde{d}_1) \, \alpha^{an}_4(\widetilde{\rho},\widetilde{d}_0) \, ,
\end{equation}
where $\widetilde{d}_0=d_0(\widetilde{S})$ and $\widetilde{d}_1=d_1(\widetilde{S})$. Note that $\widetilde{d}_0$ and $\widetilde{d}_1$ satisfy $|d_0-\widetilde{d}_0|\leq \grandO{1}(\rho+d')$ and $|\widetilde{d}_1-d_1| \leq \grandO{1} (\rho+d')$.\\

\begin{figure}[!h]
\begin{center}
\includegraphics[scale=0.5]{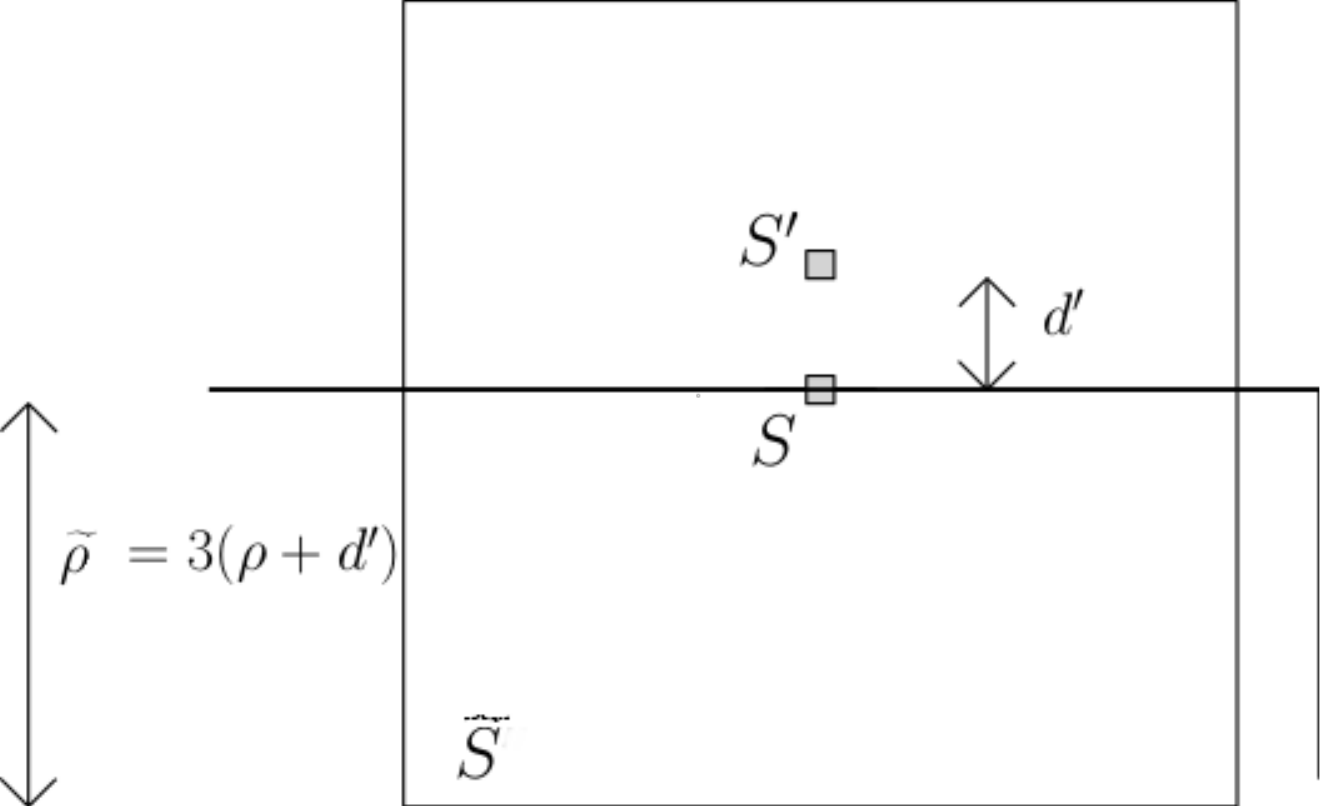}
\end{center}
\caption{The boxes $S$, $S'$ and $\widetilde{S}$ when $d'$ is at least of the order of $\rho$.}\label{f.boxes_dprime}
\end{figure}

We now distinguish between the two cases $d' \leq 4\rho$ and $d' \in [4\rho,R/1000]$:
\bi 
\item If $d' \leq 4\rho$, then $|\widetilde{\rho}-\rho|$, $|\widetilde{d}_1-d_1|$ and $|\widetilde{d}_0-d_0|$ are less than $\grandO{1}\rho$. As a result, the quasi-multiplicativity property and~\eqref{e.piv_S''} imply that
\[
\Pro \left[ \Piv_{\widetilde{S}}(\cross(2R,R)) \right] \leq  \grandO{1} \alpha^{an,++}_2(d_1+\rho,R) \, \alpha^{an,+}_3(d_0 + \rho,d_1) \, \alpha^{an}_4(\rho,d_0) \, .
\]
Since there are $\grandO{1}$ squares $S' \in \setS$ such that $d' \leq 4\rho$, the proof is over in this case.
\item Assume that $d' \in [4\rho,R/1000]$ and observe that $\Piv_{S'}(\cross(2R,R))$ is included in the intersection of the two independent events $\Piv_{\widetilde{S}}(\cross(2R,R))$ and $\neg \dense_{1/100}(\widetilde{S}\setminus S)$ (indeed, if $\dense_{1/100}(\widetilde{S} \setminus S)$ holds, then there cannot exist a Voronoi cell that intersects both $S'$ and $[-2R,2R] \times [-R,R]$, see Figure~\ref{f.boxes_dprime}). By using~\eqref{e.piv_S''} and the fact that $\Pro \left[ \neg \dense_{1/100}(\widetilde{S}\setminus S) \right] \leq \grandO{1}\exp(-\Omega(1)(d')^2)$, we obtain that $\Pro \left[ \Piv_{S'}(\cross(2R,R)) \right]$ is at most
\[
\grandO{1} \exp(-\Omega(1)(d')^2) \, \alpha^{an,++}_2(\widetilde{d}_1+\widetilde{\rho},R) \, \alpha^{an,+}_3(\widetilde{d}_0 + \widetilde{\rho},\widetilde{d}_1) \, \alpha^{an}_4(\widetilde{\rho},\widetilde{d}_0) \, .
\]
By the quasi-mutliplicativity property and since $\exp(-\Omega(1)(d')^2)$ decays super-polynomially fast, the above at most
\[
\grandO{1} \exp(-\Omega(1)(d')^2) \, \alpha^{an,++}_2(d_1+\rho,R) \, \alpha^{an,+}_3(d_0 + \rho,d_1) \, \alpha^{an}_4(\rho,d_0) \, .
\]
Let us now sum over each $S'$ such that $d' \in [4\rho,R/1000]$. Since, for each integer $k \in [\log_2(4\rho),\log_2(R/1000)]$, there exist at most $\grandO{1}2^{2k}$ squares $S'$ such that $d' \in [2^k,2^{k+1}]$, the sum is at most
\begin{multline*}
\grandO{1} \sum_{k=\log_2(4\rho)}^{\log_2(R/1000)} 2^{2k} \exp(-\Omega(1)2^{2k}) \, \alpha^{an,++}_2(d_1+\rho,R) \, \alpha^{an,+}_3(d_0 + \rho,d_1) \, \alpha^{an}_4(\rho,d_0)\\
\leq \grandO{1} \alpha^{an,++}_2(d_1+\rho,R) \, \alpha^{an,+}_3(d_0 + \rho,d_1) \, \alpha^{an}_4(\rho,d_0) \, .
\end{multline*}
\ei
This ends the proof.
\end{proof}

Now, we can prove Proposition~\ref{p.asymp_R2alpha4}.

\begin{proof}[Proof of Proposition~\ref{p.asymp_R2alpha4}] Let $\setS_1$ be the set of squares of the grid $(2 \rho \Z)^2$ that are included in $[-2R,2R] \times [-R,R]$ and are at distance at least $R/3$ from the sides of this rectangle, and let $\setS_2 \supseteq \setS_1$ be the set of squares of the grid $(2 \rho \Z)^2$ that intersect $[-2R,2R] \times [-R,R]$. First, note that if we use Lemma~\ref{l.piv_bulk}, we obtain that
\[
\sum_{S \in \setS_1} \Pro \left[ \Piv_S(\cross(2R,R)) \right] \asymp \left( \frac{R}{\rho} \right)^2 \alpha^{an}_4(\rho,R) \, .
\]
Hence, it is sufficient to prove that
\[
\sum_{S \text{ square of the grid } (2 \rho \Z)^2}  \Pro \left[ \Piv_S(\cross(2R,R)) \right] \leq \grandO{1} \, \left( \frac{R}{\rho} \right)^2 \alpha^{an}_4(\rho,R) \, .
\]
Moreover, by Lemma~\ref{l.piv_far_away}, it is sufficient to prove the estimate by summing only on $\setS_2$. Let $S \in \setS_2$. By using Lemma~\ref{l.piv_boundary} combined with the estimates~\eqref{e.alpha4easy} (to control $\alpha^{an,+}_3(\cdot,\cdot)$) and~\eqref{e.poly} (to control $\alpha^{an,++}_2(\cdot,\cdot)$), we obtain that there exists an exponent $a > 0$ such that
\[
\Pro \left[ \Piv_S(\cross(2R,R)) \right] \leq \grandO{1} \, \left( \frac{d_1 + \rho}{R} \right)^a \, \alpha^{an}_4(\rho,d_0) \, \alpha^{an}_4(d_0+\rho,d_1) \, .
\]
The quasi-multiplicativity property (together with~\eqref{e.poly}) implies that
\[
\alpha^{an}_4(\rho,d_0) \, \alpha^{an}_4(d_0+\rho,d_1) \leq \grandO{1} \alpha^{an}_4(\rho,d_0+\rho) \, \alpha^{an}_4(d_0+\rho,d_1+\rho) \leq \grandO{1} \, \alpha^{an}_4(\rho,d_1+\rho) \, .
\]
If we use once again the quasi-multiplicativity property and the estimate~\eqref{e.alpha4easy}, we obtain that
\[
\Pro \left[ \Piv_S (\cross(2R,R)) \right] \leq \grandO{1} \, \left( \frac{d_1 + \rho}{R} \right)^a \, \left( \frac{R}{d_1+\rho} \right)^2 \, \alpha^{an}_4(\rho,R) \, .
\]
Now, note that the number of squares $S \in \setS_2$ such that $d_1+\rho \in [(2^k-1)\rho,2^{k+1}\rho]$ is $0$ if $k \geq \log_2(R/\rho)$ and is at most $\grandO{1} \, 2^{2k}$ otherwise. Therefore,
\begin{eqnarray*}
\sum_{S \in \setS_2} \Pro \left[ \Piv_S (\cross(2R,R)) \right] & \leq & \grandO{1} \alpha^{an}_4(\rho,R) \sum_{k = 0}^{\lfloor \log_2(R/\rho) \rfloor} 2^{2k} \left( \frac{2^k \rho}{R} \right)^{a-2}\\
& \leq & \grandO{1} \, \alpha^{an}_4(\rho,R) \, \left( \frac{\rho}{R} \right)^{a-2} \sum_{k = 0}^{\lfloor \log_2(R/\rho) \rfloor} 2^{ka}\\
& \leq & \grandO{1} \, \alpha^{an}_4(\rho,R) \, \left( \frac{\rho}{R} \right)^{a-2} \left( \frac{R}{\rho} \right)^{a}\\
& = & \grandO{1} \, \alpha^{an}_4(\rho,R) \, \left( \frac{R}{\rho} \right)^2 \, ,
\end{eqnarray*}
which is the desired result.
\end{proof}

Now, let us discuss the same kind of questions for arm events instead of crossing events, i.e. let us prove Proposition~\ref{p.arm_event_asymp_R2alpha4}. The main difference is that we will have to use Item~(ii) of Proposition~\ref{p.alpha4} instead of the weaker estimate~\eqref{e.alpha4easy}. As previously, let $y$ be a point of the plane, let $\rho \geq 1$, let $S = B_\rho(y)$ and let $R \in [10\rho,+\infty)$. Also, let $j \in \N^*$. We will need the following lemmas which are similar to Lemmas~\ref{l.piv_bulk}, \ref{l.piv_boundary} and~\ref{l.piv_far_away}.
\begin{lem}\label{l.piv_arm_1}
Let $y$, $\rho$, $R$ and $S=B_\rho(y)$ as above and assume that $S \subseteq A(R/4,R/2)$. Then,
\[
\Pro \left[ \Piv_S( \arm_j(1,R) ) \right] \leq \grandO{1} \alpha^{an}_j(R) \, \alpha^{an}_4(\rho,R) \, . 
\]
\end{lem}

The following is a generalization of Lemma~\ref{l.piv_arm_1}.

\begin{lem}\label{l.piv_arm_2}
Let $y$, $\rho$, $R$ and $S=B_\rho(y)$ as above and assume that $S \subseteq B_{R/2}$. Also, let $d$ be the distance between $y$ and $0$. Then,
\[
\Pro \left[ \Piv_S( \arm_j(1,R) ) \right] \leq \grandO{1} \alpha^{an}_j(R) \, \alpha^{an}_4(\rho,d) \hspace{1em} \text{ if } d \geq 2\rho \, ,
\]
and
\[
\Pro \left[ \Piv_S( \arm_j(1,R) ) \right] \leq \grandO{1} \alpha^{an}_j(\rho,R) \hspace{1em} \text{  otherwise.}
\]
\end{lem}

Let $d_0=d_0(S)$ and $d_1=d_1(S)$ be defined as in the study of $\cross(2R,R)$, except that we consider distances to the box $B_R$ instead of the rectangle $[-2R,2R] \times [-R,R]$.

\begin{lem}\label{l.piv_arm_3}
Let $y$, $\rho$, $R$ and $S=B_\rho(y)$ as above and assume that $S \cap A(R/2,R) \neq \emptyset$. Then,
\[
\Pro \left[ \Piv_S( \arm_j(1,R) ) \right] \leq \grandO{1} \alpha^{an}_j(R) \, \alpha^{an,++}_3(d_1+\rho,R) \, \alpha^{an,+}_3(d_0+\rho,d_1) \, \alpha^{an}_4(\rho,d_0) \, . 
\]
\end{lem}

The following lemma is the analogue of Lemma~\ref{l.piv_far_away}:
\begin{lem}\label{l.piv_arm_4}
Let $\rho \geq 1$ and let $R \geq 100\rho$. Also, let $S$ be a square of the grid $2\rho\Z^2$ that intersects $\partial B_{R}$. Moreover, let $\setS$ be the set of all squares $S'$ of the grid $2\rho \Z^2$ that do not intersect $B_{R}$ and are such that $S$ is the argmin of $S'' \mapsto \dist(S'',S)$ where $S''$ ranges over the squares of the grid $2\rho\Z^2$ that intersect $\partial B_{R}$. Then,
\[
\sum_{S' \in \setS} \Pro \left[ \Piv_{S'}(\cross(2R,R)) \right] \leq \grandO{1} \, \alpha_j^{an}(R) \, \alpha^{an,++}_3(d_1+\rho,R) \, \alpha^{an,+}_3(d_0 + \rho,d_1) \, \alpha^{an}_4(\rho,d_0) \, .
\]
\end{lem}

\begin{proof}[Proof of Lemmas~\ref{l.piv_arm_1}, \ref{l.piv_arm_2}, \ref{l.piv_arm_3} and~\ref{l.piv_arm_4}]
The proof of these lemmas is very similar to the proof of the analogous results for crossing events (Lemmas~\ref{l.piv_bulk},~\ref{l.piv_boundary} and~\ref{l.piv_far_away}). However, there is a new difficulty when $j$ is odd and larger than $1$. More precisely, if some box in the bulk is pivotal (and if $\eta$ is sufficiently dense around this box) then there is a $4$-arm event around this box if $j$ is even and there is either a $4$-arm event or a $6$-arm event if $j$ is odd. For more details, see Appendix~\ref{a.joddpivarm}. See also~\cite{nolin2008near} (e.g. Figure~12 therein) where Nolin deals with the same problem for Bernoulli percolation on the triangular lattice.
\end{proof}

\begin{proof}[Proof of Proposition~\ref{p.arm_event_asymp_R2alpha4}] Let $\setS_1$ be the set of squares of the grid $(2 \rho \Z)^2$ that intersect $B_{R/2}^c$. By using Lemmas~\ref{l.piv_arm_3} and~\ref{l.piv_arm_4} and by following the proof of Proposition~\ref{p.asymp_R2alpha4}, we obtain that
\[
\sum_{S \in \setS_1} \Pro \left[ \Piv_S ( \arm_j(1,R) ) \right] \leq \grandO{1} \alpha^{an}_j(R) \left( \frac{R}{\rho} \right)^2 \alpha^{an}_4(\rho,R) \, .
\]
Let $\setS_2$ be the set of squares of the grid $(2 \rho \Z)^2$ that are included in $B_{R/2}$. Lemma~\ref{l.piv_arm_2} implies that
\[
\sum_{S \in \setS_2} \Pro \left[ \Piv_S ( \arm_j(1,R) ) \right] \leq \grandO{1} \alpha^{an}_j(\rho,R) + \grandO{1} \alpha^{an}_j(R) \sum_{k=0}^{\lfloor \log_2 \left( \frac{R}{\rho} \right) \rfloor} 2^{2k} \, \alpha^{an}_4(\rho,2^k\rho) \, .
\]
The quasi-multiplicativity property and the fact that $\alpha_4^{an}(2^k\rho,R) \geq \Omega(1) \left( 2^k\rho/R \right)^{2-\epsilon}$ (see Item~(ii) of Proposition~\ref{p.alpha4}) imply that
\[
\alpha^{an}_4(\rho,2^k\rho) \leq \grandO{1} \alpha^{an}_4(\rho,R) \, \left( \frac{R}{2^k\rho} \right)^{2-\epsilon} \, .
\]
Therefore, $\sum_{S \in \setS_2} \Pro \left[ \Piv_S ( \arm_j(1,R) ) \right]$ is less than or equal to
\begin{multline*}
\grandO{1} \alpha^{an}_j(\rho,R) + \grandO{1} \alpha^{an}_j(R) \, \alpha^{an}_4(\rho,R) \sum_{k=0}^{\lfloor \log_2 \left( \frac{R}{\rho} \right) \rfloor} 2^{2k} \left( \frac{R}{2^k\rho} \right)^{2-\epsilon}\\
\leq  \grandO{1} \alpha^{an}_j(\rho,R) + \grandO{1} \alpha^{an}_j(R) \, \alpha^{an}_4(\rho,R) \left( \frac{R}{\rho} \right)^2 \, .
\end{multline*}
We are done since $\setS_1 \cup \setS_2 = \lbrace \text{squares of the grid } (2 \rho \Z)^2 \rbrace$.
\end{proof}

\section{Extension of the results to the near-critical phase}\label{s.pneq1/2}

In this section, we extend the results of other sections to the near-critical phase. Remember the definition of the correlation length $L^{an}(p)$ from Definition~\ref{d.corrlength}. Let us first prove the following result.
\begin{lem}
\label{l.exp_decay}
For every $p > 1/2$, $L^{an}(p) < +\infty$.
\end{lem}
\begin{proof}
This is a simple consequence of the exponential decay property Theorem~2 of~\cite{bollobas2006critical} or Theorem~1 of~\cite{duminil2017exponential}: for every $p<1/2$, there exists a constant $c = c(p) > 0$ such that
\[
\alpha^{an}_{1,p}(R) \leq \exp(-c(p) \, R) \, .
\]
Indeed, by duality, this implies that for every $p >1/2$ there exists a constant $c'=c'(p)>0$ such that
\[
\Pro_p \left[ \cross(2R,R) \right] \geq 1-\exp(-c'(p) R) \, .
\]
\end{proof}


\subsection{Extension of the annealed and quenched box-crossing properties}\label{ss.pneq1/2_1}

Let us use the idea of Lemma~$4.17$ of~\cite{ahlberg2016sharpness} in order to extend the annealed box-crossing property to the near-critical regime.


\begin{prop}\label{p.near-crit_Tassion}
Let $\rho > 0$. There exists a constant $c=c(\rho) \in (0,1)$ such that, for every $p \in (1/2,3/4]$ and every $R \in (0,L^{an}(p)]$,
\[
c \leq \Pro_p \left[ \text{\textup{Cross}}(\rho R,R) \right] \leq 1-c \, .
\]
The constant $c$ may also depend on $\epsilon_0$ in the definition of $L^{an}(p)$.
\end{prop}

\begin{proof}
The left-hand-inequality is a direct consequence of Theorem~\ref{t.Tassion} (and is true for any $R \in (0,+\infty)$). Let us prove the right-hand-inequality. Let $\text{Circ}(r_1,r_2)$ be the event that there is a black circuit (i.e. an injective continuous function from $\R/\Z$ to the black region) in the annulus $A(r_1,r_2)$ surrounding the origin. Note that this event holds if and only if there is no white path from $\partial B_{r_1}$ to $\partial B_{r_2}$. Thanks to~\eqref{e.poly}, we know that there exists $h > 0$ such that $\Pro_{1/2} \left[ \text{Circ}(\rho,M\rho) \right] \geq 1-\frac{1}{h} M^{-h}$ for any $\rho \geq 1$ and $M \geq 1$. Fix some $N \in \N^*$ such that
\[
\left( 1 - \frac{1}{h} N^{-h}  \right)^3 \geq \left( 1-\epsilon_0 \right)^{1/2} \, ,
\]
where $\epsilon_0$ is the constant used to define $L^{an}(p)$. Next, fix some constant $\overline{c} \in (0,1)$ sufficiently small so that
\[
\left( 1-\overline{c}^{\frac{1}{(4N)^2}} \right)^4 \geq (1-\epsilon_0)^{1/2} \, .
\]
By gluing arguments, it is sufficient to prove that for every $r \in [N,\frac{L^{an}(p)}{2}]$ we have
\[
\Pro_p \left[ \cross^*(2r,r) \right] = 1- \Pro_p \left[ \cross (r,2r) \right] \geq \overline{c} \, .
\]
Assume (for a contradiction) that there exists $r \in [N,\frac{L^{an}(p)}{2}]$ such that $\Pro_p \left[ \cross (r,2r) \right] > 1-\overline{c}$. By the standard square-root trick, this implies that there exist a segment $I_r$ included in the left side of $[-r,r] \times [-2r,2r]$ and a segment $I_r'$ included in the right side of $[-r,r] \times [-2r,2r]$ such that: (a) the length of $I_r$ and $I_r'$ is $r/N$ and (b) the probability (under $\Pro_p$) that there is black path in $[-r,r] \times [-2r,2r]$ from $I_r$ to $I_r'$ is at least $1-\overline{c}^{\frac{1}{(4N)^2}}$. Let $\cross(I_r,I_r')$ denote this last event.

Now, note that there exist four events obtained by applying a translation or a reflection symmetry to $\cross(I_r,I_r')$ and three events obtained by applying a translation to $\text{Circ}(r/N,r)$ such that, if these seven events hold, then $\cross(4r,2r)$ holds, see Figure~\ref{f.gluing_paths}. By applying the (annealed) FKG-Harris inequality, we obtain that
\[
\Pro_p \left[ \cross(4r,2r) \right] \geq \left( (1-\overline{c}^{\frac{1}{(4N)^2}} \right)^4 \, \left( 1-\frac{1}{h} N^{-h} \right)^3 \geq (1-\epsilon_0)^{1/2}  \, (1-\epsilon_0)^{1/2} = 1-\epsilon_0 \, ,
\]
which is a contradiction since $2r \leq L^{an}(p)$. Note that we have used that (since $p > 1/2$)
\[
\Pro_p \left[ \text{Circ}(r/N,r) \right] \geq \Pro_{1/2} \left[ \text{Circ}(r/N,r) \right] \, .
\]
\begin{figure}[!h]
\begin{center}
\includegraphics[scale=0.6]{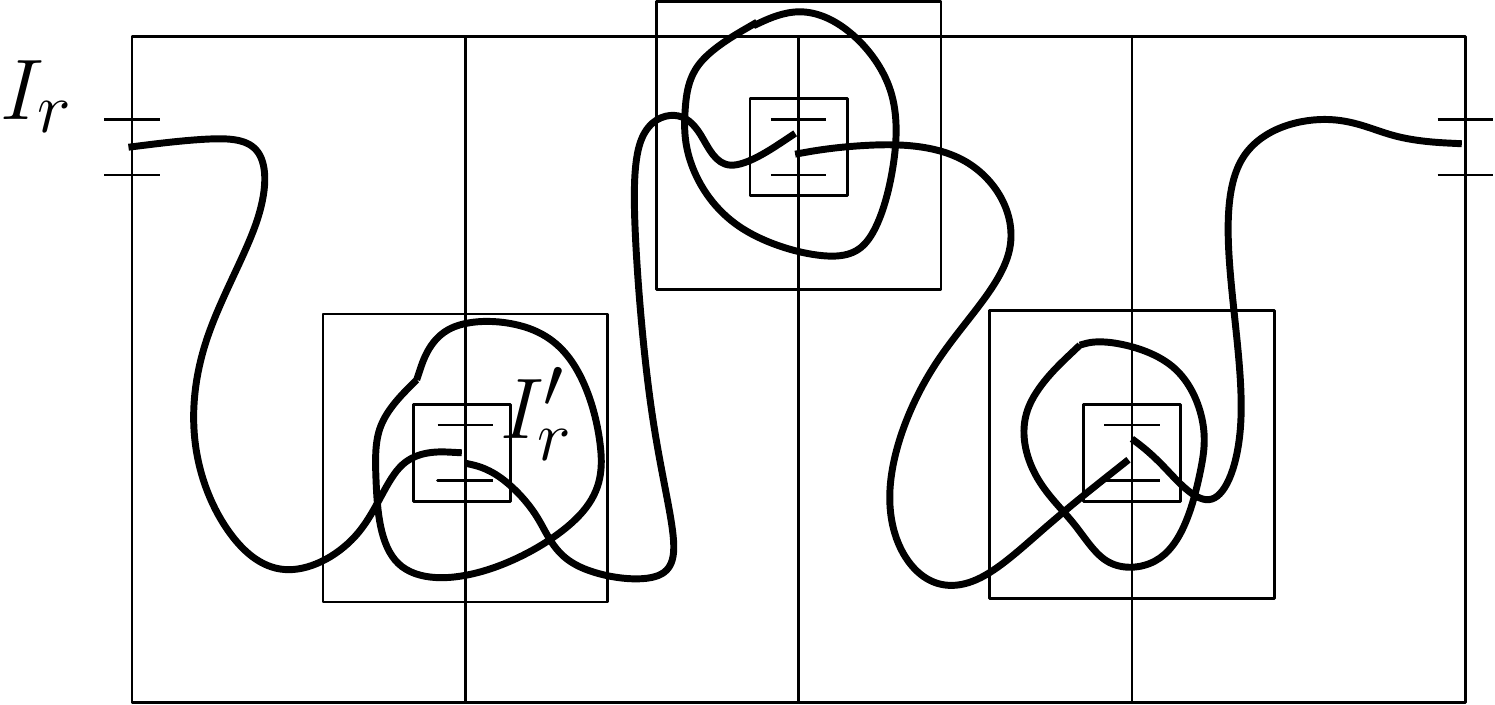}
\end{center}
\caption{Seven events to obtain $\cross(4r,2r)$.}\label{f.gluing_paths}
\end{figure}
\end{proof}

Now, we extend the quenched box-crossing result Theorem~\ref{t.AGMT}.

\begin{prop}\label{p.extension_AGMT}
Let $\rho > 0$. We have the following:
\bi
\item[i)] There exist an absolute constant $\epsilon> 0$ and a constant $C=C(\rho) < +\infty$ such that, for every $p \in (1/2,3/4]$ and every $R \in (0,+\infty)$ we have
\[
\Var \left( \Prob^\eta_p \left[ \cross(\rho R,R \right] \right) \leq C \, R^{-\epsilon} \, .
\]
This implies the following estimate:
\item[ii)] For every $\gamma \in (0,+\infty)$, there exists a positive constant $c=c(\rho,\gamma) \in (0,1)$ such that, for every $p \in (1/2,3/4]$ and every $R \in (0,L^{an}(p)]$,
\[
\Pro \left[ \, c \leq \Prob^\eta_p \left[ \cross(\rho R,R) \right] \leq 1-c \, \right] \geq 1 - R^{-\gamma} \, .
\]
\ei
The constants $C$ and $c$ may also depend on $\epsilon_0$ in the definition of $L^{an}(p)$.
\end{prop}

\begin{proof}
The way we obtain Item~ii) from Item~i) is exactly the same as in the proof of Theorem~\ref{t.AGMT} (see~\cite{ahlberg2015quenched}) except that we use Proposition~\ref{p.near-crit_Tassion} instead of Theorem~\ref{t.Tassion}. So, let us prove Item~i). To this purpose, we rely on Appendix~\ref{a.tAGMT} where we recall the main steps of the proof of Theorem~\ref{t.AGMT}. In the case $p > 1/2$, the first step is exactly the same and we obtain that
\begin{equation}\label{e.var_biased}
\Var \left( \Prob^\eta_p \left[ \, \cross(\rho R,R) \, \right] \right) \leq \E \left[ \sum_{x \in \eta} \Prob_p^\eta \left[ \Piv_x^q(\cross(\rho R,R)) \right]^2 \right] \, .
\end{equation}

For the second step, we cannot use the BK inequality in the case $p>1/2$ since we do not know whether this inequality is true or not. So we prove the result corresponding to this step for $p>1/2$ by using the analogous result for $p=1/2$. More precisely, since $p > 1/2$, we have the following:
\begin{equation}\label{e.compare_p_and_1/2}
\Pro \left[ \Prob^\eta_p \left[ \arm_1^{*,\text{cell}}(S,R) \right] \geq R^{-\epsilon} \right] \leq \Pro \left[ \Prob^\eta_{1/2} \left[ \arm_1^{*,\text{cell}}(S,R) \right] \geq R^{-\epsilon} \right] \, ,
\end{equation}
where $S$ is the $1 \times 1$ square centered at $0$ and $\arm_1^{*,\text{cell}}(S,R)$ is the event defined in the paragraph above~\eqref{e.prop_3.11_AGMT}. Thanks to~\eqref{e.compare_p_and_1/2}, the following is a direct consequence of~\eqref{e.prop_3.11_AGMT}: For every $\gamma > 0$, there exists $\epsilon > 0$ such that the following holds:
\[
\Pro \left[ \Prob^\eta_p \left[ \arm_1^{*,\text{cell}}(S,R) \right] \geq R^{-\epsilon} \right] \leq \frac{1}{\epsilon} R^{-\gamma} \, .
\]
This ends the second step. The third and last step is exactly the same as in Appendix~\ref{a.tAGMT}. (Here we use that the theorems by Schramm and Steif stated in Appendix~\ref{a.Schramm-Steif} - and more precisely Corollary~\ref{c.Schramm_Steif} which is the inequality that we need - hold for every $p$. The only dependence on $p$ in Corollary~\ref{c.Schramm_Steif} is a factor $\frac{1}{4p(1-p)}$, but  this is not a problem since we have restricted ourself to the case $p \in (1/2,3/4]$.) This ends the proof.
\end{proof}

\begin{rem}\label{r.extension_AGMT}
Now, we can explain the reason why, in~Appendix~\ref{a.tAGMT}, we have (slightly) changed the algorithm used to estimate the sum $\E \left[ \sum_{x \in \eta} \Prob_p^\eta \left[ \Piv_x^q(\cross(\rho R,R)) \right]^2 \right]$. The reason is that we wanted to bound the revealment of the algorithm with a quantity that involves \textbf{only white arms} (so that we can use the obvious inequality~\eqref{e.compare_p_and_1/2}), which is not possible with the algorithm chosen in~\cite{ahlberg2015quenched}.
\end{rem}

\subsection{Extension of the results of Sections~\ref{s.first},~\ref{s.pivotals} and~\ref{s.quasi}}\label{ss.pneq1/2_2}

Before proving the annealed scaling relations, we need to prove that the results of Sections~\ref{s.first},~\ref{s.pivotals} and~\ref{s.quasi} are also true in the near-critical phase. More precisely, we need to prove that these results are also true for $p \in (1/2,3/4]$, providing that we assume that every length is less than or equal to the correlation length $L^{an}(p)$. An important fact is that the different constants (e.g. the constant $C=C(j)$ of the quasi-multiplicatitvity property Proposition~\ref{p.quasi}) will \textbf{not depend on $p$} but only on the parameter $\epsilon_0$ from the definition of the correlation length.\\

In Subsection~\ref{ss.pneq1/2_1}, we have proved that the annealed and quenched box crossing estimates also hold in the near-critical phase. In Sections~\ref{s.first},~\ref{s.pivotals} and~\ref{s.quasi}, we use only two properties specific to the parameter $p=1/2$:
\bi 
\item \textbf{The annealed BK inequality} (see Subsection~\ref{ss.correlation_in}). The only place where we have used this inequality is in the proof of the second part of Proposition~\ref{p.alpha4} (see Subsection~\ref{ss.alpha4}). In this proof, we have used this inequality and the fact that $\arm_{2j}(r_1,r_2)$ is included in the $j$-disjoint occurrence of $\arm_1(r_1,r_2)$ to prove that
\[
\alpha^{an}_{2j,1/2}(r_1,r_2) \leq \left( \alpha^{an}_{1,1/2}(r_1,r_2) \right)^j \, .
\]
Let $\arm^*_1(r_1,r_2)$ denote the event that there is a white path from $\partial B_{r_1}$ to $\partial B_{r_2}$. Note that $\arm_{2j}(r_1,r_2)$ is also included in the $j$-disjoint occurrence of $\arm_1^*(r_1,r_2)$, hence,
\[
\alpha^{an}_{2j,p}(r_1,r_2) \leq \Pro_p \left[ \arm_1^*(r_1,r_2)^{\square j} \right] \leq \Pro_{1/2} \left[ \arm_1^*(r_1,r_2)^{\square j} \right] \leq \Pro_{1/2} \left[ \arm_1^*(r_1,r_2) \right]^j \, ,
\]
since $p>1/2$ and by the annealed BK inequality at $p=1/2$. Hence, there exists $a > 0$ such that $\alpha^{an}_{2j,p}(r_1,r_2) \leq \grandO{1} (r_1/r_2)^a$, which is exactly what we needed in the proof of the second part of Proposition~\ref{p.alpha4}.
\item The fact that \textbf{the model is self-dual}. We have actually used this property implicitly all along Sections~\ref{s.pivotals} and~\ref{s.quasi}. If $Q$ is a quad, let $\cross^*(Q)$ be the event that there is a crossing of $Q$ by a \textbf{white} path. We have used a lot Proposition~\ref{p.a_lot_of_quads} both for the event $\cross(Q)$ and the event $\cross^*(Q)$ although we have proved this proposition only for the event $\cross(Q)$. When $p=1/2$, this is not a problem since $\cross(Q)$ and $\cross^*(Q)$ have the same $\Prob^\eta_{1/2}$-probabilities; but when $p>1/2$ we need to prove that Proposition~\ref{p.a_lot_of_quads} also holds with $\cross^*(Q)$ instead of $\cross(Q)$ as soon as we consider a domain $D$ such that $\text{diam}(D) \leq L^{an}(p)$. The proof is actually exactly the same except that we have to use Proposition~\ref{p.extension_AGMT} instead of Theorem~\ref{t.AGMT}.
\ei

\section{Proof of the annealed scaling relations}\label{s.kesten}

In this section, we prove our main result Theorem~\ref{t.scaling} by using the results of all the other sections. We first prove Proposition~\ref{p.alpha_jandalpha_1/2} and the scaling relation~\eqref{e.scaling2} of Theorem~\ref{t.scaling}. We follow the classical strategy developped by Kesten~\cite{kesten1987scaling} for Bernoulli percolation, see also~\cite{werner2007lectures,nolin2008near}. The main difference is that we deal with both the annealed notion and the quenched notion of pivotal events. We refer to Subsection~\ref{ss.intro_piv} for these two notions of pivotal events. Let us recall that, as explained in Section~\ref{s.pneq1/2}, all our results on arm and pivotal events also hold for $p \in (1/2,3/4]$ (with constants that do not depend on $p$) as soon as we work under the correlation length $L^{an}(p)$.

\begin{proof}[Proof of Proposition~\ref{p.alpha_jandalpha_1/2} and of~\eqref{e.scaling2} from Theorem~\ref{t.scaling}]
We will need the following lemma:
\begin{lem}\label{l.scaling}
Let $p \in (1/2,3/4]$ and $R \in [1,L^{an}(p)]$. We have
\begin{eqnarray}
\frac{d}{dp} \Pro_p \left[ \cross(2R,R) \right] & \asymp & R^2 \, \alpha^{an}_{4,p}(R) \, , \label{e.russocross}\\
\forall j \in \N^*, \, \left| \frac{d}{dp} \log ( \alpha^{an}_{j,p}(R) ) \right| & \leq & \grandO{1} \, R^2 \, \alpha^{an}_{4,p}(R) \, . \label{e.russoalphaj}
\end{eqnarray}
The constants in $\asymp$ and $\grandO{1}$ may only depend on the choice of $\epsilon_0$ in Definition~\ref{d.corrlength} (and on $j$ for $\grandO{1}$).
\end{lem}
Before proving Lemma~\ref{l.scaling}, let us explain why this lemma implies Proposition~\ref{p.alpha_jandalpha_1/2} and~\eqref{e.scaling2}.
If we integrate~\eqref{e.russocross} from $1/2$ to $p$, we obtain that
\[
\int_{1/2}^p R^2 \, \alpha^{an}_{4,u}(R) \, du  \asymp \Pro_{p} \left[ \cross(2R,R) \right] - \Pro_{1/2} \left[ \cross(2R,R) \right] \leq 1 \, .
\]
Moreover, if we integrate~\eqref{e.russoalphaj} from $1/2$ to $p$, we obtain that
\[
\left| \log (\alpha^{an}_{j,p}(R) ) - \log (\alpha^{an}_{j,1/2}(R) ) \right| \leq \grandO{1} \, \int_{1/2}^p R^2 \, \alpha^{an}_{4,u}(R) \, du \, .
\]
Hence: $\left| \log (\alpha^{an}_{j,p}(R) ) - \log (\alpha^{an}_{j,1/2}(R) ) \right| \leq \grandO{1}$ i.e. $\alpha^{an}_{j,p}(R) \asymp \alpha^{an}_{j,1/2}(R) \,$. Together with the quasi-multiplicativity property, this implies Proposition~\ref{p.alpha_jandalpha_1/2}.\\

Now, let us integrate~\eqref{e.russocross} from $1/2$ to $p$ with the choice $R=L^{an}(p)$. If we use Proposition~\ref{p.alpha_jandalpha_1/2} with $j=4$, we obtain that
\begin{multline*}
(p-1/2) \,L^{an}(p)^2 \, \alpha^{an}_{4,1/2}(L^{an}(p))\\
\asymp \Pro_{p} \left[ \cross(2L^{an}(p),L^{an}(p)) \right] - \Pro_{1/2} \left[ \cross(2L^{an}(p),L^{an}(p)) \right]\\
\in [1-\epsilon_0-\Pro_{1/2} \left[ \cross(2L^{an}(p),L^{an}(p)) \right],1] \, ,
\end{multline*}
which implies the scaling relation~\eqref{e.scaling2} from Theorem~\ref{t.scaling} since $\epsilon_0$ is sufficiently small.

\begin{proof}[Proof of Lemma~\ref{l.scaling}]
One of Kesten's ideas is to use Russo's differential formula (see for instance Theorem~$2.25$ in~\cite{grimmett1999percolation}) that can be stated as follows: Let $n \in \N^*$ and let $A \subseteq \lbrace -1,1 \rbrace^n$ be an increasing event. Also, let $\Prob^n_p=(p\delta_1+(1-p)\delta_{-1})^{\otimes n}$. Then,
\[
\frac{d}{dp} \Prob^n_p \left[ A \right] = \sum_{i = 1}^n \Prob^n_p \left[ \Piv_i^n(A) \right] \, .
\]
(See Subsection~\ref{ss.intro_piv} for our notations for pivotal events.) To use this formula, we have to work at the quenched level. Note that a.s. the number of points of $\eta$ whose cell intersects $[-2R,2R] \times [-R,R]$ is finite. Hence, if we condition on $\eta$, the event $\cross(2R,R)$ depends on finitely many points. So, we can use Russo's formula and we obtain that
\[
\frac{d}{dp} \Prob^\eta_p \left[ \cross(2R,R) \right] = \sum_{x \in \eta} \Prob_p^\eta \left[ \Piv^q_x(\cross(2R,R)) \right] \, .
\]
Now, let $R \in [1,L^{an}(p)]$ and let $( S_i )_{i \in \N}$ be an enumeration of the squares of the grid $\Z^2$. For all $i$, let $N_i(\cross(2R,R))$ be the number of points $x \in \eta \cap S_i$ which are quenched-pivotal for $\cross(2R,R)$. Note that, if one fixes $R$ and let $M > 0$, then
\[
\text{Card} \{ x \, : \, \Prob_p^\eta \left[ \Piv^q_x(\cross(2R,R)) \right] > 0 \} \, .
\]
is larger than $M$ with probability less than $\grandO{1}e^{-\Omega(1)M^2}$. This implies that one can use dominated convergence and obtain that
\begin{equation}\label{e.russocross_upper}
\frac{d}{dp} \Pro_p \left[ \cross(2R,R) \right] = \sum_{i \in \N} \E_p \left[ N_i(\cross(2R,R)) \right] \, . 
\end{equation}
By using the fact that a.s. $\Piv_{S_i}^q( \cross(2R,R) ) \subseteq \Piv_{S_i}( \cross(2R,R) )$ and that $\Piv_{S_i}( \cross(2R,R) )$ is independent of the configuration in $S_i$ we obtain that
\begin{eqnarray*}
\E_p \left[ N_i(\cross(2R,R)) \right] & \leq & \E_p \left[ |  \eta \cap S_i  | \un_{\Piv^q_{S_i}( \cross(2R,R) )} \right]\\
& \leq & \E_p \left[ | \eta \cap S_i | \un_{\Piv_{S_i}( \cross(2R,R) )} \right]\\
& = & \E_p \left[ | \eta \cap S_i | \right] \, \Pro_p \left[ \Piv_{S_i}( \cross(2R,R) ) \right]\\
& \leq & \grandO{1} \Pro_p \left[ \Piv_{S_i}(\cross(2R,R)) \right] \, .
\end{eqnarray*}
If we combine the above with~\eqref{e.russocross_upper} and if we use Proposition~\ref{p.asymp_R2alpha4} (or rather the analogous result in the near-critical regime), we obtain that
\[
\frac{d}{dp} \Pro_p \left[ \cross(2R,R) \right] \leq \grandO{1} \sum_{i \in \N} \Pro_p \left[ \Piv_{S_i}(\cross(2R,R)) \right] \asymp \grandO{1} R^2 \, \alpha^{an}_{4,p}(R) \, ,
\]
which is the upper-bound of~\eqref{e.russocross}. The lower-bound is a simple consequence of Lemma~\ref{l.piv_one_point}. Indeed, this lemma implies that, if $S_i$ is in the ``bulk'' of $[-2R,2R] \times [-R,R]$, then,
\[
\E_p \left[ N_i(\cross(2R,R)) \right] \geq \Pro_p \left[ \left\lbrace \left| \eta \cap S_i \right| = 1 \right\rbrace \cap \Piv^q_{S_i} \left( \cross(2R,R) \right) \right] \geq \Omega(1) \, \alpha^{an}_{4,p}(R) \, .
\]

Now, let us prove~\eqref{e.russoalphaj} for $j = 1$. Since $\arm_1(R)$ is also an increasing event, by the same techniques as above we have
\[
\frac{d}{dp} \alpha^{an}_{1,p}(R) \leq \grandO{1} \sum_{i \in \N} \Pro_p \left[ \Piv_{S_i}(\arm_1(R)) \right] \, .
\]
Together with Proposition~\ref{p.arm_event_asymp_R2alpha4}, this implies~\eqref{e.russoalphaj} for $j = 1$.\\

Let us now prove~\eqref{e.russoalphaj} for $j \geq 2$. In this case, the events $\arm_j(R)$ are not monotonic. In particular, we cannot use Russo's formula. We rather use the following inequality that holds for any event and whose prove is exactly the same as Russo's formula: Let $n \in \N^*$ and $A \subseteq \lbrace -1,1 \rbrace^n$. Then,
\[
\left| \frac{d}{dp} \Prob^n_p \left[ A \right] \right| \leq \sum_{i = 1}^n \Prob^n_p \left[ \Piv_i^n(A) \right] \, .
\]
Thus, the proof of~\eqref{e.russoalphaj} in the case $j\geq 2$ is the same as in the case $j=1$.
\end{proof}
This ends the proof of Proposition~\ref{p.alpha_jandalpha_1/2} and of the scaling relation~\eqref{e.scaling2} from Theorem~\ref{t.scaling}.
\end{proof}

It only remains to prove the scaling relation~\eqref{e.scaling1} from Theorem~\ref{t.scaling}. Thanks to Proposition~\ref{p.alpha_jandalpha_1/2} applied to $j=1$, it is sufficient to prove the following lemma:
\begin{lem}\label{l.scaling2}
Let $p \in (1/2,3/4]$. We have:
\[
\theta^{an}(p) \asymp \alpha^{an}_{1,p} \left( L^{an}(p) \right) \, ,
\]
where the constants in $\asymp$ may only depend on the constant $\epsilon_0$ in the definition of $L^{an}(p)$.
\end{lem}

\begin{proof}
First, note that it is sufficient to prove the result for $p \in (1/2,p_0)$ for some $p_0 \in (1/2,3/4]$. For every bounded Borel subset of the plane $D$, let $\widetilde{\dense}(D)$ be the event that, for any $u \in D$, there exists $x \in \eta \cap D$ \textbf{that is black} and satisifies $||x-u||_2 \leq \diam(D)/100$ (we study this event since it is annealed increasing). With the same proof as Lemma~\ref{l.dense}, we can easily obtain that, if $p \in (1/2,3/4]$ and if $R$ is sufficiently large ($R \geq R_0 \geq 1000$, say), then,
\[
\Pro \left[ \widetilde{\dense}([-2R,2R] \times [-R,R]) \right] \geq 1-\epsilon_0 \, .
\]
Let $p_0$ be the larger parameter $p \in [1/2,3/4]$ such that $L^{an}(p) \geq  R_0$. Note that $p_0 > 1/2$. Consider some parameter $p \in (1/2,p_0)$. We now apply a Peierls argument. If $Q$ is a $4L^{an}(p) \times 2L^{an}(p)$ rectangle, then we say that $Q$ is \textbf{good} if: i) $Q$ is crossed lengthwise, ii) the two $2L^{an}(p) \times 2L^{an}(p)$ squares whose union is $Q$ are crossed from top to bottom and from left to right and iii) $\widetilde{\dense}(Q)$ holds. Note that $\Pro \left[ Q \text{ is good} \right] \geq 1-4\epsilon_0$ by definition of the correlation length and that $\{ Q \text{ is good} \}$ is annealed-increasing.

Now, let $\mathcal{L}(p)$ be the square lattice times $2L^{an}(p)$. We say that an edge $e=\{x,y\}$ of this lattice is good if the $4 L^{an}(p) \times 2 L^{an}(p)$ rectangle which is the union of the two $2L^{an}(p) \times 2L^{an}(p)$ squares centered at $x$ and $y$ is good. Note that we have defined a $2$-dependent percolation model on $\mathcal{L}(p)$ with parameter at least $1-4\epsilon_0$. A standard Peierls argument implies that, if $\epsilon_0$ is small enough, the probability that there is an infinite good path starting from $0$ is larger than some positive constant that depends only on $\epsilon_0$.

Now, note that if the two following properties hold then the event $\{ 0 \leftrightarrow \infty \}$ holds: (a) in the $2$-dependent percolation model on $\mathcal{L}(p)$, the twelve edges of $\mathcal{L}(p)$ closest to $0$ are good and there is an infinite path made of good edges starting from $0$; (b) $\arm_1(1,3L^{an}(p))$ holds and $B_1=[-1,1]^2$ is entirely colored black. This observation and the above paragraph (together with the annealed FKG-Harris inequality) imply the lemma.
\end{proof}

\section{The quasi-multiplicativity property}\label{s.quasi}

In this section, we only work at $p=1/2$, hence we forget the subscript $p$ in the notations. In Subsections~\ref{ss.quasi_even},~\ref{ss.QM_consequences} and~\ref{ss.half-plane}, we only rely on the results of Subsections~\ref{ss.warm} and~\ref{ss.a_lot_of_quads}. In Subsections~\ref{ss.QM_odd}, we also use the results of Subsection~\ref{ss.preliminary} which are consequences of Subsections~\ref{ss.quasi_even},~\ref{ss.QM_consequences} and~\ref{ss.half-plane}.

\subsection{The case $j$ even}\label{ss.quasi_even}

In this subsection, we prove the quasi-multiplicativity property Proposition~\ref{p.quasi} in the case $j$ even:

\begin{prop}\label{p.quasieven}
Let $j \in \N^*$ even. There exists a constant $C=C(j) \in [1,+\infty)$ such that, for all $1 \leq r_1 \leq r_2 \leq r_3$,
\begin{equation}\label{e.quasi}
\frac{1}{C} \, \alpha^{an}_j(r_1,r_3) \leq \alpha^{an}_j(r_1,r_2) \, \alpha^{an}_j(r_2,r_3) \leq C \, \alpha^{an}_j(r_1,r_3) \, .
\end{equation}
\end{prop}

\begin{rem}
As we will see in Subsection~\ref{ss.half-plane}, the same proof will imply the quasi-multiplicativity property for the quantities $\alpha^{an,+}_j(\cdot,\cdot)$, for any $j$.
\end{rem}
As explained in Subsection~\ref{ss.strategy_quasi}, we first need to define what is a ``good percolation configuration'' (i.e. a configuration for which it is not difficult to extend the $j$ arms). \textbf{We write the proof of Proposition~\ref{p.quasieven} for $j=4$ since the proof for other even integers is the same}.

\subsubsection{What does ``looking good'' means for the Voronoi percolation configurations} 

\paragraph{The point configuration.} We consider $\delta \in (0,1/1000)$ and $R \in [\delta^{-2},+\infty)$. In the proof of Proposition~\ref{p.quasieven}, we use the following notations (the notations from the right-hand-side are those from Definition~\ref{d.dense} and Proposition~\ref{p.a_lot_of_quads}):
\begin{eqnarray*}
\dense_\delta(R) & := & \dense_\delta \left( A(R/2,2R) \right) \, ,\\
\qbc_\delta(R) & := & \qbc_\delta^1 \left( A(3R/4,3R/2) \right) \, ,\\
\qbc^{ext}(R) & := & \qbc_{1/100}^1 \left( A(R,4R) \right) \cap \dense_{1/100} \left( A(R,4R) \right) \, ,\\
\qbc^{int}(R) & := & \qbc_{1/100}^1 \left( A(R/4,R) \right) \cap \dense_{1/100} \left( A(R/4,R) \right)  \, .
\end{eqnarray*}
Mind the presence of the events ``$\dense(\cdot)$'' in the definition of $\qbc^{ext}(R)$ and $\qbc^{int}(R)$ (to simplify the notations). Next, we define the two following events:
\begin{equation}
\gp^{ext}_\delta(R) := \left\lbrace \Pro \left[ \dense_\delta(R) \cap \qbc_\delta(R) \cap \qbc^{ext}(R) \cond \eta \cap A(R/2,2R) \right] \geq 3/4 \right\rbrace
\end{equation}
and
\begin{equation}
\gp^{int}_\delta(R) := \left\lbrace \Pro \left[ \dense_\delta(R) \cap \qbc_\delta(R) \cap \qbc^{int}(R) \cond \eta \cap A(R/2,2R) \right] \geq 3/4 \right\rbrace
\end{equation}
(for ``Good Point configuration''). In words, $\gp^{\ext}_\delta(R)$ (respectively $\gp^{int}_\delta(R)$) is the event that, conditionally on $\eta \cap A(R/2,2R)$, the probability that $\dense_\delta(R) \cap \qbc_\delta(R) \cap \qbc^{ext}(R)$ (respectively $\dense_\delta(R) \cap \qbc_\delta(R) \cap \qbc^{int}(R)$) holds is at least $3/4$. Note that, if $\dense_\delta(R)$ holds and if the Voronoi cell of some $x \in \eta$ intersects $A(3R/4,3R/2)$, then $x \in A(R/2,2R)$. Hence, $\dense_\delta(R) \cap \qbc_\delta(R)$ is measurable with respect to $\eta \cap A(R/2,2R)$. As a result, we have
\[
\gp^{ext}_\delta(R) = \dense_\delta(R) \cap \qbc_\delta(R) \cap \left\lbrace \Pro \left[ \qbc^{ext}(R) \cond \eta \cap A(R/2,2R) \right] \geq 3/4 \right\rbrace \, ,
\]
and the analogous property for $\gp^{int}_\delta(R)$. The reason why we do not choose to define $\gp^{ext}_\delta(R) = \dense_\delta(R) \cap \qbc_\delta(R) \cap \qbc^{ext}(R)$ is that we want $\gp^{ext}_\delta(R)$ to be measurable with respect to $\eta \cap A(R/2,2R)$ (and similarly for $\gp^{int}_\delta(R)$). This will be crucial in the whole proof.

\paragraph{The interfaces.} In Subsection~\ref{ss.preliminary}, we have estimated the events $\gi_\delta^{ext}(R)$ and $\gi_\delta^{int}(R)$ which are events that ``the interfaces are well separated''. In particular, we have proved Lemma~\ref{l.interfaces} by using that the exponent of the $3$-arm event in the half-plane is $2$. As we will see in Subsection~\ref{ss.half-plane}, the quasi-multiplicativity is a crucial ingredient in the computation of this exponent. Consequently, we cannot use Lemma~\ref{l.interfaces} in the present proof. We rather choose to consider a variant of the quantities $s^{ext}(r,R)$ and $s^{int}(r,R)$. More precisely:

We still consider $\delta \in (0,1/1000)$ and $R \in [\delta^{-2},+\infty)$. We also consider some $r \in [1,R]$. Following the appendix of~\cite{schramm2010quantitative}, we let $\widetilde{s}^{ext}(r,R)$ be the least distance between any pair of endpoints on $\partial B_R$ of two interfaces that go from $\partial B_r$ to $\partial B_R$. We write $\widetilde{\gi}^{ext}_\delta(R) = \lbrace \widetilde{s}^{ext}(3R/4,R) \geq 10 \delta R \rbrace$ and
\[
G^{ext}_\delta(R) = \gp^{ext}_\delta(R) \cap \widetilde{\gi}^{ext}_\delta(R) \, .
\]
Note that the event $\dense_\delta(R) \cap \widetilde{\gi}^{ext}_\delta(R)$ is measurable with respect to $\omega \cap A(R/2,2R)$. Therefore, $G^{ext}_\delta(R)$ is measurable with respect to $\omega \cap A(R/2,2R)$.

Similarly, we let $\widetilde{s}^{int}(r,R)$ be the least distance between any pair of endpoints on $\partial B_r$ of two interfaces that go from $\partial B_R$ to $\partial B_r$ and we write $\widetilde{\gi}^{int}_\delta(R) = \lbrace \widetilde{s}^{int}(R,3R/2) \geq 10 \delta R \rbrace$. We write $G^{int}_\delta(R) = \gp^{int}_\delta(R) \cap \widetilde{\gi}^{int}_\delta(R)$. The event $G^{int}_\delta(R)$ is measurable with respect to $\omega \cap A(R/2,2R)$.

\begin{rem}
As noted above, conditioning on $\dense_\delta(R)$ implies nice spatial independence properties. In what follows, we will often work with quads $Q$ and $Q'$ at distance more than $\delta R$ from each other and we will often use implicitly that, if we condition on $\dense_\delta(R)$, then there is no Voronoi cell that intersects the two quads, which implies that the events $\cross(Q)$ and $\cross^*(Q')$ are (conditionally) independent.
\end{rem}

We have the following estimates:

\begin{lem}\label{l.good}
There exists $\epsilon > 0$ such that, for every $\delta  \in (0,1/1000)$ and every $R \in [\delta^{-2},+\infty)$, we have
\begin{equation}\label{e.G+geq}
\Pro \left[ G^{ext}_\delta(R) \right] \geq 1 - \frac{1}{\epsilon} \delta^\epsilon
\end{equation}
and
\begin{equation}\label{e.G-geq}
\Pro \left[ G^{int}_\delta(R) \right] \geq 1 - \frac{1}{\epsilon} \delta^\epsilon \, .
\end{equation}
\end{lem}

\begin{proof} We write only the proof of~\eqref{e.G+geq} since the proof of~\eqref{e.G-geq} is exactly the same. With the same proof as Lemma~\ref{l.dense} and thanks to Proposition~\ref{p.a_lot_of_quads}, we have
\[
\Pro \left[ \dense_\delta(R) \cap \qbc_\delta(R) \right] \geq 1-\grandO{1} \left( \delta^{-2} \exp \left( -\Omega(1) (\delta \, R)^2 \right) + R^{-1} \right) \, .
\]
Since $R \geq \delta^{-2}$, the above is at least
\[
1-\grandO{1} \left( \delta^{-2} \, \exp \left( -\Omega(1) \delta^{-2} \right) + \delta^2 \right) \geq 1-\grandO{1} \delta^{2} \, .
\]
If we apply Lemma~\ref{l.dense} and Proposition~\ref{p.a_lot_of_quads} once again, we obtain that $\Pro \left[ \qbc^{ext}(R) \right] \geq 1-\grandO{1}R^{-1}$. Therefore, with probability at least $1 - \grandO{1} R^{-1}$, we have
\[
\Pro \left[ \qbc^{ext}(R) \cond \eta \cap A(R/2,2R) \right] \geq 3/4 \, .
\]
As a result,
\[
\Pro \left[ \gp^{ext}_\delta(R) \right] \geq 1-\grandO{1} (R^{-1}+\delta^2) \geq 1-\grandO{1} \delta^2 \, .
\]
It only remains to prove that $\Pro \left[ \widetilde{\gi}^{ext}_\delta(R) \right] \geq 1 - \grandO{1} \delta^\epsilon$ for some $\epsilon > 0$. To this purpose, we follow the proof of Lemma~A$.2$ in~\cite{schramm2010quantitative} (which is written for site percolation on $\T$). First, we let $\alpha \subseteq \partial B_R$ be an arc of diameter $R/8$ and we let $Y$ be the set of points in $B_R$ at distance at most $R/8$ from $\alpha$. Let $\alpha_1$ be one of the two arcs in $(\partial Y \cap \partial B_R) \setminus \alpha$. Let $k$ be the number of interfaces crossing from $\partial Y \setminus \partial B_R$ to $\alpha$ and let $\beta_1, \cdots, \beta_k$ be these interfaces ordered in a way that, if $i_1 < i_2$, $\beta_{i_1}$ separates $\alpha_1$ from $\beta_{i_2 }$ in $Y$ (we will say that ``$\beta_{i_2}$ is on the right-hand-side of $\beta_{i_1}$'' and we will write $Y_i$ for the component of $Y \setminus \beta_i$ separated from $\alpha_1$ by $\beta_i$). Let $z_i$ denote the endpoint of $\beta_i$ on $\alpha$. We want to prove that there exists an absolute constant $\epsilon > 0$ such that
\begin{equation}\label{e.wellsep}
\Pro \left[ \forall i \in \lbrace 1, \cdots, {k-1} \rbrace, | z_i - z_{i+1} | \geq 10 \delta R \right] \geq 1-\grandO{1} \delta^\epsilon \, .
\end{equation}

The strategy in~\cite{schramm2010quantitative} is to condition on $\{ i \leq k \}$ and on $\beta_i$, use the fact that the percolation configuration on the right-hand-side of $\beta_i$ remains unbiased and finally conclude thanks to the box-crossing property. The fact that the (conditioned) configuration on the right-hand-side of $\beta_i$ is unbiased is not true in the case of Voronoi percolation since it gives information about the structure of the random tiling.

The strategy we choose is to condition on some $\eta$ such that $\dense_\delta(R) \, \cap \, \widetilde{\qbc}_\delta(R)$ holds, where
\[
\widetilde{\qbc}_\delta (R) := \lbrace \forall Q \in \widetilde{\mathcal{Q}}_{\delta}\left( A(R/2,2R) \right), \, \Prob^\eta \left[ \cross(Q) \right] \geq \widetilde{c}(1) \rbrace
\]
(see Definition~\ref{d.a_lot_of_quads_bis} for the definition of the set of quads $\widetilde{\mathcal{Q}}_{\delta}(D)$; the constant $\widetilde{c}(1)$ is the constant that comes from Proposition~\ref{p.a_lot_of_quads_bis}). Now, since $\eta$ is fixed, if we condition  on $\{ i \leq k \}$ and on $\beta_i$, then the (conditioned) configuration on the right-hand-side of $\beta_i$ remains unbiased. Moreover, the fact that $\widetilde{\qbc}_\delta (R)$ holds implies that we can use the box-crossing properties that are used in the proof of Lemma~A$.2$ of~\cite{schramm2010quantitative}. Finally (and we refer to~\cite{schramm2010quantitative} for more details), we obtain that, for some absolute constant $\widetilde{\epsilon} > 0$ and for any $\eta \in \dense_\delta(R) \, \cap \, \widetilde{\qbc}_\delta (R)$, we have
\[
\Prob^\eta \left[ \forall i \in \lbrace 1, \cdots, {k-1} \rbrace, | z_i - z_{i+1} | \geq 10 \delta R \right] \geq 1-\grandO{1} \delta^{\widetilde{\epsilon}} \, .
\]
(The fact that $\widetilde{\epsilon}$ does not depend on $\delta$ is crucial and comes from the fact that $\widetilde{c}(1)$ does not depend on $\delta$.) Next, note that Lemma~\ref{l.dense} and Proposition~\ref{p.a_lot_of_quads_bis} imply that
\begin{eqnarray*}
\Pro \left[ \dense_\delta(R) \cap \widetilde{\qbc}_\delta (R) \right] & \geq & 1-\grandO{1} \left( \delta^{-2} \exp \left( -\Omega(1) (\delta \, R)^2 \right) + R^{-1} \right)\\
& \geq & 1-\grandO{1}\delta^{2} \, .
\end{eqnarray*}
So, we have obtained~\eqref{e.wellsep} (with $\epsilon = \widetilde{\epsilon} \wedge 2$). It is not difficult to see (by choosing an appropriate covering of $\partial B_R$ by $\grandO{1}$ arcs $\alpha$) that this implies that
\[
\Pro \left[ \widetilde{\gi}^{ext}_\delta(R) \right] = \Pro \left[ \widetilde{s}^{ext}(3R/4,R) \geq 10 \delta R \right] \geq 1 - \grandO{1} \delta^\epsilon \, .
\]
\end{proof}

\subsubsection{Extension of the arms when the configuration is good}

Lemma~\ref{l.extension} below is the analogue of Lemma~A$.3$ of~\cite{schramm2010quantitative} and roughly says that if the $4$-arm event holds at some scale and if the configuration is good at this scale then we can extend the four arms to a larger scale with non-negligible probability. If $\delta \in (0,1/1000)$, $R \in [\delta^{-2},+\infty)$ and $r \in [1,R]$, we write
\[
g_{4,\delta}^{ext}(r,R) = \Pro \left[ \text{\textbf{A}}_4(r,R) \cap G_\delta^{ext}(R) \right] \, .
\]
Similarly, if $\delta \in (0,1/1000)$, $r \in [\delta^{-2},+\infty)$ and $R \in [r,+\infty)$, we write
\[
g_{4,\delta}^{int}(r,R) = \Pro \left[ \text{\textbf{A}}_4(r,R) \cap G_{\delta}^{int}(r) \right] \, .
\]

\begin{rem}\label{r.poly_et_delta_bar}
Note that~\eqref{e.poly} implies that there exists a constant $c > 0$ such that, for all $1 \leq \rho_1 \leq \rho_2$ that satisfy $\rho_1 \geq \rho_2/4^3$, we have $\alpha^{an}_4(\rho_1,\rho_2) \geq c$. Together with Lemma~\ref{l.good}, this implies that, if $\delta \in (0,1/1000)$ is sufficiently small, then for all $R  \in [\delta^{-2},+\infty)$ and all $r \in [R/4^3,R]$, we have
\[
g^{ext}_{4,\delta}(r,R) \geq c/2 \, .
\]
Similarly, if $\delta \in (0,1/1000)$ is sufficiently small, then for all $r \in [\delta^{-2},+\infty)$ and all $R \in [r,4^3r]$, we have
\[
g_{4,\delta}^{int}(r,R) \geq c/2 \, .
\]
\end{rem}

\begin{lem}\label{l.extension}
There exists $\overline{\delta} \in (0,1/1000)$ such that, for any $\delta \in (0,1/1000)$, there is some constant $a=a(\delta) \in (0,1)$ satisfying the following:
\begin{enumerate}
\item For every $R \in [(\overline{\delta} \vee \delta)^{-2},+\infty)$ and every $r \in [1,R/4]$, we have
\begin{equation}\label{e.gplus}
g_{4,\overline{\delta}}^{ext}(r,4R) \geq a \, g^{ext}_{4,\delta}(r,R) \, .
\end{equation}
\item For every $r \in [4 (\overline{\delta} \vee \delta)^{-2} ,+\infty)$ and every $R \in [4r,+\infty)$, we have
\begin{equation}\label{e.gmoins}
g_{4,\overline{\delta}}^{int}(r/4,R) \geq a \, g^{int}_{4,\delta}(r,R) \, .
\end{equation}
\end{enumerate}
Moreover, we can (and do) assume that $\overline{\delta}$ is sufficiently small so that Remark~\ref{r.poly_et_delta_bar} holds with $\delta=\overline{\delta}$.
\end{lem}



\begin{proof}[Proof of Lemma~\ref{l.extension}]
Let us first prove~\eqref{e.gplus}. Let $\overline{\delta} \in (0,1/1000)$ to be determined later and consider $R$, $r$ and $\delta$ as in the statement of the lemma. We write $\Prob^\eta_{B_{2R}}$ for the probability measure $\Pro$ conditioned on $\eta \cap B_{2R}$. Note that this is the probability measure obtained by coloring $\eta \cap B_{2R}$ uniformly and by sampling (independently of the coloring of $\eta \cap B_{2R}$) a colored Poisson point process in $\R^2 \setminus B_{2R}$ of intensity $\text{Leb}_{\R^2 \setminus B_{2R}} \otimes \left( \frac{\delta_{-1}+\delta_1}{2} \right)$.\\

Fix some $\eta \in \gp^{ext}_\delta(R)$ such that $\Prob^\eta_{B_{2R}} \left[ \arm_4(r,R) \cap \lbrace \widetilde{s}^{ext}(r,R) \geq 10 \delta  R \rbrace  \right] > 0$ and
write $\beta_0, \cdots, \beta_{k-1}$ for the interfaces that cross $A(r,R)$ in counter-clockwise order. We assume that the right-hand-side of $\beta_0$ (if one goes from $\partial B_r$ to $\partial B_R$) is black. First, we work under the following conditional probability measure:
\[
\nu^\eta_{r,R,(\beta_j)_j} := \Prob^\eta_{B_{2R}} \left[ \cdot \cond \arm_4(r,R) \cap \gp^{ext}_\delta(R) \cap \lbrace \widetilde{s}^{ext}(r,R) \geq 10 \delta R \rbrace, \, \beta_0, \cdots, \beta_{k-1} \right] \, .
\]
We keep such an $\eta$ fixed until we explicitly say that we take the expectation under $\Pro$ (see below~\eqref{e.key_lemma_ext}). Let us define four rectangles $Q^{ext}(R,0), \cdots, Q^{ext}(R,3)$ (which belong to the set of quads $\mathcal{Q}_{1/100} \left( A(R,4R) \right)$ from Definition~\ref{d.a_lot_of_quads}) in Figure~\ref{f.Qplus(R,i)}.

\begin{figure}[!h]
\begin{center}
\includegraphics[scale=0.50]{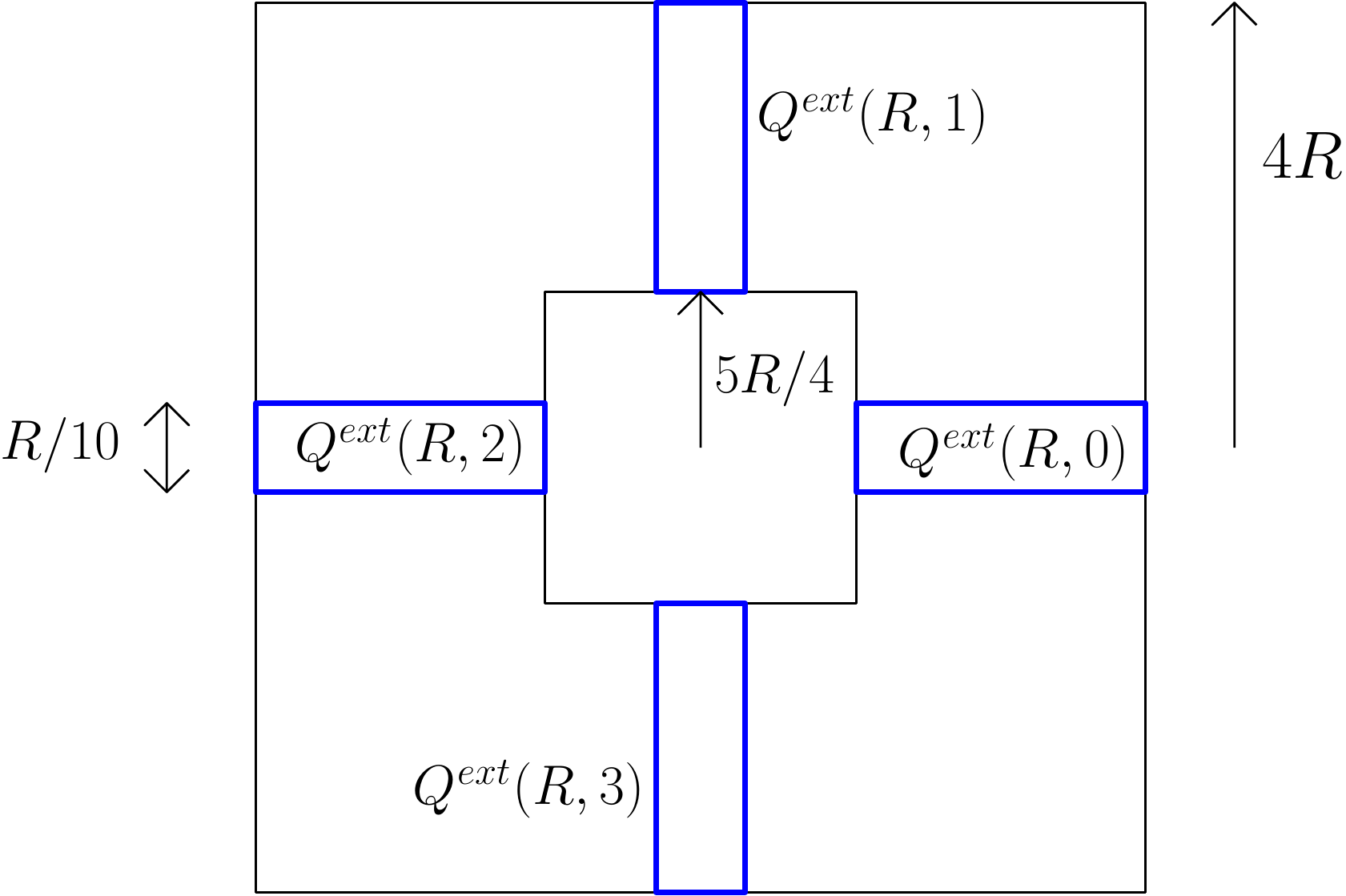}
\end{center}
\caption{The quads $Q^{ext}(R,0), \cdots, Q^{ext}(R,3)$.}
\label{f.Qplus(R,i)}
\end{figure}

It is not difficult to see that we can choose four quads $Q(\beta_i) \in  \mathcal{Q}_\delta \left( A(3R/4,3R/2) \right)$, $i \in \lbrace 0, \cdots, 3 \rbrace$, such that: (a) the intersection of $Q(\beta_i)$ and $\beta_{i-1} \cup \beta_{i}$ is one of the two distinguished sides of $Q(\beta_i)$, (b) $Q(\beta_i) \cap B_R$ is in the region between $\beta_{i-1}$ and $\beta_i$, (c) if $Q(\beta_i)$ is crossed, then $Q^{ext}(R,i)$ is crossed widthwise, (d) if $0 \leq i \neq j \leq 3$, then there is no Voronoi cell that intersects both $Q(\beta_i)$ and $Q^{ext}(R,j)$, (e) if $0 \leq i \neq j \leq 3$, then there is no Voronoi cell that intersects both $Q(\beta_i)$ and $Q(\beta_j)$. See Figure~\ref{f.Q(beta)}.

\begin{figure}[!h]
\begin{center}
\includegraphics[scale=1]{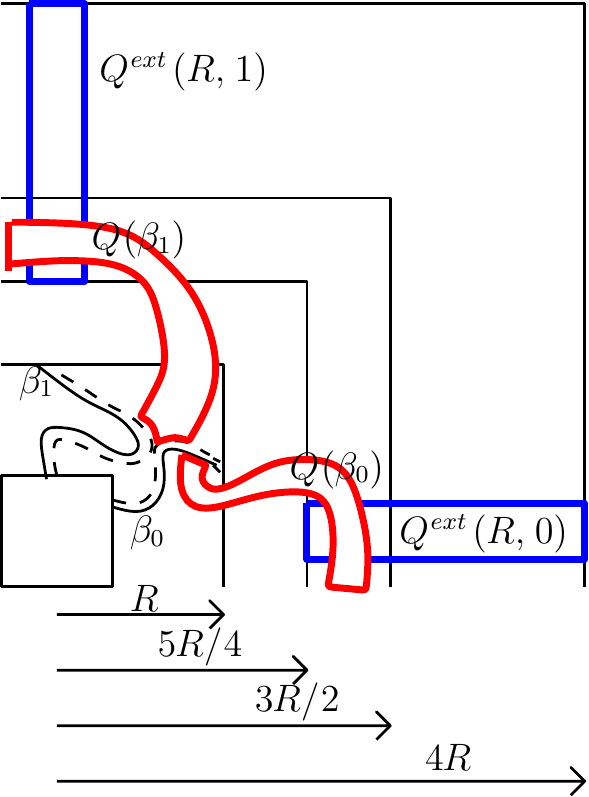}
\end{center}
\caption{The quads $Q(\beta_1)$ and $Q(\beta_2)$.}
\label{f.Q(beta)}
\end{figure}

Note that, since the quads $Q(\beta_i)$ belong to the set $\mathcal{Q}_\delta \left( A(3R/4,3R/2) \right)$, then the probability under $\nu^\eta_{r,R,(\beta_j)_j}$ that $Q(\beta_i)$ is crossed is at least $c(\delta,1)$, where $c(\delta,1)$ is the constant of Proposition~\ref{p.a_lot_of_quads}. (Note that we have implicitly used the (quenched) Harris-FKG inequality since conditioning on $(\beta_j)_j$ affects the percolation process as follows: if the right-hand-side of $\beta_j$ is black (respectively white) then there is a black (respectively white) crossing from $\beta_j$ to $\beta_{j-1}$.) Now, let $F = F(\beta_0, \cdots, \beta_{k-1})$ denote the event that there are black crossings in $Q(\beta_0)$ and $Q(\beta_2)$ and white crossings in $Q(\beta_1)$ and $Q(\beta_3)$. We have
\begin{equation}\label{e.a'delta}
\nu^\eta_{r,R,(\beta_j)_j} \left[ F \right] \geq c(\delta,1)^4 \, .
\end{equation}
Our next goal is to prove the following:
\begin{equation}\label{e.3/4a'delta}
\nu^\eta_{r,R,(\beta_j)_j} \left[ \qbc^{ext}(R) \cond F \right] \geq 3/4 \, .
\end{equation}
To this purpose, remember that
\[
\gp^{ext}_\delta(R) = \dense_\delta(R) \cap \qbc_\delta(R) \cap \left\lbrace \Pro \left[ \qbc^{ext}(R) \cond \eta \cap A(R/2,2R) \right] \geq 3/4 \right\rbrace \, .
\]
Since $\sigma \left( \qbc^{ext}(R), \eta \cap A(R/2,2R) \right)$ is independent of $\eta \cap B_{R/2}$, we have
\[
\Pro \left[ \qbc^{ext}(R) \cond \eta \cap A(R/2,2R) \right] = \Prob_{B_{2R}}^\eta \left[ \qbc^{ext}(R) \right] \, .
\]
Moreover, i) $\nu^\eta_{r,R,(\beta_j)_j} \left[ \cdot \, | \, F \right]$ is the probability measure $\Prob^\eta_{B_{2R}}$ conditioned on $\gp^{ext}_\delta(R)$ and on other events which are measurable with respect to $\omega \cap B_{2R}$ and ii) $\qbc^{ext}(R)$ is $\Prob^\eta_{B_{2R}}$-independent of $\omega \cap B_{2R}$. This implies~\eqref{e.3/4a'delta}.\\

We are now in shape to extend the arms to scale $4R$ since $\qbc^{ext}(R)$ gives quenched box crossing estimates for enough quads in $A(R,4R)$. If $F$ holds and if there are black crossings of $Q^{ext}(R,0)$ and $Q^{ext}(R,2)$ and white crossings of $Q^{ext}(R,1)$ and $Q^{ext}(R,3)$, then $\text{\textbf{A}}_4(r,4R)$ holds. As a result, the (quenched) Harris-FKG inequality implies that there exists an absolute constant $c' > 0$ such that
\[
\nu^\eta_{r,R,(\beta_j)_j} \left[ \text{\textbf{A}}_4(r,4R) \cond F \cap \qbc^{ext}(R), \, \eta \right] \geq c' \, ,
\]
hence,
\begin{equation}\label{e.extendeq}
\nu^\eta_{r,R,(\beta_j)_j} \left[ \text{\textbf{A}}_4(r,4R) \cond F \cap \qbc^{ext}(R) \right] \geq c' \, .
\end{equation}
Let us use once again that $\nu^\eta_{r,R,(\beta_j)_j}$ is the probability measure $\Prob_{B_{2R}}^\eta$ conditioned on events measurable with respect to $\omega \cap B_{2R}$. Let us also use that (under $\nu^\eta_{r,R,(\beta_j)_j}$) the event $F$ is measurable with respect to $\omega \cap B_{2R}$. Moreover, $G^{ext}_{\bar{\delta}}(4R)$ is measurable with respect to $\omega \setminus B_{2R}$. Together with Lemma~\ref{l.good}, this implies that
\begin{equation}\label{e.GC4Req}
\nu^\eta_{r,R,(\beta_j)_j} \left[ G^{ext}_{\overline{\delta}}(4R) \cond F \right] \geq 1-\frac{1}{\epsilon} \overline{\delta}^\epsilon.
\end{equation}
If we combine~\eqref{e.3/4a'delta}, \eqref{e.extendeq} and~\eqref{e.GC4Req}, we obtain that
\[
\nu^\eta_{r,R,(\beta_j)_j} \left[ \text{\textbf{A}}_4(r,4R) \cap G^{ext}_{\overline{\delta}}(4R) \cond F \right] \geq 3c'/4 - \frac{1}{\epsilon}\overline{\delta}^\epsilon.
\]
We choose $\overline{\delta}$ sufficiently small so that $3c'/4 - \frac{1}{\epsilon}\overline{\delta}^\epsilon \geq c'/2$ (here the fact that $c'$ does not depend on $\delta$ is crucial). Now, by combining the above inequality with~\eqref{e.a'delta} we obtain that
\begin{equation}\label{e.key_lemma_ext}
\nu^\eta_{r,R,(\beta_j)_j} \left[ \text{\textbf{A}}_4(r,4R) \cap G^{ext}_{\overline{\delta}}(4R) \right] \geq \frac{c(\delta,1)^4 \, c'}{2} \, .
\end{equation}
If we take the expectation under $\Prob^\eta_{B_{2R}}$ and then under $\Pro$, we obtain that
\begin{multline*}
g^{ext}_{4,\overline{\delta}}(r,4R) = \Pro \left[ \text{\textbf{A}}_4(r,4R) \cap G^{ext}_{\overline{\delta}}(4R) \right]\\
\geq \frac{c(\delta,1)^4 \, c'}{2} \, \Pro \left[ \text{\textbf{A}}_4(r,R) \cap \gp^{ext}_\delta(R) \cap \lbrace \widetilde{s}^{ext}(r,R) \geq 10 \delta R \rbrace \right].
\end{multline*}
Note that, if $1 \leq r_1 \leq r_2 \leq r_3$, then $\widetilde{s}^{ext}(r_1,r_3) \geq \widetilde{s}^{ext}(r_2,r_3)$, hence,
\[
\Pro \left[ \arm_4(r,R) \cap \gp^{ext}_\delta(R) \cap \lbrace \widetilde{s}^{ext}(r,R) \geq 10 \delta R \rbrace \right] \geq g^{ext}_{4,\delta}(r,R) \, .
\]
Finally, we have obtained~\eqref{e.gplus} (with $a = a(\delta) = c(\delta,1)^4 \, c'/2$).\\

Note that we have obtained the following more precise result: Let $\widetilde{\arm}^{ext}_4(r,R)$ denote the event that there are four arms of alternating colors $\gamma_0, \cdots, \gamma_3$ from $\partial B_r$ to $\partial B_R$ such that $\gamma_i \cap A(R/2,R) \subseteq Q^{ext}(R/4,i)$ (see Figure~\ref{f.Qplus(R,i)} for the definition of these rectangles). Let $\delta$, $r$ and $R$ be as in the statement of Item~$1$ of Lemma~\ref{l.extension}. Then,
\[
\Pro \left[ \widetilde{\arm}_4^{ext}(r,4R) \cap G^{ext}_{\overline{\delta}}(4R) \right] \geq a \, 
g^{ext}_{4,\delta}(r,R) \, .
\]
Actually, if we follow the proof we can see that we also have the following: Let $F_R$ be an event measurable with respect to $\omega \setminus B_{2R}$ such that $\Pro \left[ F_R \right] \geq 1-c'/4$. Then,
\[
\Pro \left[ \widetilde{\arm}_4^{ext}(r,4R) \cap G^{ext}_{\overline{\delta}}(4R) \cap F_R \right] \geq \frac{a}{2} \, g^{ext}_{4,\delta}(r,R) \, .
\]

The proof of~\eqref{e.gmoins} is exactly the same. As in the case of~\eqref{e.gplus}, we can also obtain a stronger result: Let $Q^{int}(r,0), \cdots, Q^{int}(r,3)$ be the four rectangles defined on Figure~\ref{f.Qmoins(r,i)} and write $\widetilde{\arm}^{int}_4(r,R)$ for the event that there are four arms of alternating colors $\gamma_0, \cdots, \gamma_3$ from $\partial B_r$ to $\partial B_R$ such that $\gamma_i \cap A(r,2r) \subseteq Q^{int}(4r,i)$. If $\overline{\delta}$ is sufficiently small and if $\delta$, $r$ and $R$ are as in Item~$2$ of Lemma~\ref{l.extension}, then the following holds: There exists $a=a(\delta) > 0$ and $c' > 0$ such that, for every event $F_r$ measurable with respect to $\omega \cap B_{r/2}$ that satisfies $\Pro \left[ F_r \right] \geq 1-c'$, we have
\begin{equation}\label{e.extension_better}
\Pro \left[ \widetilde{\arm}^{int}_4(r/4,R) \cap G^{int}_{\overline{\delta}}(r/4) \cap F_r \right] \geq a \, g_{4,\delta}^{int}(r,R) \, .
\end{equation}
\end{proof}

\begin{figure}[!h]
\begin{center}
\includegraphics[scale=0.50]{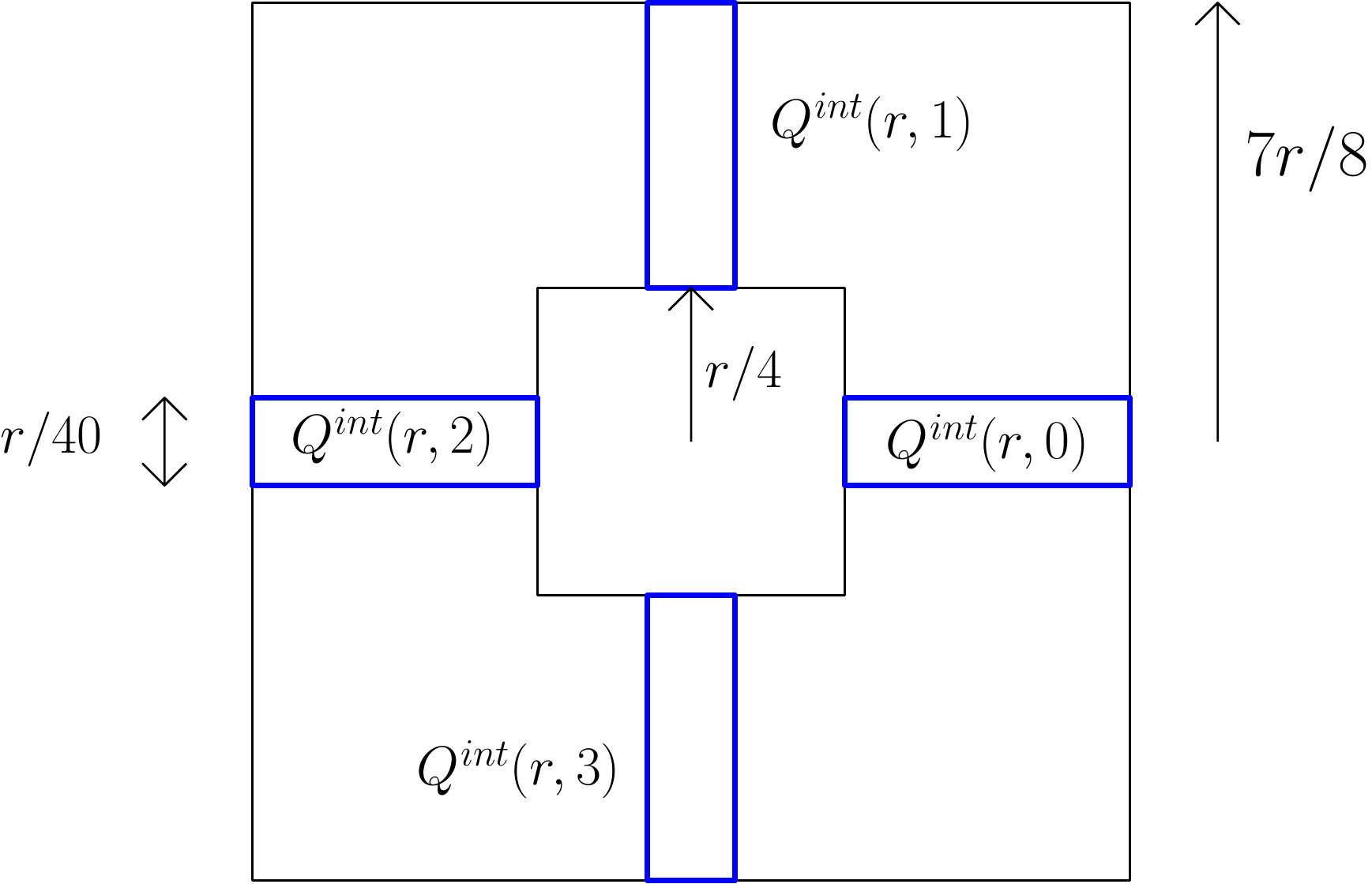}
\end{center}
\caption{The quads $Q^{int}(r,0), \cdots, Q^{int}(r,3)$.}
\label{f.Qmoins(r,i)}
\end{figure}

\subsubsection{The probability to look good if the $4$-arm event holds}

We now consider the following events: 
\begin{eqnarray*}
\widehat{\text{\textbf{A}}}^{ext}_4(r,R) & = & \left\lbrace \Pro \left[ \text{\textbf{A}}_4(r,R) \cond \omega \cap B_R \right] > 0 \right\rbrace \, ,\\
\widehat{\text{\textbf{A}}}^{int}_4(r,R) & = & \left\lbrace \Pro \left[ \text{\textbf{A}}_4(r,R) \cond \omega \setminus B_r \right] > 0 \right\rbrace \, ,
\end{eqnarray*}
and the two following quantities:
\[
f_4^{ext}(r,R) = \Pro \left[ \widehat{\text{\textbf{A}}}^{ext}_4(r,R) \right] \, ; \, f_4^{int}(r,R) = \Pro \left[ \widehat{\text{\textbf{A}}}^{int}_4(r,R) \right] \, .
\]
We want to prove that the quantities $\alpha^{an}_4(r,R)$, $g^{ext}_{4,\bar{\delta}}(r,R)$, $g^{int}_{4,\bar{\delta}}(r,R)$, $f^{ext}_4(r,R)$ and $f^{int}_4(r,R)$ are of the same order. We have the following result (where $\overline{\delta}$ is the constant of Lemma~\ref{l.extension}):

\begin{lem}\label{l.looksgood}
There exist $C_1 \in [1,+\infty)$ and $\overline{r} \in [\overline{\delta}^{-2},+\infty)$ such that, for every $r \in [\overline{r},+\infty)$ and $R \in [16r,+\infty)$,
\begin{equation}\label{e.looksgoodplus}
g^{ext}_{4,\overline{\delta}}(r,R) \geq f^{ext}_4(r,R)/C_1 \, ,
\end{equation}
and
\begin{equation}\label{e.looksgoodmoins}
g^{int}_{4,\overline{\delta}}(r,R) \geq f^{int}_4(r,R)/C_1 \, .
\end{equation}
\end{lem}
We have the following corollary (which is a direct consequence of Lemma~\ref{l.looksgood} and Remark~\ref{r.poly_et_delta_bar}):
\begin{cor}\label{c.looksgood}
There exists a constant $C_2 \in [1,+\infty)$ such that, for every $r \in [\overline{r},+\infty)$ and every $R \in [r,+\infty)$,
\[
g^{ext}_{4,\overline{\delta}}(r,R) \leq \alpha^{an}_4(r,R) \leq f^{ext}_4(r,R) \leq C_2 \, g^{ext}_{4,\overline{\delta}}(r,R) \, ,
\]
and
\[
g^{int}_{4,\overline{\delta}}(r,R) \leq \alpha^{an}_4(r,R) \leq f^{int}_4(r,R) \leq C_2 \, g^{int}_{4,\overline{\delta}}(r,R) \, .
\]
\end{cor}

\begin{proof}[Proof of Lemma~\ref{l.looksgood}]
We only prove~\eqref{e.looksgoodplus} since the the proof of~\eqref{e.looksgoodmoins} is essentially the same. Let $\delta \in (0,\overline{\delta})$ to be chosen later, let $\overline{r} = 4^3 \,\delta^{-2}$, and let $r$ and $R$ be as in the statement of the lemma. First, note that if $\dense_\delta(R/4)$ holds then every $x \in \eta$ whose Voronoi cell intersects $A(r,R/4)$ is in $B_R$. Hence, $\widehat{\arm}_4^{ext}(r,R) \cap \dense_\delta(R/4) \subseteq \arm_4(r,R/4)$. Remember that $\dense_\delta(R/4) \subseteq G^{ext}_\delta(R/4)$. As a result,
\begin{eqnarray*}
\widehat{\text{\textbf{A}}}_4^{ext}(r,R) & \subseteq & \big( \text{\textbf{A}}_4(r,R/4) \cap G^{ext}_\delta(R/4) \big) \cup \big( \widehat{\text{\textbf{A}}}_4^{ext}(r,R) \setminus G^{ext}_\delta(R/4) \big)\\
& \subseteq & \big( \text{\textbf{A}}_4(r,R/4) \cap G_{\delta}^{ext}(R/4) \big) \cup \big( \widehat{\text{\textbf{A}}}_4^{ext}(r,R/16) \setminus G^{ext}_\delta(R/4) \big) \, .
\end{eqnarray*}
As a result, $f_4^{ext}(r,R)$ is smaller than or equal to
\begin{align*}
& \Pro \left[ \text{\textbf{A}}_4(r,R/4) \cap G^{ext}_\delta(R/4) \right] + \Pro \left[ \widehat{\text{\textbf{A}}}_4^{ext}(r,R/16) \setminus G^{ext}_\delta(R/4) \right]\\
& = g^{ext}_{4,\delta}(r,R/4) + \Pro \left[  \widehat{\text{\textbf{A}}}_4^{ext}(r,R/16) \setminus G^{ext}_\delta(R/4) \right] \, .
\end{align*}

Remember that $G^{ext}_\delta(R/4)$ is measurable with respect to $\omega \cap A(R/8,R/2)$ and that $\widehat{\text{\textbf{A}}}_4^{ext}(r,R/16)$ is measurable with respect to $\omega \cap B_{R/16}$. Hence, by spatial independence, the above equals
\begin{equation}\label{e.key_in_looksgood}
g^{ext}_{4,\delta}(r,R/4) + f^{ext}_4(r,R/16) \cdot \Pro \left[ \neg G^{ext}_\delta(R/4) \right] \, .
\end{equation}
Since $R/4 \geq \delta^{-2}$, Lemma~\ref{l.good} implies that
\[
f^{ext}_4(r,R) \leq g^{ext}_{4,\delta}(r,R/4) + \frac{1}{\epsilon} \delta^\epsilon f^{ext}_4(r,R/16) \, .
\]
By repeating the above argument, we obtain that $f^{ext}_4(r,R)$ is at most
\[
\sum_{i=0}^{l-1} \left( \left( \frac{1}{\epsilon} \delta^\epsilon \right)^i g^{ext}_{4,\delta} \left( r,\frac{R}{4 \cdot 16^i} \right) \right) + \left( \frac{1}{\epsilon} \delta^\epsilon \right)^l \, ,
\]
where $l = \lfloor \log_{16} \left( R/r \right) \rfloor$. Lemma~\ref{l.extension} then implies that that the above is at most
\begin{equation}\label{e.sumg_1}
\sum_{i=0}^{l-1} \left( \left( \frac{1}{\epsilon} \delta^\epsilon \right)^i a(\delta)^{-1} a(\overline{\delta})^{-2i} \, g^{ext}_{4,\overline{\delta}}(r,R) \right) + \left( \frac{1}{\epsilon} \delta^\epsilon  \right)^l \, .
\end{equation}
Lemma~\ref{l.extension} also implies the following inequality:
\begin{equation}\label{e.sumg_2}
g^{ext}_{4,\overline{\delta}}(r,R) \geq a(\overline{\delta})^{2l-1} g^{ext}_{4,\overline{\delta}} \left( r,\frac{R}{4 \cdot 16^{l-1}} \right) \, . 
\end{equation}
Remember Remark~\ref{r.poly_et_delta_bar}: there exists an absolute constant $c > 0$ such that
\begin{equation}\label{e.claim}
g^{ext}_{4,\overline{\delta}} \left( r,\frac{R}{4 \cdot 16^{l-1}} \right) \geq c \, .
\end{equation}
Let us end the proof: We choose $\delta \in (1,\overline{\delta})$ small enough so that $\frac{1}{\epsilon} \delta^\epsilon a(\overline{\delta})^{-2} \leq 1/2$. If we combine~\eqref{e.sumg_1}, \eqref{e.sumg_2} and~\ref{e.claim} we obtain that
\begin{eqnarray*}
f_4^{ext}(r,R) & \leq & g^{ext}_{4,\overline{\delta}}(r,R) \sum_{i=0}^{+\infty} \left( \left( \frac{1}{\epsilon} \delta^\epsilon \right)^i a(\delta)^{-1} a(\overline{\delta})^{-2i} \right) + g^{ext}_{4,\overline{\delta}}(r,R) \left( \frac{1}{\epsilon} \delta^\epsilon  \right)^l a(\overline{\delta})^{-(2l-1)} (c')^{-1} \\ 
& \leq & g^{ext}_{4,\overline{\delta}}(r,R) \left( \frac{2}{a(\delta)} + \frac{a(\overline{\delta})}{  c} \right) \, ,
\end{eqnarray*}
which ends the proof.
\end{proof}

\subsubsection{Proof of the quasi-multiplicativity property}

We are now in shape to prove Proposition~\ref{p.quasieven}. We first prove it for $r_1$ sufficiently large and we prove separately the left-hand and right-hand inequalities. Below,  $\overline{r}$ is the constant of Lemma~\ref{l.looksgood}. Remember that $\overline{r} \geq \overline{\delta}^{-2}$ where $\overline{\delta}$ is the constant of Lemma~\ref{l.extension}.
\begin{proof}[Proof of the left-hand-inequality of Proposition~\ref{p.quasieven} in the case $r_1 \geq \overline{r}$]
We have
\begin{eqnarray*}
\alpha^{an}_4(r_1,r_3) & \leq & \Pro \left[ \text{\textbf{A}}_4(r_1,r_2) \cap \text{\textbf{A}}_4(r_2,r_3) \right]\\
& \leq & \Pro \left[ \widehat{\text{\textbf{A}}}_4^{ext}(r_1,r_2) \cap \widehat{\text{\textbf{A}}}_4^{int}(r_2,r_3) \right]\\
& = & \Pro \left[ \widehat{\text{\textbf{A}}}^{ext}_4(r_1,r_2) \right] \cdot \Pro \left[ \widehat{\text{\textbf{A}}}^{int}_4(r_2,r_3) \right] \, ,
\end{eqnarray*}
by spatial independence. The above inequality can be rewritten as follows:
\[
\alpha^{an}_4(r_1,r_3) \leq f^{ext}_4(r_1,r_2) \, f^{int}_4(r_2,r_3) \, ,
\]
so Corollary~\ref{c.looksgood} implies the desired result.
\end{proof}

\begin{proof}[Proof of the right-hand-inequality of Proposition~\ref{p.quasieven} in the case $r_1 \geq 16\overline{r}$]
We distinguish between four cases:
\begin{enumerate}
\item Assume that $r_1 \geq r_2/16$ and $r_2 \geq r_3/16$. Then, this is a direct consequence of~\eqref{e.poly}.
\item Assume that $r_1 \geq r_2/16$ and $r_2 \leq r_3/16$.
By Corollary~\ref{c.looksgood}, we have
\[
\alpha^{an}_4(r_1,r_2) \, \alpha^{an}_4(r_2,r_3) \leq \alpha^{an}_4(r_2,r_3) \leq \grandO{1} g^{int}_{4,\overline{\delta}}(r_2,r_3) \, .
\]
By applying Lemma~\ref{l.extension} (here, we use that $r_2 \geq 16\overline{r} \geq 16\overline{\delta}^{-2}$ since $r_1 \geq 16\overline{r}$), we obtain that the above is at most $\grandO{1} \alpha^{an}_4(r_2/16,r_3)$ (which is at most $\grandO{1} \alpha^{an}_4(r_1,r_3)$ since $r_1 \geq r_2/16$).
\item The case ``$r_1 \leq r_2/16$ and $r_2 \geq r_3/16$'' is treated similarly.
\item Now, we treat the case $r_1 \leq r_2/16$ and $r_2 \leq r_3/16$. First, we write the following simple inequality:
\[
\alpha^{an}_4(r_1,r_2) \, \alpha^{an}_4(r_2,r_3) \leq \alpha^{an}_4(r_1,\frac{r_2}{3}) \, \alpha^{an}_4(3r_2,r_3) \, .
\]

Corollary~\ref{c.looksgood} implies that
\begin{equation}\label{e.1st_step_gluing}
\alpha^{an}_4(r_1,\frac{r_2}{3}) \, \alpha^{an}_4(3r_2,r_3) \leq \grandO{1} \; g^{ext}_{4,\overline{\delta}}(r_1,\frac{r_2}{3}) \, g^{int}_{4,\overline{\delta}}(3r_2,r_3) \, .
\end{equation}
 
Now, the proof is very similar to the one of Lemma~\ref{l.extension}. However, we have to be a little more careful because we have to deal with the interactions between scales $r_2/3$ and $3r_2$. First, as it is explained in the paragraph below~\eqref{e.3/4a'delta}, we can write the events $\gp^{ext}_{\delta}(R)$ (and similarly $\gp^{int}_\delta(R)$) a little differently. More precisely, we have
\begin{multline}\label{e.key_end_1}
\gp^{ext}_{\overline{\delta}}(r_2/3)\\ = \dense_{\overline{\delta}}(r_2/3) \cap \qbc_{\overline{\delta}}(r_2/3) \cap \left\lbrace \Pro \left[ \qbc^{ext}(r_2/3) \cond \eta \cap B_{2r_2/3} \right] \geq 3/4 \right\rbrace
\end{multline}
and similarly
\begin{multline}\label{e.key_end_2}
\gp^{int}_{\overline{\delta}}(3r_2)\\ = \dense_{\overline{\delta}}(3r_2) \cap \qbc_{\overline{\delta}}(3r_2) \cap \left\lbrace \Pro \left[ \qbc^{int}(3r_2) \cond \eta \setminus B_{3r_2/2} \right] \geq 3/4 \right\rbrace \, .
\end{multline}
Since $\qbc^{ext}(\cdot)$ and $\qbc^{int}(\cdot)$ do not depend on the coloring, we actually have
\[
\Pro \left[ \qbc^{ext}(r_2/3) \cond \eta \cap B_{2r_2/3} \right] = \Pro \left[ \qbc^{ext}(r_2/3) \cond \omega \cap B_{2r_2/3} \right] \text{ and }
\]
\[
\Pro \left[ \qbc^{int}(3r_2) \cond \eta \setminus B_{3r_2/2} \right] = \Pro \left[ \qbc^{int}(3r_2) \cond \omega \setminus B_{3r_2/2} \right] \, .
\]
As a result,~\eqref{e.1st_step_gluing} can be rewritten as follows:
\begin{align*}
& \alpha^{an}_4(r_1,\frac{r_2}{3}) \, \alpha^{an}_4(3r_2,r_3)\\
& \leq \grandO{1} \Pro \Big[ \arm_4(r_1,r_2/3) \cap \widetilde{\gi}^{ext}_{\overline{\delta}}(r_2/3) \cap \dense_{\overline{\delta}}(r_2/3) \cap \qbc_{\overline{\delta}}(r_2/3)\\
& \hspace{2em} \cap \left\lbrace \Pro \left[ \qbc^{ext}(r_2/3) \cond \omega \cap B_{2r_2/3} \right] \geq 3/4 \right\rbrace \Big]\\
& \hspace{1em} \times \Pro \Big[ \arm_4(3r_2,r_3) \cap \widetilde{\gi}^{int}_{\overline{\delta}}(3r_2) \cap \dense_{\overline{\delta}}(3r_2) \cap \qbc_{\overline{\delta}}(3r_2)\\
& \hspace{3em} \cap \left\lbrace \Pro \left[ \qbc^{int}(3r_2) \cond \omega \setminus B_{3r_2/2} \right] \geq 3/4 \right\rbrace \Big] \, .
\end{align*}
We need the following lemma:
\begin{lem}\label{l.key_in_qm}
Let $\mathcal{F}$ and $\mathcal{G}$ be two sub-$\sigma$-algebras, let $A_1 \in \mathcal{F}$, $A_2 \in \mathcal{G}$, and let $B_1$ and $B_2$ be two events such that $\sigma(B_1,\mathcal{F})$ is independent of $\mathcal{G}$ and $\sigma(B_2, \mathcal{G})$ is independent of $\mathcal{F}$. Then,
\[
\Pro \left[ A_1 \cap B_1 \cap A_2 \cap B_2 \right] \geq \frac{1}{2} 
\Pro \left[ A_1 \cap \left\lbrace \Pro \left[ B_1 \cond \mathcal{F} \right] \geq 3/4 \right\rbrace \right] \Pro \left[ A_2 \cap \left\lbrace \Pro \left[ B_2 \cond \mathcal{G} \right] \geq 3/4 \right\rbrace \right] \, .
\]
\end{lem}
\begin{proof}
We have
\begin{align*}
& \Pro \left[ A_1 \cap B_1 \cap A_2 \cap B_2 \right]\\
& \geq \E \left[ \un_{A_1 \cap B_1 \cap A_2 \cap B_2} \un_{\{ \Pro \left[ B_1 \, | \, \mathcal{F} \vee \mathcal{G} \right] \geq 3/4 \}} \un_{\{ \Pro \left[ B_2 \, | \, \mathcal{F}\vee \mathcal{G} \right] \geq 3/4 \}} \right]\\
& = \E \left[ \un_{A_1 \cap A_2} \E \left[  \un_{B_1 \cap B_2} \cond \mathcal{F}\vee \mathcal{G} \right] \un_{\{ \Pro \left[ B_1 \, | \, \mathcal{F}\vee \mathcal{G} \right] \geq 3/4 \}} \un_{\{ \Pro \left[ B_2 \, | \, \mathcal{F}\vee \mathcal{G} \right] \geq 3/4 \}}  \right]\\
& \geq \frac{1}{2} \E \left[  \un_{A_1} \un_{\{ \Pro \left[ B_1 \, | \, \mathcal{F}\vee \mathcal{G} \right] \geq 3/4 \}} \un_{A_2} \un_{\{ \Pro \left[ B_2 \, | \, \mathcal{F}\vee \mathcal{G} \right] \geq 3/4 \}} \right] \, .
\end{align*}
But, since $\sigma(B_1,\mathcal{F})$ is independent of $\mathcal{G}$, we have
\[
\Pro \left[ B_1 \cond \mathcal{F}\vee \mathcal{G} \right] = \Pro \left[ B_1  \cond \mathcal{F} \right]  \, .
\]
Similarly,
\[
\Pro \left[ B_2 \cond \mathcal{F}\vee \mathcal{G} \right] = \Pro \left[ B_2  \cond \mathcal{G} \right] \, .
\]
This implies the result since $\mathcal{F}$ is independent of $\mathcal{G}$.
\end{proof}
If we apply this lemma to $\mathcal{F} = \sigma(\omega \cap B_{2r_2/3})$, $\mathcal{G} = \sigma(\omega \setminus B_{3r_2/2})$, $A_1 = \arm_4(r_1,r_2/3) \cap \widetilde{\gi}^{ext}_{\overline{\delta}}(r_2/3) \cap \dense_{\overline{\delta}}(r_2/3) \cap \qbc_{\overline{\delta}}(r_2/3)$, $A_2 = \arm_4(3r_2,r_3) \cap \widetilde{\gi}^{int}_{\overline{\delta}}(3r_2) \cap \dense_{\overline{\delta}}(3r_2) \cap \qbc_{\overline{\delta}}(3r_2)$, $B_1 = \qbc^{ext}(r_2/3)$ and $B_2 = \qbc^{int}(3r_2)$, we obtain that
\begin{multline*}
\alpha^{an}_4(r_1,\frac{r_2}{3}) \, \alpha^{an}_4(3r_2,r_3)\\
\leq \grandO{1} \Pro \Big[ \arm_4(r_1,r_2/3) \cap \widetilde{\gi}^{ext}_{\overline{\delta}}(r_2/3) \cap \dense_{\overline{\delta}}(r_2/3) \cap \qbc_{\overline{\delta}}(r_2/3) \cap \qbc^{ext}(r_2/3) \\
\cap \arm_4(3r_2,r_3) \cap \widetilde{\gi}^{int}_{\overline{\delta}}(3r_2) \cap \dense_{\overline{\delta}}(3r_2) \cap \qbc_{\overline{\delta}}(3r_2) \cap \qbc^{int}(3r_2) \Big] \, .
\end{multline*}
Now, we can condition on $\eta$ and on the interfaces and conclude (with arguments similar to the proof of Lemma~\ref{l.extension}) that the above is at most $\grandO{1} \alpha^{an}_4(r_1,r_3)$.
\end{enumerate}
\end{proof}

We have obtained the quasi-multiplicativity property for $r_1 \geq 16\overline{r}$: there exists a constant $C' \in [1,+\infty)$ such that, for every $16\overline{r} \leq r_1 \leq r_2 \leq r_3$,
\begin{equation}\label{e.quasibelowrbar}
\frac{1}{C'} \, \alpha^{an}_4(r_1,r_3) \leq \alpha^{an}_4(r_1,r_2) \, \alpha^{an}_4(r_2,r_3) \leq C' \, \alpha^{an}_4(r_1,r_3) \, .
\end{equation}
In order to obtain the full result, we need the following lemma:
\begin{lem}\label{l.extensionto1}
For every $\overline{\overline{r}}$ sufficiently large, there exists a constant $C_3=C_3(\overline{\overline{r}}) < +\infty$ such that, for every $r \in [1,\overline{\overline{r}}]$ and every $R \in [\overline{\overline{r}},+\infty)$, we have
\[
\alpha^{an}_4(\overline{\overline{r}},R) \leq C_3 \, \alpha^{an}_4(r,R) \, .
\]
\end{lem}

Before proving this lemma, let us explain why it enables us to conclude the proof of Proposition~\ref{p.quasieven}. Fix a quantity $\overline{\overline{r}} \geq 16 \overline{r}$ sufficiently large so that Lemma~\ref{l.extensionto1} holds. Let $r_1 \leq \overline{\overline{r}}$, $r_2 \in [r_1,+\infty)$ and $r_3 \in [r_2,+\infty)$. We distinguish between three cases:
\bi 
\item[1.] If $r_3 \leq \overline{\overline{r}}$ we are done thanks to~\eqref{e.poly}.
\item[2.] If $r_3 \geq \overline{\overline{r}} \geq r_2$ then we can use Lemma~\ref{l.extensionto1} to obtain that
\begin{eqnarray*}
\alpha^{an}_4(r_1,r_3) & \leq & \alpha^{an}_4(r_2,r_3)\\
& = & \frac{1}{\alpha^{an}_4(r_1,r_2)} \, \alpha^{an}_4(r_1,r_2) \, \alpha^{an}_4(r_2,r_3)\\
& \leq & \frac{1}{\alpha^{an}_4(1,\overline{\overline{r}})} \, \alpha^{an}_4(r_1,r_2) \, \alpha^{an}_4(r_2,r_3)\\
& \leq & \frac{1}{\alpha^{an}_4(1,\overline{\overline{r}})} \, \alpha^{an}_4(r_2,r_3)\\
& \leq & \frac{1}{\alpha^{an}_4(1,\overline{\overline{r}})} \, \alpha^{an}_4(\overline{\overline{r}},r_3)\\
& \leq & \frac{C_3}{\alpha^{an}_4(1,\overline{\overline{r}})} \, \alpha^{an}_4(r_1,r_3) \, .
\end{eqnarray*} 
The above implies the left-hand and right-hand sides of the quasi-multiplicativity property.
\item[3.] If $r_2 \geq \overline{\overline{r}}$, we can use Lemma~\ref{l.extensionto1} and~\eqref{e.quasibelowrbar} to obtain that
\begin{eqnarray*}
\alpha^{an}_4(r_1,r_3) & \leq & \alpha^{an}_4(\overline{\overline{r}},r_3)\\
& \leq & C' \, \alpha^{an}_4(\overline{\overline{r}},r_2) \, \alpha^{an}_4(r_2,r_3)\\
& \leq & C' \, C_3 \, \alpha^{an}_4(r_1,r_2) \, \alpha^{an}_4(r_2,r_3)\\
& \leq & C' \, C_3 \, \alpha^{an}_4(\overline{\overline{r}},r_2) \, \alpha^{an}_4(r_2,r_3)\\
& \leq & C' \, C_3 \, C' \, \alpha^{an}_4(\overline{\overline{r}},r_3)\\
& \leq & C' \, C_3 \, C' \, C_3 \, \alpha^{an}_4(r_1,r_3) \, ,
\end{eqnarray*}
and we are done.
\ei

\begin{proof}[Sketch of proof of Lemma~\ref{l.extensionto1}]
The proof is very similar to the proof of Lemma~\ref{l.piv_one_point}. Therefore, we only sketch it. Let $\dense^N(r)$ be the event defined in the proof of Lemma~\ref{l.piv_one_point}. Corollary~\ref{c.looksgood} and the inequality~\eqref{e.extension_better} imply that there exists an absolute constant $c \in (0,1)$ such that, for every $r$ sufficiently large, there exists $N=N(r)$ satisfying
 \[
\forall R \geq 4r, \, \Pro \left[ \widetilde{\arm}^{int}_4(r,R) \cap \dense^N(r) \right] \geq c \, \alpha_4^{an}(r,R) \, ,
\] 
where $\widetilde{\arm}^{int}_4(r ,R)$ is the event defined above~\eqref{e.extension_better}. Now, if we follow the proof of Lemma~\ref{l.piv_one_point}, we obtain that we can extend the four arms with probability larger than some constant that depends only on $r$ and $N$. More precisely, we obtain that there exists a constant $c'=c'(r,N)$ such that
\[
\alpha^{an}_4(R) \geq c' \, \Pro \left[ \widetilde{\arm}^{int}_4(r,R) \cap \dense^N(r) \right] \, .
\]
This ends the proof.
\end{proof}

\subsection{A consequence of the quasi-multiplicativity property}\label{ss.QM_consequences}

In this subsection, we prove Proposition~\ref{p.fandalpha} (where, instead of the events $\widehat{\arm}_j^{ext}(r,R)$ and $\widehat{\arm}_j^{int}(r,R)$ studied in Subsection~\ref{ss.quasi_even}, we consider the analogous event $\widehat{\arm}_j(r,R)$). We prove it only in the case $j$ even. See Subsection~\ref{ss.QM_odd} for the case $j$ odd.
\begin{proof}[Proof of Proposition~\ref{p.fandalpha} in the case $j$ even] We write the proof for $j=4$ since the proof for any $j \in \N^*$ even is essentially the same. Let $\dense(R) = \dense_{1/100}(A(R/2,2R))$. We have
\begin{eqnarray*}
\widehat{\arm}_4(r,R) & \subseteq & \widehat{\arm}_4^{int}(r,R/2) \cup \big( \widehat{\arm}_4(r,R) \setminus \dense(R) \big) \, .
\end{eqnarray*}
Hence,
\[
f_4(r,R) \leq f^{int}_4(r,R/2) + \Pro \left[ \neg \dense(R) \right] \, .
\]
By Corollary~\ref{c.looksgood}, if $r$ is sufficiently large, then $f^{int}_4(r,R/2) \asymp \alpha^{an}_4(r,R/2)$. Moreover, thanks to the quasi-multiplicativity property, $\alpha^{an}_4(r,R/2) \asymp \alpha^{an}_4(r,R)$. Furthermore, $\Pro \left[ \neg \dense(R) \right] \leq \grandO{1} \exp(-\Omega(1) R^2)$ while $\alpha^{an}_4(r,R)$ decays polynomially fast in $r/R \geq 1/R$. Hence, if $r$ sufficiently large ($r \geq r_0$, say) and if $R \in [r,+\infty)$, then,
\[
f_4(r,R) \leq \grandO{1} \alpha^{an}_4(r,R) \, .
\]
If $1 \leq r \leq r_0$, then we have
\[
f_4(r,R) \leq f_4(r_0,R) \leq \grandO{1} \alpha^{an}_4(r_0,R) \leq \grandO{1} \alpha_4^{an}(r,R) \, ,
\]
where the last inequality follows from the quasi-multiplicativity property. This ends the proof.
\end{proof}

\subsection{Arm events in the half-plane}\label{ss.half-plane}

In this subsection, we study $j$-arm events in the half-plane for any $j \in \N^*$.

\begin{rem}\label{r.ext_to_half}
In Subsection~\ref{ss.quasi_even}, we have restricted ourself to the case $j$ even since we wanted to deal with arms of alternating colors. In the case of the half-plane, whatever $j$ is odd or even, the arms are of alternating colors. As a result, if we follow the arguments of Subsection~\ref{ss.quasi_even}, we obtain the quasi-multiplicativity property for $j$-arm events in the half-plane for any $j \in \N^*$. We also obtain the analogues of Propositions~\ref{p.fandalpha} and~\ref{p.techniquegeneral}. (Of course, the proofs also work for arm events in a wedge, for instance in the quarter-plane.)
\end{rem}

We now use the quasi-multiplicativity property to compute the exponents of the $2$ and $3$-arm events in the half-plane.

\begin{proof}[Proof of Items~i) and~ii) of Proposition~\ref{p.universal}] (We follow~\cite{werner2007lectures}, first exercise sheet.) First, note that thanks to the quasi-multiplicativity property, it is sufficient to prove the result for $r=1$ and for $R \geq 1$ sufficiently large. We define the two following events (where $\half$ is the upper half-plane):
\bi 
\item[1.] For every $j \in \Z$, let $I_j = [j,j+1] \times \lbrace 0 \rbrace$ and write $F^{2,+}_j(R)$ for the event that there exist $y \in I_j$ and $\gamma_1, \, \gamma_2$ two paths such that: (a) $\gamma_1$ and $\gamma_2$ are included in $B_R(y) \cap \half$, (b) $\gamma_1$ and $\gamma_2$ join $y$ to $\partial B_R(y)$, (c) $\gamma_1$ is black and $\gamma_2$ is white and (d) $\gamma_1$ is on the right-hand-side of $\gamma_2$. (Note that this implies in particular that $y$ belongs to the intersection of two Voronoi cells.)
\item[2.] Let $S$ be a $1 \times 1$ square of the grid $\Z^2$ and write $F^{3,+}_S(R)$ for the event that there exists $y \in S$ that is the lowest point in $B_R(y)$ of a black component that intersects $\partial B_R(y)$.
\ei
Let $\eta \in \dense_{1/100} \left( B_{2R} \right) \cap \qbc_{1/100}^3 \left( B_{2R} \right)$ (see Subsection~\ref{ss.formal_events} for the definition of these events). If we follow the first exercise sheet of~\cite{werner2007lectures} (in this exercise sheet, one has to use the BK inequality; this is not a problem since we work at the quenched level), we obtain that there exists a constant $C \in [1,+\infty)$ such that\footnote{Actually, for the $3$-arm event the proof in the case of Voronoi percolation is easier than in the case of Bernoulli percolation since, for Voronoi percolation, a.s. a cluster cannot have two lowest points.}
\[
\frac{1}{C} \leq \sum_{j \, : \, I_j \cap B_{R/2} \neq \emptyset} \Prob^\eta \left[ F^{2,+}_j(R) \right] \leq C
\]
and
\begin{equation}\label{e.F3}
\frac{1}{C} \leq \sum_{S \, : \, S \cap B_{R/2} \neq \emptyset} \Prob^\eta \left[ F^{3,+}_S(R) \right] \leq C \, .
\end{equation}
Lemma~\ref{l.dense} and Proposition~\ref{p.a_lot_of_quads} imply that
\[
\Pro \left[ \dense_{1/100} \left( B_{2R} \right) \cap \qbc_{1/100}^3 \left( B_{2R} \right) \right] \geq 1- \left( \grandO{1}e^{-\Omega(1) \, R^2} + \grandO{1}R^{-3} \right) \geq 1 - \grandO{1} R^{-3} \, .
\]
Let us conclude the proof in the case of the $3$-arm event (the case of the $2$-arm event is treated similarly). If we combine the above estimate with~\eqref{e.F3}, we obtain that
\begin{multline*}
\frac{1}{C}(1-\grandO{1} R^{-3}) \leq \sum_{S \, : \, S \cap B_{R/2} \neq \emptyset} \Pro \left[ F^{3,+}_S(R) \right]\\
\leq C + \grandO{1} R^{-3} \text{Card} \{ S \, : \, S \cap B_{R/2} \neq \emptyset \} \leq C + \grandO{1} R^{-1} \, .
\end{multline*}
Since the annealed model is translation invariant, $\Pro \left[ F_S^{3,+}(R) \right]$ does not depend on $S$, and if $R$ is sufficiently large, we have
\[
\Pro \left[ F_S^{3,+}(R) \right] \asymp R^{-2} \, .
\]
Therefore, it is sufficient to prove that, for every $R$ sufficiently large,
\[
\alpha^{an,+}_3(R) \asymp \Pro \left[ F_S^{3,+}(R) \right] \, .
\]
\bi 
\item[i)] The proof that $\Pro \left[ F^{3,+}_S(R) \right] \leq \grandO{1} \alpha^{an,+}_3(R)$ is essentially the same as the one of the inequality $\Pro \left[ \arm_4^\square(S,R) \right] \leq \grandO{1} \, \alpha^{an}_4(\rho,R)$ of Proposition~\ref{l.warm-up_piv}. Hence, we leave it to the reader.
\item[ii)] We also leave the proof that $\Pro \left[ F^{3,+}_S(R) \right] \geq \Omega(1) \, \alpha^{an,+}_3(R)$ to the reader since one can show this by extending the arms ``by hands'' exactly as in the proof of Lemma~\ref{l.extensionto1}.
\ei
This ends the proof.
\end{proof}

\subsection{The case $j$ odd}\label{ss.QM_odd}

Let us prove the quasi-multiplicativity property in the case $j$ odd.
\begin{proof}[Proof of Proposition~\ref{p.quasi} in the case $j$ odd] To deal with an odd number of arms, it is not sufficient to work with the events $\widetilde{\gi}^{ext}_\delta(R)$ and $\widetilde{\gi}^{int}_\delta(r)$ that we have studied in Subsection~\ref{ss.quasi_even}. More precisely, in order to extend two consecutive arms of the same color, we need to work with a configuration of interfaces that satisfy the following condition: ``the endpoints are far away from the other interfaces'' (and not only ``the endpoints are far away from each other''). In other words, we need to work with a configuration of interfaces that satisfy the event $\gi^{ext}_\delta(R)$ (or $\gi^{int}_\delta(r)$) defined above Lemma~\ref{l.interfaces}. Since the proof of Lemma~\ref{l.interfaces} (that gives estimates on the quantities $\Pro \left[ \gi^{ext}_\delta(R) \right]$ and $\Pro \left[ \gi^{int}_\delta(r) \right]$) only relies on Subsections~\ref{ss.quasi_even},~\ref{ss.QM_consequences} and~\ref{ss.half-plane}, we can now use this result.

If we modify the definition of $G^{ext}_\delta(R)$ and let
\[
G^{ext}_\delta(R)=\gp^{ext}_\delta(R) \cap \gi_\delta^{ext}(R)
\]
instead of
\[
G^{ext}_\delta(R)=\gp^{ext}_\delta(R) \cap \widetilde{\gi}_\delta^{ext}(R) \, ,
\]
and if we definie similarly $G^{int}_\delta(R)=\gp^{int}_\delta(R) \cap \gi_\delta^{int}(R)$, then the proof of the quasi-multiplicativity property in the case $j$ odd is the same as in the case $j$ even (except that we now need Lemma~\ref{l.interfaces} to prove the analogue of Lemma~\ref{l.good}).
\end{proof}

We end this subsection by noting that: (a) now, the proof of Proposition~\ref{p.fandalpha} in the case $j$ odd is the same as in the case $j$ even and (b) we can compute the universal arm-exponent for the $5$-arm event. Let us be a little more precise about the computation of this exponent:

\begin{proof}[Proof of Item~iii) of Proposition~\ref{p.universal}] We work with the following event:\\
Let $S$ be a $1 \times 1$ square of the grid $\Z^2$ and write $F^5_S(R)$ for the event that there exists a point $x \in \eta \cap S$ such that: (a) the cell of $x$ is white, (b) there exist five paths $\gamma_1, \cdots, \gamma_5$ (in counter-clockwise order, say) that join the cell of $x$ to $\partial B_R(x)$, (c) $\gamma_i$ is white (respectively black) if $i$ is odd (respectively even) and (d) if $i \neq j$ then there is no Voronoi cell that is intersected by both $\gamma_i$ and $\gamma_j$.

If we follow the proof of Items~i) and~ii) of Proposition~\ref{p.universal} (i.e. if we follow both the first exercise sheet of \cite{werner2007lectures} and Subsection~\ref{ss.half-plane} of the present paper), we obtain that there exists a constant $C \in [1,+\infty)$ such that
\[
\forall \eta \in \dense_{1/100}(B_{2R}) \cap \qbc^3_{1/100}(B_{2R}), \, \frac{1}{C} \leq \sum_{S \, : \, S \cap B_{R/2} \neq \emptyset} \Prob^\eta \left[ F_S(R) \right] \leq C \, ,
\]
and thus that it is sufficient to prove that, for every $R \geq 1$ sufficiently large, we have
\[
\alpha^{an}_5(R) \asymp \Pro \left[ F_S^5(R) \right] \, .
\]
Since we now have the quasi-multiplicativity property in the case $j=5$, the proof of this last estimate is the same as the analogous estimates discussed in Subsection~\ref{ss.half-plane}.
\end{proof}

\appendix

\section{An extension of Schramm and Steif's algorithm theorem}\label{a.Schramm-Steif}

In this appendix, we state an extension of Schramm and Steif's algorithm theorem that has been proved by Roberts and Sengul in~\cite{roberts2016exceptional}. We first need the following definition: Let $n \in \N$ and let $f \, : \, \lbrace -1,1 \rbrace^n \rightarrow \R$. An \textbf{algorithm} that determines $f$ is a procedure that asks the values of the bits step by step where at each step the algorithm can ask for the value of one or several bits and the choice of the new bit(s) to ask is based on the values of the bits previously queried. We also ask that the algorithm stops once $f$ is determined. We denote by $\Prob^n_p$ the probability measure on $\Omega^n := \lbrace -1,1 \rbrace^n$ defined by
\[
\Prob^n_p = \left( p\delta_1+(1-p)\delta_{-1}\right)^{\otimes n} \, .
\]
A crucial quantity is the revealment of an algorithm $\mathcal{A}$. This is defined as follows:
\[
\delta_\mathcal{A}^p = \underset{i \in \lbrace 1, \cdots, n \rbrace}{\max} \Prob^n_p \left[ i \text{ is queried by } \mathcal{A} \right] \, .
\]
To state Schramm and Steif's result, we also need to introduce the notion of discrete Fourier decomposition: Let $S \subseteq \lbrace 1, \cdots, n \rbrace$ and let $\omega \in \Omega^n$. We write
\[
\chi_S^p(\omega) = \prod_{i \in S} \left( \sqrt{\frac{1-p}{p}} \un_{\omega_i = 1} - \sqrt{\frac{p}{1-p}} \un_{\omega_i = -1}  \right) \, .
\]
Note that $\left( \chi_S^p \right)_{S \subseteq \lbrace 1, \cdots, n \rbrace}$ is an orthonormal family of $L^2 \left( \Omega^n , \Prob^n_p \right)$, thus we can define $\left( \widehat{f}^p_S \right)_S$, the Fourier coefficients of $f \, : \, \Omega^n \rightarrow \R$ at level $p$, as follows:
\[
f = \sum_{S \subseteq \lbrace 1, \cdots, n \rbrace} \widehat{f}^p(S) \chi_S^p \, .
\]
The result by Schramm and Steif is the following (they proved it for $p=1/2$ but the proof for any $p$ is the same):
\begin{thm}[Theorem~$1.8$ of~\cite{schramm2010quantitative}]
For every $f \, : \, \Omega^n \rightarrow \R$, every algorithm $\mathcal{A}$ that determines $f$ and every $k \in \N^*$; we have
\[
\sum_{S \subseteq \lbrace 1, \cdots, n \rbrace \, : \, |S| = k} \widehat{f}^p(S)^2 \leq \delta^p_\mathcal{A} \, k \, \Ex_p^n \left[ f^2 \right] \, .
\]
\end{thm}
We need the following extension of this theorem: Let $I \subseteq \lbrace 1, \cdots, n \rbrace$ and, if $\mathcal{A}$ is some algorithm, let us write
\[
\delta_\mathcal{A}^p(I) = \underset{i \in I}{\max} \, \Prob^n_p \left[ i \text{ is queried by } \mathcal{A} \right] \, .
\]

\begin{prop}[Theorem~2.3 of~\cite{roberts2016exceptional}]\label{p.gen_Schramm-Steif_1}
For every $f \, : \, \Omega^n \rightarrow \R$, every algorithm $\mathcal{A}$ that determines $f$, every $I \subseteq \lbrace 1, \cdots, n \rbrace$ and every $k \in \N^*$, we have
\[
\sum_{S \subseteq I \, : \, |S| = k} \widehat{f}^p(S)^2 \leq \delta^p_\mathcal{A}(I) \, k \, \Ex^n_p \left[ f^2 \right] \, .
\]
\end{prop}
\begin{proof}
The proof is very close to the proof of Theorem~$1.8$ of~\cite{schramm2010quantitative}, except that we need to (slightly) change the definition of the function $g$ therein. More precisely, we need to work with
\[
g \, : \,  \omega \mapsto \sum_{S \subseteq I \, : \, |S| = k} \widehat{f}^p(S) \, \chi_S^p(\omega) \, .
\]
See~\cite{roberts2016exceptional} for more details.
\end{proof}

The reason why we are interested in the above theorem is the following property (see Subsection~\ref{ss.intro_piv} for the definition of the pivotal event $\Piv^n_i(A)$):
\begin{prop}\label{p.piv_Fourier}
Let $A \subseteq \Omega^n$ be an increasing event. Also, let $f$ be the $\pm 1$-indicator function of $A$ (i.e. $f = 2\un_A - 1$). Then, for every $i \in \lbrace 1, \cdots, n \rbrace$, we have
\[
\widehat{f}^p(\lbrace i \rbrace) = 2\sqrt{p \, (1-p)} \; \Prob^n_p \left[ \Piv^n_i(A) \right] \, .
\]
\end{prop}
\begin{proof}
The proof is exactly the same as in the case $p=1/2$, which can be found for instance in~\cite{garban2014noise}, Proposition~$4.5$.
\end{proof}

The two above propositions imply the following corollary, which is the result that we will need:
\begin{cor}\label{c.Schramm_Steif}
Let $A \subseteq \Omega^n$ be an increasing event. Then, for every algorithm $\mathcal{A}$ that determines $\un_A$ and every $I \subseteq \lbrace 1, \cdots, n \rbrace$, we have
\[
\sum_{i \in I} \Prob^n_p \left[ \Piv^n_i(A) \right]^2 \leq \frac{1}{4 p \, (1-p)} \, \delta^p_\mathcal{A}(I) \, .
\]
\end{cor}

\section{The proof of the quenched box-crossing property in \cite{ahlberg2015quenched}}\label{a.tAGMT}

In this section, we only work at $p=1/2$, hence we forget the subscript $p$ in the notations. We recall the main steps of the proof of Theorem~\ref{t.AGMT} by Ahlberg, Griffiths, Morris and Tassion (which is Theorem~1.4 in~\cite{ahlberg2015quenched}). There are two reasons why we need to recall this proof: i) In~\cite{ahlberg2015quenched}, the theorem is proved for the analogous model in which $\eta$ is a family of $n$ points sampled uniformly and independently in some fixed rectangle. As pointed out in~\cite{ahlberg2015quenched} (below the statement of their Theorem~$1.4$) the proof in the case we are interested in (i.e. in which $\eta$ is a Poisson process in the whole plane) is essentially the same. We explain briefly why. ii) In order to extend this result to $p>1/2$ (in Subsection~\ref{ss.pneq1/2_1}) we have to modify a little the end of the proof.\\

Let us first note that (as it explained at the end of the paper~\cite{ahlberg2015quenched}) Item~ii) of Theorem~\ref{t.AGMT} is an easy consequence of Item~i) of this theorem. As a result, we only explain the strategy in order to obtain Item~$i)$.

\paragraph{A. A martingale estimate.} First, the authors of~\cite{ahlberg2015quenched} prove a martingale estimate. Let $\rho, \, R > 0$. Also, let $N \in \N^*$ and consider $\eta_N$ a configuration of $4e^{2N}$ points sampled uniformly in $[-e^N,e^N]^2$, independently of each other. Remember the definition of pivotal events from Subsection~\ref{ss.intro_piv}. The following is not exactly Theorem~$2.1$ of~\cite{ahlberg2015quenched} but the proof is the same:
\begin{equation}\label{e.martingale_AGMT}
\Var \left( \Prob^{\eta_N} \left[ \, \cross(\rho R,R) \, \right] \right) \leq \E \left[ \sum_{x \in \eta_N} \Prob^{\eta_N} \left[ \Piv^{\eta_N}_x( \cross(\rho R,R) ) \right]^2 \right] \, ,
\end{equation}
where $\Prob^{\eta_N}:=\left(\frac{\delta_1}{2}+\frac{\delta_{-1}}{2} \right)^{\eta_N}$. (Note that the point process $\eta_N$ is a.s. finite, hence we have only finitely many Voronoi cells.) Now, note that we can couple $\eta_N$ with a Poisson process of intensity $1$ in the plane (denoted by $\eta$) so that, with probability going to $1$ as $N$ goes to $+\infty$ superpolynomially fast in $N$, we have\footnote{This is for instance a consequence of Le Cam's identity which implies that:
\[
\sum_{k=0}^{+\infty} \left| \Pro \left[ | \eta_N \cap [-N,N]^2 | = k \right] - \Pro \left[ | \eta \cap [-N,N]^2 | = k \right] \right| \leq 2 \times 4N^2 \frac{4N^2}{4e^{2N}} \leq e^{-\Omega(1)N} \, .
\]}
\[
\eta \cap [-N,N]^2 = \eta_N \cap [-N,N]^2 \, .
\]
Since the event $\cross(\rho R,R)$ depends only on the points of $\eta \cap [-N,N]^2$ with probability that goes to $1$ as $N$ goes to $+\infty$ superpolynomially fast in $N$, the above together with~\eqref{e.martingale_AGMT} implies that
\begin{equation}\label{e.martingale_AGMT_bis}
\Var \left( \Prob^\eta \left[ \, \cross(\rho R,R) \, \right] \right) \leq \E \left[ \sum_{x \in \eta} \Prob^{\eta} \left[ \Piv^q_x( \cross(\rho R,R) ) \right]^2 \right] \, .
\end{equation}
(See Subsection~\ref{ss.intro_piv} for the definition of $\Piv^q_x(\cdot)$.)

\paragraph{B. An estimate on the $1$-arm event.} The next result we need is an analogue of Proposition~$3.11$ of~\cite{ahlberg2015quenched}. This proposition is proved in the case where $\eta$ is a set of $n$ independent points sampled uniformly in a rectangle but with exactly the same proof we obtain the following result:

Let $S$ be the $1 \times 1$ square centered at some point $y$ and let $\arm_1^{*,\text{cell}}(S,r)$ be the event that there exists a point $x \in \eta \cap S$ such that there is a white path from the cell of $x$ to a cell that intersects $\partial B_r(y)$ (note that the cell of $x$ is not necessarily white). For every $\gamma > 0$, there exists $\epsilon=\epsilon(\gamma) > 0$ such that the following holds:
\begin{equation}\label{e.prop_3.11_AGMT}
\Pro \left[ \Prob^\eta \left[ \arm_1^{*,\text{cell}}(S,r) \right] \geq r^{-\epsilon} \right] \leq \frac{1}{\epsilon} r^{-\gamma} \, .
\end{equation}
(Note that we have decided to study white arms instead of black arms. It will make sense in Subsection~\ref{ss.pneq1/2_1}.)

\paragraph{C. A Schramm and Steif's algorithm method.} The last step of the proof relies on a theorem from~\cite{schramm2010quantitative}. In our case, we are going to use the extension of Schramm-Steif result discussed in Appendix~\ref{a.Schramm-Steif}. Let $\setS_1$ (respectively $\setS_2$) be the subset of all the $1 \times 1$ squares of the grid $\Z^2$ that are below (respectively above) the line $\R \times \lbrace 0 \rbrace$ and that are at distance at most $(\rho + 1)R$ from the rectangle $[-\rho R,\rho R] \times [-R,R]$ (we include the squares that intersect $\R \times \lbrace 0 \rbrace$ in both $\setS_1$ and $\setS_2$). Also, let $\setS_3$ be all the remaining $1 \times 1$ squares of the grid $\Z^2$. The following is a direct consequence of~\eqref{e.martingale_AGMT_bis} (and of the symmetries of the model):
\begin{align*}
& \Var \left( \Prob^\eta \left[ \, \cross(\rho R,R) \, \right] \right)\\
& \leq \sum_{k=1}^{3} \E \left[ \sum_{S \in \setS_k} \sum_{x \in \eta \cap S} \Prob^{\eta} \left[ \Piv^q_x( \cross(\rho R,R) ) \right]^2 \right]\\
& = 2 \; \E \left[ \sum_{S \in \setS_1} \sum_{x \in \eta \cap S} \Prob^{\eta} \left[ \Piv^q_x( \cross(\rho R,R) ) \right]^2 \right] + \E \left[ \sum_{S \in \setS_3} \sum_{x \in \eta \cap S} \Prob^{\eta} \left[ \Piv^q_x( \cross(\rho R,R) ) \right]^2 \right] \, .\\
\end{align*}

Let us first deal with the sum over $\setS_3$. This sum is less than or equal to the expectation of the number of points which are at distance at least $(\rho + 1)R$ from the rectangle $[-\rho R,\rho R] \times [-R,R]$ but whose cell intersects this rectangle. It is not difficult to see that this quantity is less than $\grandO{1} e^{-\Omega(1)R^2}$ (where the constants in $\grandO{1}$ and $\Omega(1)$ may depend on $\rho$).\\

Now, let us bound the sum over $\setS_1$. Here, we follow the ideas of~\cite{ahlberg2015quenched} but we use a slightly different algorithm, which can be defined as follows: (we use the same notations as~\cite{ahlberg2014noise} where the authors use this kind of algorithm to study the Boolean model): Let $Q_0$ denote the set of all $x \in \eta$ whose cell intersects the set $\left( \R \times \lbrace R \rbrace \right) \cap \left( [-\rho R,\rho R] \times [-R,R] \right)$. Also, let $A_0$ be the set of all $x \in Q_0$ which are white. For each $k \in \N^*$, we define $A_k$ and $Q_k$ (for ``active'' and ``queried'' sets) inductively as follows:
\bi 
\item[i)] Let $Q_k$ be the set of all points $x \in \eta$ such that: (a) the cell of $x$ is adjacent to the cell of some $y \in A_{k-1}$  and (b) the cell of $x$ intersects the rectangle $[-\rho R,\rho R] \times [-R,R]$. Reveal the color of each point of $Q_k$.
\item[ii)] Let $A_k$ be the set of all $x \in Q_k$ which are white.
\item[iii)] Stop if $A_k = A_{k-1}$.
\ei

Note that this algorithm (that we denote by $\mathcal{A}_R$) determines the event that there is a white top-bottom crossing of $[-\rho R,\rho R] \times [-R,R]$ which is the complement of the event $\cross(\rho R,R)$. As a result, this algorithm determines $\cross(\rho R,R)$. Corollary~\ref{c.Schramm_Steif} implies that
\[
\E \left[ \sum_{S \in \setS_1} \sum_{x \in \eta \cap S} \Prob^{\eta} \left[ \Piv^q_x( \cross(\rho R,R) ) \right]^2 \right] \leq \delta_{\mathcal{A}_R}(\setS_1) \, ,
\]
where
\[
\delta_{\mathcal{A}_R}(\setS_1) = \underset{S \in \setS_1}{\max} \, \underset{x \in \eta}{\max} \, \Prob^\eta \left[ x \text{ is queried by } \mathcal{A}_R \right] \, .
\]
It remains to show that this last quantity is at least polynomially small in $R$. This can be done by noting that
\[
\underset{S \in \setS_1}{\max} \, \underset{x \in \eta}{\max} \, \Prob^\eta \left[ x \text{ is queried by } \mathcal{A}_R \right] \leq \underset{S \in \setS_1}{\max} \,  \Prob^\eta \left[ \arm_1^{*,\text{cell}}(S,R-1) \right] \, .
\]
The fact that the above quantity is at least polynomially small in $R$ is an easy consequence of~\eqref{e.prop_3.11_AGMT} (for instance with $\gamma = 3$). We refer to~\cite{ahlberg2015quenched} for more details.

\section{Pivotal events for $\arm_j(1,R)$ when $j$ is odd}\label{a.joddpivarm}

In this appendix, we prove Lemmas~\ref{l.piv_arm_1}, \ref{l.piv_arm_2}, \ref{l.piv_arm_3} and~\ref{l.piv_arm_4} in the case $j$ odd. We do not need them in order to prove our main result Theorem~\ref{t.scaling}. However, we need them in order to prove that Propositions~\ref{p.arm_event_asymp_R2alpha4} and~\ref{p.alpha_jandalpha_1/2} also hold when $j$ is odd. Let $S \subseteq A(R/4,R/2)$ be a $2\rho \times 2\rho$ square centered at some point $y$, let $\dense(y;\rho') = \dense_{1/100}(A(y;\rho',2\rho'))$, and assume that $\Piv_S(\arm_j(1,R)) \cap \dense(y;2^k \rho)$ holds for some $k \in \lbrace 0, \cdots, \lfloor \log_2 \left( \frac{R}{16\rho} \right) \rfloor =: k_0 \rbrace$. This implies that $\arm_j(1,R/8)$ holds. In the case $j$ even, this also implies that the $4$-arm event $\arm_4(y;2^{k+1}\rho,2^{k_0}\rho)$ holds, where $\arm_4(y;2^{k+1}\rho,2^{k_0}\rho)$ is $\arm_4(2^{k+1}\rho,2^{k_0}\rho)$ translated by $y$. If $j$ is odd, this  rather implies that the following more complicated event holds:
\[
\bigcup_{l=k}^{k_0-1} \left( \widetilde{\arm}_4(y;2^{k+1}\rho,2^{l+1}\rho) \cap \widetilde{\arm}_5(y;2^{l+2}\rho,2^{k_0}\rho) \right) \, ,
\]
where: i) $\widetilde{\arm}_4(y;\rho',\rho'')$ is the event that there is a point $x \in \eta$ such that: (a) $C(x)$ (the Voronoi cell of $x$) intersects $A(y;\rho''/2,2\rho'')$ and (b) there are four arms of alternating colors in $A(y;\rho',2\rho'')$ from $\partial B_{\rho'}(y)$ to $\partial B_{\rho''}(y) \cup \partial C(x)$ and ii) $\widetilde{\arm}_5(y;\rho',\rho'')$ is the event that there is a point $x \in \eta$ such that: (a) $C(x)$ intersects $A(y;\rho'/2,2\rho')$ and (b) there are five 	arms of alternating colors in $A(y;\rho'/2,\rho'')$ from $\partial B_{\rho'}(y) \cup \partial C(x)$ to $\partial B_{\rho''}(y)$. (Actually, instead of the $5$-arm event, we could have asked that a $6$-arm event with colors following the order $(B,B,W,B,B,W)$ holds, where $B=$black and $W=$white.) See~\cite{nolin2008near} (for instance Figure~12 therein) for a similar observation in the case of Bernoulli percolation on the triangular lattice. Now, write
\[
\widehat{\widetilde{\arm}}_4(y;\rho',\rho'') = \left\lbrace  \Pro \left[ \widetilde{\arm}_4(y;\rho',\rho'') \cond \omega \cap A(y;\rho',\rho'') \right] > 0 \right\rbrace  \, , 
\]
and
\[
\widehat{\widetilde{\arm}}_5(y;\rho',\rho'') = \left\lbrace  \Pro \left[ \widetilde{\arm}_5(y;\rho',\rho'') \cond \omega \cap A(y;\rho',\rho'') \right] > 0 \right\rbrace \, .
\]
By spatial independence and by a union-bound, we have
\begin{multline*}
\Pro \left[ \bigcup_{l=k}^{k_0-1} \left( \widehat{\widetilde{\arm}}_4(y;2^{k+1}\rho,2^{l+1}\rho) \cap \widehat{\widetilde{\arm}}_5(y;2^{l+2}\rho,2^{k_0}\rho) \right) \, \right]\\
\leq \sum_{l=k}^{k_0-1} \Pro \left[ \widehat{\widetilde{\arm}}_4(y;2^{k+1}\rho,2^{l+1}\rho) \right] \cdot \Pro \left[ \widehat{\widetilde{\arm}}_5(y;2^{l+2}\rho,2^{k_0}\rho) \right] \, .
\end{multline*}
By using arguments very similar to those of the proof of $\Pro \left[ \Piv_S (\cross(2R,R) \right] \leq \grandO{1} \, \alpha^{an}_4(\rho,R)$ in Lemma~\ref{l.piv_bulk} and of the proof of Lemma~\ref{l.pivhat}, we obtain that
\[
\Pro \left[ \widehat{\widetilde{\arm}}_4(y;\rho',\rho'') \right] \leq \grandO{1} \alpha^{an}_4(\rho',\rho'') \, ,
\]
and
\[
\Pro \left[ \widehat{\widetilde{\arm}}_5(y;\rho',\rho'') \right] \leq \grandO{1} \, \alpha^{an}_5(\rho',\rho'') \, .
\]
Proposition~\ref{p.alpha4} and Item~iii) of Proposition~\ref{p.universal} then imply that
\begin{equation}\label{e.alpha4and5}
\alpha^{an}_5(\rho_1,\rho_2) \leq \grandO{1} \left( \frac{\rho_1}{\rho_2} \right)^\epsilon \, \alpha^{an}_4(\rho_1,\rho_2) \, .
\end{equation}
Together with the above results and the quasi-multiplicativity property, this implies that
\[
\sum_{l=k}^{k_0-1} \Pro \left[ \widehat{\widetilde{\arm}}_4(y;2^{k+1}\rho,2^{l+1}\rho) \right] \cdot \Pro \left[ \widehat{\widetilde{\arm}}_5(y;2^{l+2}\rho,2^{k_0}\rho) \right] \leq \grandO{1} \, \alpha^{an}_4(\rho,R) \, .
\]
Finally, if we had said that the event $\Piv_S(\arm_j(1,R)) \cap \dense(y;2^k \rho)$ implied that the $4$-arm event $\arm_4(y;2^{k+1}\rho,R/8)$ held, then it would have given a true estimate. Now, the proof of Lemma~\ref{l.piv_arm_1} is very similar to the proof of the inequality $\Pro \left[ \Piv_S (\cross(2R,R)) \right] \leq \grandO{1} \, \alpha^{an}_4(\rho,R)$ of Lemma~\ref{l.piv_bulk}. To obtain the other lemmas, we need to make similar observations for arm events near $\partial B_R$. The only difference is that, instead of using the estimate~\eqref{e.alpha4and5}, we need to use the following similar results (whose proofs are exactly the same as the proof of the second part of Proposition~\ref{p.alpha4} written in Subsection~\ref{ss.alpha4}):
\[
\alpha^{an,+}_4(\rho',\rho'') \leq \grandO{1} \, \left( \frac{\rho'}{\rho''} \right)^\epsilon \, \alpha^{an,+}_3(\rho',\rho'') \, ,
\]
and
\[
\alpha^{an,++}_4(\rho',\rho'') \leq \grandO{1} \, \left( \frac{\rho'}{\rho''} \right)^\epsilon \, \alpha^{an,++}_3(\rho',\rho'') \, .
\]
We leave the details to the reader.

\section{The quantities $\E \left[ \Prob^\eta \left[ \arm_j(r,R) \right]^2 \right]$}\label{a.quant}

In this appendix, we only work at $p=1/2$, hence we forget the subscript $p$ in the notations. We study the following quantities:
\[
\widetilde{\alpha}_j(r,R) := \sqrt{ \E \left[ \Prob^\eta \left[ \arm_j(r,R) \right]^2 \right]} \, .
\]
More precisely, we prove that some of the results that we have proved for the quantities $\alpha^{an}_j(r,R)$ are also true for the $\widetilde{\alpha}_j(r,R)$'s. We actually do not need the results of this appendix in the present paper but we include them here since there will be crucial in the paper~\cite{quant_voro} where we prove in particular that
\begin{equation}\label{e.key_of_quant}
\widetilde{\alpha}_j(r,R) \asymp \alpha^{an}_j(r,R) \, .
\end{equation}
We refer to~\cite{quant_voro} for the motivations behind~\eqref{e.key_of_quant}.\\

Let us start the study of the quantities $\widetilde{\alpha}_j(r,R)$. First note that, by Jensen's inequality, we have
\begin{equation}\label{e.1st_quant}
\widetilde{\alpha}_j(r,R)^2 \leq \alpha^{an}_j(r,R) \leq \widetilde{\alpha}_j(r,R) \, .
\end{equation}
As a result, the following polynomial decay property is a direct consequence of~\eqref{e.poly}:
\begin{equation}\label{e.poly_bis}
\frac{1}{C} \, \left( \frac{r}{R} \right)^{C} \leq \widetilde{\alpha}_{j}(r,R) \leq C \, \left( \frac{r}{R} \right)^{1/C} \, .
\end{equation}

\subsection{The quasi-multiplicativity property}
In this subsection, we explain how the proof of the quasi-multiplicativity property written in Section~\ref{s.quasi} can be adapted in order to prove the following:
\begin{prop}\label{p.quasi_tilde}
The quasi-multiplicativity property also holds for the quantities $\widetilde{\alpha}_j(r,R)$ i.e. there exists $C=C(j) \in [1,+\infty)$ such that, for every $1 \leq r_1 \leq r_2 \leq r_3$,
\[
\frac{1}{C} \, \widetilde{\alpha}_j(r_1,r_3) \leq \widetilde{\alpha}_j(r_1,r_2) \, \widetilde{\alpha}_j(r_2,r_3) \leq C \, \widetilde{\alpha}_j(r_1,r_3) \, .
\]
\end{prop}

\begin{proof} The proof is very close to the proof of Proposition~\ref{p.quasi}. To simplify the notations, we write the proof in the case $j=4$. The proof for any other even integer is the same and the proof for any odd integer requires the same modifications as in Subsection~\ref{ss.QM_odd}. We use the same notations as in Subsection~\ref{ss.quasi_even} (remember in particular the definition of the events $G^{ext}_\delta(R)$ and $G^{int}_\delta(r)$ in the beginning of this section).

Let us first state and prove an analogue of Lemma~\ref{l.extension}. We need the following notation: If $\delta \in (0,1/1000)$, $R \in [\delta^{-2},+\infty)$ and $r \in [1,R]$, we write
\[
\widetilde{g}_{4,\delta}^{ext}(r,R) = \sqrt{\E \left[ \Prob^\eta \left[ \text{\textbf{A}}_4(r,R) \cap G_\delta^{ext}(R) \right]^2 \right]} \, .
\]
Similarly, if $\delta \in (0,1/1000)$, $r \in [\delta^{-2},+\infty)$ and $R \in [r,+\infty)$, we write
\[
\widetilde{g}_{4,\delta}^{int}(r,R) = \sqrt{ \E \left[ \Prob^\eta \left[ \text{\textbf{A}}_4(r,R) \cap G_{\delta}^{int}(r) \right]^2 \right]} \, .
\]
\begin{lem}\label{l.extension_bis}
There exists $\overline{\delta} \in (0,1/1000)$ such that, for any $\delta \in (0,1/1000)$, there is some constant $a=a(\delta) \in (0,1)$ satisfying the following:
\begin{enumerate}
\item For every $R \in [\overline{\delta}^{-2} \vee \delta^{-2},+\infty)$ and every $r \in [1,R/4]$, we have
\begin{equation}\label{e.gplus_bis}
\widetilde{g}_{4,\overline{\delta}}^{ext}(r,4R) \geq a \, \widetilde{g}^{ext}_{4,\delta}(r,R) \, .
\end{equation}
\item For every $r \in [4 ( \overline{\delta}^{-2} \vee \delta^{-2} ),+\infty)$ and every $R \in [4r,+\infty)$, we have
\begin{equation}\label{e.gmoins_bis}
\widetilde{g}_{4,\overline{\delta}}^{int}(r/4,R) \geq a \, \widetilde{g}^{int}_{4,\delta}(r,R) \, .
\end{equation}
\end{enumerate}
\end{lem}
\begin{proof}
We write only the proof of~\eqref{e.gplus_bis} since the proof of~\eqref{e.gmoins_bis} is the same. We use exactly the same notations as in the proof of Lemma~\ref{l.extension}. By~\eqref{e.key_lemma_ext}, if $\overline{\delta}$ is sufficiently small, then,
\[
\nu^\eta_{r,R,(\beta_j)_j} \left[ \text{\textbf{A}}_4(r,4R) \cap G^{ext}_{\overline{\delta}}(4R) \right] \geq c \, ,
\]
for some constant $c=c(\delta)>0$. If we take the expectation under $\Prob^\eta_{B_{2R}}$, we obtain that
\begin{multline*}
\Prob^\eta_{B_{2R}} \left[ \text{\textbf{A}}_4(r,4R) \cap G^{ext}_{\overline{\delta}}(4R) \right]\\
\geq c \, \Prob^\eta_{B_{2R}} \left[ \arm_4(r,R) \cap \gp^{ext}_\delta(R) \cap \lbrace \widetilde{s}^{ext}(r,R) \geq 10 \delta R \rbrace \right]\\
\geq c \, \Prob^\eta_{B_{2R}} \left[ \text{\textbf{A}}_4(r,R) \cap G^{ext}_\delta(R) \right] \, .
\end{multline*}
We then conclude by using both the following martingale inequality:
\begin{multline*}
\widetilde{g}_{4,\overline{\delta}}(r,4R)^2 = \E \left[ \Prob^\eta \left[ \text{\textbf{A}}_4(r,4R) \cap G^{ext}_{\overline{\delta}}(4R) \right]^2 \right] \geq \E \left[ \Prob^\eta_{B_{2R}} \left[ \text{\textbf{A}}_4(r,4R) \cap G^{ext}_{\overline{\delta}}(4R) \right]^2 \right],
\end{multline*}
and the following pointwise equality:
\[
\Prob^\eta \left[ \text{\textbf{A}}_4(r,R) \cap G^{ext}_\delta(R) \right] =  \Prob^\eta_{B_{2R}} \left[ \text{\textbf{A}}_4(r,R) \cap G^{ext}_\delta(R) \right].
\]
\end{proof}
\begin{rem}\label{r.poly_et_delta_bar_bis}
Note that Remark~\ref{r.poly_et_delta_bar} and Jensen's inequality imply the following: There exists $c' > 0$ such that, if $\delta \in (0,1/1000)$ is sufficiently small, then for all $R  \in [\delta^{-2},+\infty)$ and all $r \in [R/4^3,R]$, we have
\[
\widetilde{g}^{ext}_{4,\delta}(r,R) \geq c' \, .
\]
Similarly, if $\delta \in (0,1/1000)$ is sufficiently small, then for all $r \in [\delta^{-2},+\infty)$ and all $R \in [r,4^3r]$, we have
\[
\widetilde{g}_{4,\delta}^{int}(r,R) \geq c' \, .
\]
We can (and do) assume that the quantity $\overline{\delta}$ of Lemma~\ref{l.extension_bis} is sufficiently small so that the above holds with $\delta=\overline{\delta}$.
\end{rem}
We now state and prove an analogue of Lemma~\ref{l.looksgood}. We first need the two following notations:
\[
\widetilde{f}_4^{ext}(r,R) = \sqrt{\E \left[ \Prob^\eta \left[ \widehat{\text{\textbf{A}}}^{ext}_4(r,R) \right]^2 \right]} \, ; \, \widetilde{f}_4^{int}(r,R) = \sqrt{\E \left[ \Prob^\eta \left[ \widehat{\text{\textbf{A}}}^{int}_4(r,R) \right]^2 \right]} \, .
\]
\begin{lem}\label{l.looksgood_bis}
There exist $C_1 \in [1,+\infty)$ and $\overline{r} \in [\overline{\delta}^{-2},+\infty)$ such that, for every $r \in [\overline{r},+\infty)$ and $R \in [16r,+\infty)$,
\begin{equation}\label{e.looksgoodplus_bis}
\widetilde{g}^{ext}_{4,\overline{\delta}}(r,R) \geq \widetilde{f}^{ext}_4(r,R)/C_1 \, ,
\end{equation}
and
\begin{equation}\label{e.looksgoodmoins_bis}
\widetilde{g}^{int}_{4,\overline{\delta}}(r,R) \geq \widetilde{f}^{int}_4(r,R)/C_1 \, .
\end{equation}
\end{lem}
We have the following corollary (which is a direct consequence of Lemma~\ref{l.looksgood_bis} and Remark~\ref{r.poly_et_delta_bar_bis}):
\begin{cor}\label{c.looksgood_bis}
There exists a constant $C_2 \in [1,+\infty)$ such that, for every $r \in [\overline{r},+\infty)$ and every $R \in [r,+\infty)$,
\[
\widetilde{g}^{ext}_{4,\overline{\delta}}(r,R) \leq \widetilde{\alpha}_4(r,R) \leq \widetilde{f}^{ext}_4(r,R) \leq C_2 \, \widetilde{g}^{ext}_{4,\overline{\delta}}(r,R) \, ,
\]
and
\[
\widetilde{g}^{int}_{4,\overline{\delta}}(r,R) \leq \widetilde{\alpha}_4(r,R) \leq \widetilde{f}^{int}_4(r,R) \leq C_2 \, \widetilde{g}^{int}_{4,\overline{\delta}}(r,R) \, .
\]
\end{cor}

\begin{proof}[Proof of Lemma~\ref{l.looksgood_bis}]
Let us prove~\eqref{e.looksgoodplus_bis} (the proof of~\eqref{e.looksgoodmoins_bis} is essentially the same). As noted in the proof of Lemma~\ref{l.looksgood}, we have
\[
\widehat{\text{\textbf{A}}}_4^{ext}(r,R) \subseteq \big( \text{\textbf{A}}_4(r,R/4) \cap G_{\delta}^{ext}(R/4) \big) \cup \big( \widehat{\text{\textbf{A}}}_4^{ext}(r,R/16) \setminus G^{ext}_\delta(R/4) \big) \, .
\]
This implies that $\widetilde{f}_4^{ext}(r,R)^2$ is smaller than or equal to
\begin{align*}
& \E \left[ \left( \Prob^\eta \left[ \text{\textbf{A}}_4(r,R/4) \cap G^{ext}_\delta(R/4) \right] + \Prob^\eta \left[ \widehat{\text{\textbf{A}}}_4^{ext}(r,R/16) \setminus G^{ext}_\delta(R/4) \right] \right)^2 \right]\\
& = \widetilde{g}^{ext}_{4,\delta}(r,R/4)^2 + 2 \E \left[ \Prob^\eta \left[ \text{\textbf{A}}_4(r,R/4) \cap G^{ext}_\delta(R/4) \right] \cdot \Prob^\eta \left[  \widehat{\text{\textbf{A}}}_4^{ext}(r,R/16) \setminus G^{ext}_\delta(R/4) \right] \right]\\
& + \E \left[ \Prob^\eta \left[ \widehat{\text{\textbf{A}}}_4^{ext}(r,R/16) \setminus G^{ext}_\delta(R/4) \right]^2 \right]\\
& \leq \widetilde{g}^{ext}_{4,\delta}(r,R/4)^2 + 3 \E \left[ \Prob^\eta \left[ \widehat{\text{\textbf{A}}}^{ext}_4(r,R/16) \right] \cdot \Prob^\eta \left[ \widehat{\text{\textbf{A}}}_4^{ext}(r,R/16) \setminus G^{ext}_\delta(R/4) \right] \right]\\
& = \widetilde{g}^{ext}_{4,\delta}(r,R/4)^2 + 3 \widetilde{f}^{ext}_4(r,R/16)^2 \cdot \Pro \left[ \neg G^{ext}_\delta(R/4) \right]  \, ,
\end{align*}
by spatial independence. This inequality is the analogue of~\eqref{e.key_in_looksgood} in the proof of Lemma~\ref{l.looksgood}. Now, the proof is exactly the same as the one of Lemma~\ref{l.looksgood}.
\end{proof}

We are now in shape to prove Proposition~\ref{p.quasi_tilde}. We first prove it for $r_1$ sufficiently large.
\begin{proof}[Proof of the left-hand-inequality in the case $r_1 \geq \overline{r}$] Thanks to Corollary~\ref{c.looksgood_bis}, the proof is the same as in Subsection~\ref{ss.quasi_even}.
\end{proof}

\begin{proof}[Proof of the right-hand-inequality in the case $r_1 \geq 16\overline{r}$]
If we do not have both $r_1 \leq r_2/6$ and $r_2 \leq r_3/6$ then the proof is exactly the same as in~Subsection~\ref{ss.quasi_even}, so let us assume that $r_1 \leq r_2/6$ and $r_2 \leq r_3/6$. By Corollary~\ref{c.looksgood_bis}, we have
\[
\widetilde{\alpha}_4(r_1,\frac{r_2}{3}) \, \widetilde{\alpha}_4(3r_2,r_3) \leq \grandO{1} \widetilde{g}_{4,\overline{\delta}}^{ext}(r_1,\frac{r_2}{3}) \, \widetilde{g}_{4,\overline{\delta}}^{int}(3r_2,r_3) \, .
\]
By~\eqref{e.key_end_1} and~\eqref{e.key_end_2}, $\left( \widetilde{g}_{4,\overline{\delta}}^{ext}(r_1,\frac{r_2}{3}) \, \widetilde{g}_{4,\overline{\delta}}^{int}(3r_2,r_3) \right)^2$ equals
\begin{align*}
& \E \Big[ \Prob^\eta \Big[ \arm_4(r_1,r_2/3) \cap \widetilde{\gi}^{ext}_{\overline{\delta}}(r_2/3) \cap \dense_{\overline{\delta}}(r_2/3) \cap \qbc_{\overline{\delta}}(r_2/3)\\
& \hspace{2em} \cap \left\lbrace \Pro \left[ \qbc^{ext}(r_2/3) \cond \eta \cap B_{2r_2/3} \right] \geq 3/4 \right\rbrace \Big]^2 \Big]\\
& \hspace{1em} \times \E \Big[ \Prob^\eta \Big[ \arm_4(3r_2,r_3) \cap \widetilde{\gi}^{int}_{\overline{\delta}}(3r_2) \cap \dense_{\overline{\delta}}(3r_2) \cap \qbc_{\overline{\delta}}(3r_2)\\
& \hspace{3em} \cap \left\lbrace \Pro \left[ \qbc^{int}(3r_2) \cond \eta \setminus B_{3r_2/2} \right] \geq 3/4 \right\rbrace \Big]^2 \Big] \, .
\end{align*}
Write $X_1=\Prob^\eta \left[ \arm_4(r_1,r_2/3) \cap \widetilde{\gi}^{ext}_{\overline{\delta}}(r_2/3) \right]$, $A_1=\dense_{\overline{\delta}}(r_2/3) \cap \qbc_{\overline{\delta}}(r_2/3)$, $B_1=\qbc^{ext}(r_2/3)$, $X_2=\Prob^\eta \left[ \arm_4(3r_2,r_3) \cap \widetilde{\gi}^{int}_{\overline{\delta}}(3r_2) \right]$, $A_2=\dense_{\overline{\delta}}(3r_2) \cap \qbc_{\overline{\delta}}(3r_2)$ and $B_2=\qbc^{int}(3r_2)$. Then, the above equals
\begin{align*}
\E \left[ X_1^2 \, \un_{A_1} \, \un_{\left\lbrace \Pro \left[ B_1 \, |\, \eta \cap B_{2r_2/3} \right] \geq 3/4 \right\rbrace} \right] \E \left[ X_2^2 \, \un_{A_2} \,  \un_{\left\lbrace \Pro \left[ B_2 \, | \, \eta \setminus B_{3r_2/2} \right] \geq 3/4 \right\rbrace} \right] \, .
\end{align*}
Let us use the following lemma whose proof is the same as Lemma~\ref{l.key_in_qm}.
\begin{lem}
Let $\mathcal{F}$ and $\mathcal{G}$ be two sub-$\sigma$-algebras, let $A_1 \in \mathcal{F}$, $A_2 \in \mathcal{G}$, let $X_1$ be an $\mathcal{F}$-measurable random variable, $X_2$ a $\mathcal{G}$-measurable random variable, and let $B_1$ and $B_2$ be two events such that $\sigma(B_1,\mathcal{F})$ is independent of $\mathcal{G}$ and $\sigma(B_2,\mathcal{G})$ is independent of $\mathcal{F}$. Then,
\[
\E \left[ (X_1 \, X_2)^2 \, \un_{A_1 \cap B_1 \cap A_2 \cap B_2} \right] \geq \frac{1}{2} \E \left[ X_1^2 \un_{A_1} \, \un_{\left\lbrace \Pro \left[ B_1 \, | \, \mathcal{F} \right] \geq 3/4 \right\rbrace} \right] \cdot \E \left[ X_2^2 \un_{A_2}  \, \un_{\left\lbrace \Pro \left[ B_2 \, | \, \mathcal{G} \right] \geq 3/4 \right\rbrace} \right] \, .
\]
\end{lem}
If we apply this lemma to $\mathcal{F}=\sigma(\eta \cap B_{2r_2/3})$ and $\mathcal{G}=\sigma(\eta \setminus B_{3r_2/2})$, we obtain that $\widetilde{\alpha}_4(r_1,\frac{r_2}{3}) \, \widetilde{\alpha}_4(3r_2,r_3)$ is less than or equal to
\begin{multline*}
\grandO{1} \E \Big[ \Prob^\eta \left[ \arm_4(r_1,r_2/3) \cap \widetilde{\gi}^{ext}_{\overline{\delta}}(r_2/3) \right]^2 \, \un_{ \dense_{\overline{\delta}}(r_2/3) \cap \qbc_{\overline{\delta}}(r_2/3) \cap \qbc^{ext}(r_2/3)}\\
\times \Prob^\eta \left[ \arm_4(3r_2,r_3) \cap \widetilde{\gi}^{int}_{\overline{\delta}}(3r_2) \right] \, \un_{\dense_{\overline{\delta}}(3r_2) \cap \qbc_{\overline{\delta}}(3r_2) \cap \qbc^{int}(3r_2)} \Big] \, ,
\end{multline*}
that equals
\begin{multline*}
\grandO{1} \E \Big[ \Prob^\eta \Big[ \arm_4(r_1,r_2/3) \cap \widetilde{\gi}^{ext}_{\overline{\delta}}(r_2/3) \cap \dense_{\overline{\delta}}(r_2/3) \cap \qbc_{\overline{\delta}}(r_2/3) \cap \qbc^{ext}(r_2/3)\\
\cap \arm_4(3r_2,r_3) \cap \widetilde{\gi}^{int}_{\overline{\delta}}(3r_2) \cap \dense_{\overline{\delta}}(3r_2) \cap \qbc_{\overline{\delta}}(3r_2) \cap \qbc^{int}(3r_2) \Big]^2 \Big] \, .
\end{multline*}
Now, by ``gluing'' arguments, the above is at most $\grandO{1} \widetilde{\alpha}_4(r_1,r_3)^2$ and we are done.
\end{proof}

We have obtained the quasi-multiplicativity property for $r_1 \geq 16\overline{r}$, so (as in Subsection~\ref{ss.quasi_even}) it only remains to prove the following lemma, which is the analogue of Lemma~\ref{l.extensionto1}.
\begin{lem}\label{l.extensionto1_bis}
For every $\overline{\overline{r}}$ sufficiently large, there exists a constant $C_3=C_3(\overline{\overline{r}}) < +\infty$ such that, for every $r \in [1,\overline{\overline{r}}]$ and every $R \in [\overline{\overline{r}},+\infty)$, we have
\[
\widetilde{\alpha}_4(\overline{\overline{r}},R) \leq C_3 \, \widetilde{\alpha}_4(r,R) \, .
\]
\end{lem}
\begin{proof}[Sketch of proof of Lemma~\ref{l.extensionto1_bis}]
First note that, by making observations similar to the end of the proof of Lemma~\ref{l.extension} and by following the proof of Lemma~\ref{l.extension_bis}, we obtain the following result: Let $\widetilde{\arm}^{int}_4(r,R)$ be the event defined at the end of the proof of Lemma~\ref{l.extension}. Then, if $\overline{\delta}$ is sufficiently small and if $\delta$, $r$ and $R$ are as in Item~$2$ of Lemma~\ref{l.extension_bis}, then the following holds: there exists $a=a(\delta) > 0$ and $c > 0$ such that, for every event $F_r$ measurable with respect to $\omega \cap B_{r/2}$ satisfying $\Pro \left[ F_r \right] \geq 1-c$,
\[
\E \left[ \Prob^\eta \left[ \widetilde{\arm}^{int}_4(r/4,R) \cap G^{int}_{\overline{\delta}}(r/4) \cap F_r \right]^2 \right] \geq a \, \widetilde{g}_{4,\delta}^{int}(r,R)^2 \, .
\]
Let $\dense^N(r)$ be the event defined in the proof of Lemma~\ref{l.piv_one_point}. Corollary~\ref{c.looksgood_bis} and the above inequality imply that there exists an absolute constant $c \in (0,1)$ such that, for every $r$ sufficiently large, there exists $N=N(r)$ satisfying
\[
\forall R \geq 4r, \,  \E \left[ \Prob^\eta \left[ \widetilde{\arm}^{int}_4(r,R) \cap \dense^N(r) \right]^2 \right] \geq c \, \widetilde{\alpha}_4(r,R)^2 \, .
\] 
Now, if we follow the proof of Lemma~\ref{l.piv_one_point}, we obtain that there exists a constant $c'=c'(r,N)$ such that
\[
\widetilde{\alpha}_4(R)^2 \geq c' \, \E \left[ \Prob^\eta \left[ \widetilde{\arm}^{int}_4(r,R) \cap \dense^N(r) \right]^2 \right] \, ,
\]
which ends the proof.
\end{proof}
This also ends the proof of Proposition~\ref{p.quasi_tilde}.
\end{proof}

We also have the following analogues of Propositions~\ref{p.fandalpha} and~\ref{p.techniquegeneral} (with the same proofs):
\begin{prop}\label{p.fandalpha_bis}
Let $j \in \N^*$, let $1 \leq r \leq R$, and write
\begin{eqnarray}
\widetilde{f}_j(r,R) & = & \sqrt{\E \left[ \Prob^\eta \left[ \widehat{\arm}_j(r,R) \right]^2 \right]} \ .
\end{eqnarray}
There exists a constant $C = C(j) < +\infty$ such that
\[
\widetilde{\alpha}_j(r,R) \leq \widetilde{f}_j(r,R) \leq C \, \widetilde{\alpha}_j(r,R) \, .
\]
\end{prop}
\begin{prop}\label{p.techniquegeneral_bis}
Let $j \in \N^*$. For every $h \in (0,1)$, there exists a constant $\epsilon = \epsilon(j,h) \in (0,1)$ such that, for every $1 \leq r \leq R$ and for every event $G$ which is measurable with respect to $\omega \setminus A(2r,R/2)$ and that satisfies $\Pro \left[ G \right] \geq 1-\epsilon$, we have
\[
\E \left[ \Prob^\eta \left[ \arm_j(r,R) \cap G \right]^2 \right] \geq (1-h) \, \widetilde{\alpha}_j(r,R)^2 \, .
\]
\end{prop}

\subsection{Pivotal events}

Let us first prove the following lemma which is the analogue of Lemma~\ref{l.piv_one_point}.
\begin{lem}\label{l.piv_one_point_bis}
Let $R \geq 1$ and let $S$ be a $2 \times 2$ square included in $[-2R,2R] \times [-R,R]$ and at distance at least $R/3$ from the sides of this rectangle. Then,
\[
\E \left[ \Prob^\eta \left[ \Piv^q_S(\cross(2R,R) \right]^2 \un_{|\eta \cap S| = 1} \right] \geq \Omega(1) \widetilde{\alpha}_4(R)^2 \, .
\]
\end{lem}
\begin{proof}
With exactly the same proof as in Subsection~\ref{ss.pivotals1} (but by using Proposition~\ref{p.quasi_tilde} instead of Proposition~\ref{p.quasi}, Proposition~\ref{p.fandalpha_bis} instead of Proposition~\ref{p.fandalpha}, and Proposition~\ref{p.techniquegeneral_bis} instead of Proposition~\ref{p.techniquegeneral}), we have the following: There exists $r_0\geq 1$ and $\epsilon \in (0,1)$ such that, for every $r \geq r_0$ and for every event $G$ measurable with respect to $\omega \setminus A(2r,R/2)$ that satisfies $\Pro \left[ G \right] \geq 1-\epsilon$,
\[
\E \left[ \Prob^\eta \left[ \widetilde{\arm}^\square_4(B_r,R) \cap G \right]^2 \right] \geq \epsilon \, \widetilde{\alpha}_4(r,R)^2 \, ,
\]
where $\widetilde{\arm}^\square_4(B_r,R)$ is the event defined in the beginning of the proof of Lemma~\ref{l.piv_one_point}. Now, the proof is exactly the same as the one of Lemma~\ref{l.piv_one_point}.
\end{proof}

The following is a direct consequence of Lemma~\ref{l.piv_one_point_bis} and of results from~\cite{ahlberg2015quenched}:
\begin{cor}
There exists $\epsilon > 0$ such that, for every $R \in (0,+\infty)$,
\[
\widetilde{\alpha}_4(R) \leq \frac{1}{\epsilon} R^{1+\epsilon} \, .
\]
\end{cor}
\begin{proof}
In the proof of the first part of Proposition~\ref{p.alpha4}, we have explained how to prove the analogous result for $\alpha^{an}_4(R)$. In this proof (written in Subsection~\ref{ss.pivotals1}), we have used the following result from~\cite{ahlberg2015quenched}:
\[
\sum_{S} \E \left[ \Prob^\eta \left[ \Piv^q_S(\cross(2R,R) \right]^2 \un_{|\eta \cap S| = 1} \right] \leq \grandO{1} R^{-\Omega(1)} \, ,
\]
where the sum is over all the squares of the grid $\Z^2$ included in $[-2R,2R] \times [-R,R]$ and at distance at least $R/3$ from the sides of this rectangle. We have then used Jensen's inequality and Lemma~\ref{l.piv_one_point}. If we do not use Jensen's inequality and if we use Lemma~\ref{l.piv_one_point_bis} instead of Lemma~\ref{l.piv_one_point}, we obtain the desired result.
\end{proof}

As in Subsection~\ref{ss.alpha4}, we prove an estimate about the $4$ and $5$-arm events. Note that we know that $\alpha^{an}_5(r,R) \asymp \left( \frac{r}{R} \right)^2$ but this does not imply that $\widetilde{\alpha}_5(r,R) \asymp (r/R)^2$ (we will prove this last estimate in~\cite{quant_voro}). We have the following:
\begin{prop}\label{p.alpha4_bis}
There exists an absolute constant $\epsilon > 0$ such that, for every $1 \leq r \leq R$,
\[
\widetilde{\alpha}_5(r,R) \leq \frac{1}{\epsilon} \, \widetilde{\alpha}_4(r,R) \, \left( \frac{r}{R} \right)^{\epsilon} \, .
\] 
\end{prop}
\begin{proof}
First, fix some $N \in \N^*$ sufficiently large so that
\[
\widetilde{\alpha}_5(\rho,\rho')^2 \geq \frac{1}{N} \, \left( \frac{\rho}{\rho'} \right)^{N-1} \, .
\]
Let $M \geq 100$, $\rho \geq M$ and
\[
\gp(\rho,M) = \bigcap_{k=0}^{\lfloor \log_2 ( M ) \rfloor - 1} \dense_{1/100}(A(2^k\rho,5\cdot 2^k \rho)) \cap \qbc^N_{1/100} (A(2^k \rho,5\cdot 2^k \rho)) \, .
\]
Note that
\[
\Pro \left[ \gp(\rho,M) \right] \geq 1 - \grandO{1} \rho^{-N} \geq 1 - \grandO{1} M^{-N} \, .
\]
As in the proof of the second part of Proposition~\ref{p.alpha4} (written in Subsection~\ref{ss.alpha4}), there exists a constant $c \in (0,1)$ such that, if $\eta \in \gp(\rho,M)$, then
\[
\Prob^\eta \left[ \arm_5(\rho,M\rho) \right] \leq (1-c)^{ \log \left( M \right)} \, \Ex^\eta \left[ Y^3 \un_{Y \geq 4} \right] \, ,
\]
where $Y$ is the number of interfaces from $\partial B_{\rho}$ to $\partial B_{M \rho}$. By Reimer's inequality~\cite{reimer2000proof}, we have $\Prob^\eta \left[ Y \geq j \right]  \leq \Prob^\eta \left[ \arm_1(\rho,M\rho) \right]^j$. If $\eta \in \gp(\rho,M)$, then $\Prob^\eta \left[ \arm_1(\rho,M\rho) \right] \leq (1-a)^{\log \left( M \right) }$ for some $a \in (0,1)$. Moreover, if $\eta \in  \gp(\rho,M)$, then $\Prob^\eta \left[ Y \geq 4 \right] = \Prob^\eta \left[ \arm_4(\rho,M \rho) \right] \geq M^{-b}/b$ for some $b \in (0,+\infty)$. Hence,
\[
\Ex^\eta \left[ Y^3 \un_{Y \geq 4} \right] \leq d \, \Prob^\eta \left[ \arm_4(\rho,M\rho) \right] \, ,
\]
for some $d \in (0,+\infty)$. As a result,
\[
\widetilde{\alpha}_5(\rho,M\rho)^2 \leq \grandO{1} \, (1-c)^{2\log \left( M \right) } \, \widetilde{\alpha}_4(\rho,M\rho)^2 + \grandO{1}M^{-N} \, .
\]
Now, the proof is essentially the same as the proof of the second part of Proposition~\ref{p.alpha4} (except that we use Proposition~\ref{p.quasi_tilde} instead of Proposition~\ref{p.quasi}).
\end{proof}

In~\cite{quant_voro}, we will need results similar to the results of Subsections~\ref{ss.pivotals1} and~\ref{ss.pivotals2} but for the quantity $\E \left[ \Prob^\eta \left[ \Piv_S(A) \right]^2 \right]$ instead of $\Pro \left[ \Piv_S(A) \right]$ (where $A$ is a crossing or an arm event). In particular, we will need the following lemma whose proof is very closed to the proof of Lemma~\ref{l.piv_bulk} and of the different lemmas of Subsection~\ref{ss.pivotals2}.
\begin{lem}\label{l.piv_hat_bis}
Define $\Piv_D^E(A)$ as in the beginning of Subsection~\ref{ss.pivotals2}. Let $\rho,r,R \in [1,+\infty]$ such that $\rho \leq r/10$ and $r \leq R/2$, let $y$ be a point of the plane, and let $S=B_\rho(y)$ be the $2\rho \times 2\rho$ square centered at $y$. We have the following:
\bi 
\item[i)] Let $\rho_1 \in [\rho,+\infty)$ and $\rho_2 \in [\rho_1,+\infty)$ and assume that $B_{\rho_2}(y) \subseteq A(r,R)$. Then,
\[
\E \left[ \Prob^\eta \left[ \Piv_S^{A(y;\rho_1,\rho_2)}\left( \arm_j(r,R) \right) \right]^2 \right] \leq \grandO{1} \widetilde{\alpha}_4(\rho_1,\rho_2)^2 \, .
\]
\item[ii)] Let $\rho_1\in [1,R/4]$ and assume that $S \subseteq A(\rho_1,4\rho_1)$. Then,
\[
\E \left[ \Prob^\eta \left[ \Piv_S^{A(r,\rho_1) \cup A(4\rho_1,R)}\left( \arm_j(r,R) \right) \right]^2 \right] \leq \grandO{1} \widetilde{\alpha}_j(r,R)^2 \, .
\]
\ei
\end{lem}
\begin{proof}
We prove only i) since the proof of ii) is essentially the same (actually, in the case $j$ odd, the proof of i) is slightly more technical). We first prove i) in the case $j$ even and then in the case $j$ odd. We use the following notation, where $0<\rho'\leq \rho''$:
\[
\dense(\rho',\rho'')=\dense_{1/100} (A(y;\rho',2\rho')) \cap \dense_{1/100} ( A(y;\rho'',2\rho'') ) \, .
\]
Since $j$ is even, then for any $k \in \{ 0,\cdots,\lfloor \log_2(\rho_2/(4\rho_1))\rfloor =: k_0 \}$, $\Piv_S^{A(y;\rho_1,\rho_2)}(\arm_j(r,R))$ is included in
\[
 \arm_4(y;2^{k+1}\rho_1,\rho_2/2) \cup \left( \Piv_S^{A(y;\rho_1,\rho_2)}(\arm_j(r,R)) \setminus \dense(2^k\rho_1,\rho_2/2) \right) \, ,
\]
where $\arm_4(y;\cdot,\cdot)$ is the $4$-arm event translated by $y$. As a result, $\Piv_S^{A(y;\rho_1,\rho_2)}(\arm_j(r,R))$ is also included in
\begin{align*}
& \arm_4(y;2\rho_1,\rho_2/2) \bigcup \left( \bigcup_{k=0}^{k_0} \arm_4(y;2^{k+2}\rho_1,\rho_2/2) \setminus \dense(2^k\rho_1,\rho_2/2) \right) \bigcup \neg \dense(2^{k_0+1}\rho_1,\rho_2/2) \\
& \subseteq \widehat{\arm}_4(y;2\rho_1,\rho_2/2) \bigcup \left( \bigcup_{k=0}^{k_0} \widehat{\arm}_4(y;2^{k+2}\rho_1,\rho_2/2) \setminus \dense(2^k\rho_1,\rho_2/2) \right) \bigcup \neg \dense(2^{k_0+1}\rho_1,\rho_2/2) \, ,
\end{align*}
where
\[
\widehat{\arm}_4(y;\rho',\rho'') = \left\lbrace  \Pro \left[ \arm_4(y;\rho',\rho'') \cond \omega \cap A(y;\rho',\rho'') \right] > 0 \right\rbrace  \, .
\]
By $\sigma$-additivity, $\Prob^\eta \left[ \Piv_S^{A(y;\rho_1,\rho_2)}(\arm_j(r,R)) \right]$ is less than or equal to
\[
\Prob^\eta \left[ \widehat{\arm}_4(y;2\rho_1,\rho_2/2) \right] + \sum_{k=0}^{k_0} \Prob^\eta \left[ \widehat{\arm}_4(y;2^{k+2}\rho_1,\rho_2/2) \right] \un_{\neg \dense(2^k\rho_1,\rho_2/2)} + \un_{\neg \dense(2^{k_0+1}\rho_1,\rho_2/2)} \, .
\]
Now, note that, for every $k \in \{0,\cdots,k_0\}$,
\bi 
\item $\Prob^\eta \left[ \widehat{\arm}_4(y;2^{k+2}\rho_1,\rho_2/2) \right]$ is independent of $\un_{\neg \dense(2^k\rho_1,\rho_2/2)}$;
\item $\E \left[ \Prob^\eta \left[ \widehat{\arm}_4(y;2^{k+2}\rho_1,\rho_2/2) \right]^2 \right] \leq C (2^k)^C \widetilde{\alpha}_4(y;\rho_1,\rho_2)$ for some $C <+\infty$ by~Propositions~\ref{p.fandalpha_bis} and~\ref{p.quasi_tilde};
\item $\Pro \left[ \dense(2^k\rho_1,\rho_2/2) \right] \leq \grandO{1}\exp(-\Omega(1) (2^k \rho_1)^2)$.
\ei
Note also that
\bi 
\item $\E \left[ \Prob^\eta \left[ \widehat{\arm}_4(y;2\rho_1,\rho_2/2) \right]^2 \right] \leq C \widetilde{\alpha}_4(\rho_1,\rho_2)$ and:
\item $\Pro \left[ \dense(2^{k_0+1}\rho_1,\rho_2/2) \right] \leq \grandO{1}\exp(-\Omega(1) (2^{k_0} \rho_1)^2)$.
\ei
Let us end the proof by showing that
\begin{multline*}
\E \Big[ \Big( \Prob^\eta \left[ \widehat{\arm}_4(y;2\rho_1,\rho_2/2) \right] +\\
\sum_{k=0}^{k_0} \Prob^\eta \left[ \widehat{\arm}_4(y;2^{k+2}\rho_1,\rho_2/2) \right] \un_{\neg \dense(2^k\rho_1,\rho_2/2)} + \un_{\neg \dense(2^{k_0+1}\rho_1,\rho_2/2)} \Big)^2 \Big]
\end{multline*}
is at most $\grandO{1} \widetilde{\alpha}_4(y;\rho_1,\rho_2)^2$. If we expand the above square, apply the Cauchy-Schwarz inequality to each of the $2^{k_0+3}$ terms and use the five items above, then we obtain the desired result since
\[
\sum_{k,l} \sqrt{(2^k)^{C} \exp(-\Omega(1) (2^k \rho_1)^2) (2^l)^C \exp(-\Omega(1) (2^l \rho_1)^2)} < +\infty \, .
\]
Let us now assume that $j$ is odd. In this case, as explained in Appendix~\ref{a.joddpivarm}, the pivotal event can be described for instance by using the $4$ and the $5$-arm events (and not only the $4$-arm event). As in Appendix~\ref{a.joddpivarm}, the result is obtained by using a comparison estimate between the $4$ and $5$-arm events. The only difference is that in the present case we use Proposition~\ref{p.alpha4_bis} instead of the analogous comparison estimate between $\alpha_4^{an}(\cdot,\cdot)$ and $\alpha^{an}_5(\cdot,\cdot)$.
\end{proof}

\bibliographystyle{alpha}
\bibliography{ref_perco}

\ni
{\bf Hugo Vanneuville} \\
Univ. Lyon 1\\
UMR5208, Institut Camille Jordan, 69100 Villeurbanne, France\\
vanneuville@math.univ-lyon1.fr\\
\url{http://math.univ-lyon1.fr/~vanneuville/}\\
Supported by the ERC grant Liko No 676999\\

\end{document}